\documentclass[11pt]{article}

\usepackage{amsbsy}
\usepackage{amssymb}
\usepackage{amsthm,amsfonts}
\usepackage{amsmath}
\usepackage{fullpage,hyperref,mathrsfs,txfonts,epsfig,graphicx,graphics,verbatim}

\begin{document}

\newcommand{\bwr}{\boldsymbol{\wr}}

% ENVIRONMENTS & THEOREMS

\newenvironment{nmath}{\begin{center}\begin{math}}{\end{math}\end{center}}

\newtheorem{thm}{Theorem}[section]
\newtheorem{lem}[thm]{Lemma}
\newtheorem{remark}[thm]{Remark}
\newtheorem{prop}[thm]{Proposition}
\newtheorem{cor}[thm]{Corollary}
\newtheorem{conj}[thm]{Conjecture}
\newtheorem{dfn}[thm]{Definition}
\newtheorem{prob}[thm]{Problem}
\newtheorem{ques}[thm]{Question}

% COMMANDS
\newcommand{\f}{\varphi}
\newcommand{\A}{\mathcal{A}}
\newcommand{\B}{\mathcal{B}}
\newcommand{\K}{\mathcal{K}}
\newcommand{\TSP}{\mathrm{TSP}}
\newcommand{\F}{\mathbb{F}}
\newcommand{\FF}{\mathcal{F}}
\newcommand{\Lip}{\mathrm{Lip}}
\newcommand{\1}{\mathbf{1}}
\newcommand{\s}{\sigma}
\renewcommand{\P}{\mathcal{P}}
\renewcommand{\O}{\Omega}
\renewcommand{\S}{\Sigma}
\renewcommand{\Pr}{\mathbb{P}}
\renewcommand{\approx}{\asymp}
\newcommand{\T}{\mathrm{T}}
\newcommand{\co}{\mathrm{co}}
\newcommand{\Isom}{\mathrm{Isom}}
\newcommand{\edge}{\mathrm{edge}}
\renewcommand{\span}{\mathrm{span}}
\newcommand{\0}{\mathbf{0}}
\newcommand{\bounded}{\mathrm{bounded}}
\newcommand{\e}{\varepsilon}
\newcommand{\im}{\mathrm{i}}
\newcommand{\restrict}{\upharpoonright}
\newcommand{\supp}{{\mathrm{\bf supp}}}
\renewcommand{\l}{\lambda}
\newcommand{\U}{\mathcal{U}}
\newcommand{\M}{\mathscr{M}}
\newcommand{\calH}{\mathcal{H}}
\newcommand{\G}{\Gamma}
\newcommand{\g}{\gamma}
\renewcommand{\L}{\mathscr{L}}
\newcommand{\hcf}{\mathrm{hcf}}
\renewcommand{\a}{\alpha}
\newcommand{\N}{\mathbb{N}}
\newcommand{\R}{\mathbb{R}}
\newcommand{\Z}{\mathbb{Z}}
\newcommand{\C}{\mathbb{C}}

\newcommand{\E}{\mathbb{E}}
\newcommand{\alp}{\alpha^*}

\newcommand{\bb}[1]{\mathbb{#1}}
\renewcommand{\rm}[1]{\mathrm{#1}}
\renewcommand{\cal}[1]{\mathcal{#1}}

\newcommand{\fin}{\nolinebreak\hspace{\stretch{1}}$\lhd$}

\title{$L_p$ compression, traveling salesmen, and stable walks\\{\small Dedicated with admiration to the memory of Oded Schramm}}
\author{Assaf Naor\footnote{Research supported in part by NSF grants DMS-0528387, CCF-0635078 and CCF-0832795, BSF grant 2006009, and the Packard Foundation.
}\\Courant
Institute\\{\tt naor@cims.nyu.edu} \and Yuval Peres \\Microsoft
Research\\{\tt peres@microsoft.com}}
\date{}

\maketitle

\begin{abstract}
We show that if $H$ is a group of polynomial growth whose growth
rate is at least quadratic then the $L_p$ compression of the wreath
product $\Z\bwr H$ equals
$\max\left\{\frac{1}{p},\frac{1}{2}\right\}$. We also show that the
$L_p$ compression of $\Z\bwr \Z$ equals
$\max\left\{\frac{p}{2p-1},\frac23\right\}$ and the $L_p$
compression of $(\Z\bwr\Z)_0$ (the zero section of $\Z\bwr \Z$,
equipped with the metric induced from $\Z\bwr \Z$) equals
$\max\left\{\frac{p+1}{2p},\frac34\right\}$. The fact that the
Hilbert compression exponent of $\Z\bwr\Z$ equals $\frac23$ while
the Hilbert compression exponent of $(\Z\bwr\Z)_0$ equals $\frac34$
is used to show that there exists a Lipschitz function
$f:(\Z\bwr\Z)_0\to L_2$ which cannot be extended to a Lipschitz
function defined on all of $\Z\bwr \Z$.
\end{abstract}

%\tableofcontents

\section{Introduction}

Let $G$ be an infinite group which is generated by a finite
symmetric set $S\subseteq G$ and let $d_G$ denote the left-invariant
word metric induced by $S$ (formally we should use the notation
$d_S$, but all of our statements below will be independent of the
generating set). Assume for the moment that the metric space
$(G,d_G)$ does not admit a bi-Lipschitz embedding into Hilbert
space\footnote{This assumption is not very restrictive, and in fact
it is conjectured that if $(G,d_G)$ does admit a bi-Lipschitz
embedding into Hilbert space then $G$ must have an Abelian subgroup
of finite index. We refer to~\cite{dCTV07} for more information on
this conjecture and its proof in some interesting special cases.}.
In such a setting the next natural step is to try to measure the
extent to which the geometry of $(G,d_G)$ is non-Hilbertian. While
one can come up with several useful ways to quantify
non-embeddabililty, the present paper is a contribution to the
theory of compression exponents: a popular and elegant way of
measuring non-bi-Lipschitz embeddability of infinite groups that was
introduced by Guentner and Kaminker in~\cite{GK04}.

The Hilbert compression exponent of $G$, denoted $\alpha^*(G)$, is
defined as the supremum of those $\alpha\ge 0$ for which there exists a
Lipschitz function $f:G\to L_2$ satisfying $\|f(x)-f(y)\|_2\ge
cd_G(x,y)^{\alpha}$ for every $x,y\in G$ and some constant $c>0$
which is independent of $x,y$. More generally, given a target metric
space $(X,d_X)$ the compression exponent of $G$ in $X$, denoted
$\alpha_X^*(G)$, is the supremum over $\alpha\ge 0$ for which there
exists a Lipschitz function $f:G\to X$ satisfying $d_X(f(x),f(y))\ge
cd_G(x,y)^{\alpha}$. When $X=L_p$ for some $p\ge 1$ we shall use the
notation $\alpha_p^*(G)=\alpha_{L_p}^*(G)$ (thus
$\alpha_2^*(G)=\alpha^*(G)$).

When $(X,\|\cdot\|_X)$ is a Banach space one can analogously define
the equivariant compression exponent of $G$ in $X$, denoted
$\alpha_X^\#(G)$, as the supremum over $\alpha\ge 0$ for which there
exists a $G$-equivariant\footnote{A mapping $\psi:G\to X$ is called
$G$-equivariant if there exists an action $\tau$ of $G$ on $X$ by
affine isometries and a vector $v\in X$ such that $\psi(x)=\tau(x)v$
for all $x\in G$. Equivalently there exists an action $\pi$ on $X$ by
linear isometries such that $\psi$ is a $1$-cocycle with respect to
$\pi$ (we denote this by $\psi\in Z^1(G,\pi)$), i.e., for every $x,y\in
G$ we have $\psi(xy)=\pi(x)\psi(y)+\psi(x)$. A key useful point here is that
in this case $\|\psi(x)-\psi(y)\|_X$ is an invariant semi-metric on $G$.}
mapping $\psi:G\to X$ satisfying $\|\psi(x)-\psi(y)\|_X\ge
cd_G(x,y)^{\alpha}$. We write as above
$\alpha_p^\#(G)=\alpha_{L_p}^\#(G)$ and
$\alpha^\#(G)=\alpha_2^\#(G)$. Recall that $G$ is said to have the
Haagerup property if there exists an equivariant function $\psi:G\to
L_2$ such that $\inf\{\|\psi(x)-\psi(y)\|_2:\ d_G(x,y)\ge t\}$ tends to
infinity with $t$. We refer to the book~\cite{CCJJV01} for more
information on the Haagerup property and its applications. Thus the
notion of equivariant compression exponent can be viewed as a
quantitative refinement of the Haagerup property, and  this is
indeed the way that bounds on the equivariant compression exponent
are usually used.

The parameters $\alpha_X^*(G)$ and $\alpha_X^\#(G)$ do not depend on
the choice of symmetric generating set $S$, and are therefore
genuine algebraic invariants of the group $G$. In~\cite{GK04} it was
shown that if $\alpha^\#(G)>\frac12$ then $G$ is amenable. This
result was generalized in~\cite{NP07}, where it was shown that for
$p\ge 1$ if $X$ is a Banach space whose modulus of uniform
smoothness has power type $p$ (i.e. for every two unit vectors
$x,y\in X$ and $\tau>0$ we have $\|x+\tau y\|_X+\|x-\tau y\|_X\le
2+c\tau^p$ for some $c>0$ which does not depend on $x,y,\tau$) and
$\alpha_X^\#(G)>\frac{1}{p}$ then $G$ is amenable. It was also shown
in~\cite{GK04} that if $\alpha^*(G)>\frac12$ then the reduced $C^*$
algebra of $G$ is exact.

%The relation between between $\alpha_p^*(G)$ and $\alpha_p^\#(G)$ is
%unclear.

Despite their intrinsic interest and a considerable amount of effort
by researchers in recent years, the invariants $\alpha_X^*(G),
\alpha_X^\#(G)$ have been computed in only a few cases. It was shown
in~\cite{ADS06} that for any $\alpha\in [0,1]$ there exists a
finitely generated group $G$ with $\alpha^*(G)=\alpha$. In light of
this fact it is quite remarkable that, apart from a few exceptions,
in most of the known cases in which compression exponents have been
computed they turned out to be equal to $1$ or $0$. A classical
theorem of Assouad~\cite{Ass83} implies that groups of polynomial
growth have Hilbert compression exponent $1$. On the other hand,
Gromov's random groups~\cite{Gro03} have Hilbert compression
exponent $0$. Bourgain's classical metrical characterization of
superreflexivity~\cite{Bourgain86} implies that finitely generated
free groups have Hilbert compression exponent 1 (this interpretation
of Bourgain's theorem was first noted in~\cite{GK04}), and more
generally it was shown in~\cite{BS08} that hyperbolic groups have
Hilbert compression $1$ and in~\cite{CN05} that so does any discrete
group acting properly and co-compactly on a finite dimensional
CAT(0) cubical complex. In~\cite{Tess06} it was shown that
co-compact lattices in connected Lie groups, irreducible lattices in
semi-simple Lie groups of rank at least $2$, polycyclic groups and
certain semidirect products with $\Z$ (including wreath
products\footnote{The (restricted) wreath product of $G$ with $H$,
denoted $G\bwr H$, is defined as as the group of all pairs $(f,x)$
where $f:H\to G$ has finite support (i.e. $f(z)= e_G$, the identity
element of $G$, for all but finitely many $z\in H$) and $x\in H$,
equipped with the product $ (f,x)(g,y)\coloneqq \left(z\mapsto
f(z)g(x^{-1}z),xy\right)$. If $G$ is generated by the set
$S\subseteq G$ and $H$ is generated by the set $T\subseteq H$ then
$G\bwr H$ is generated by the set $\{(e_{G^H},t):\ t\in T\}\cup
\{(\delta_s,e_H):\ s\in S\}$, where $\delta_s$ is the function which
takes the value $s$ at $e_H$ and the value $e_G$ on
$H\setminus\{e_H\}$. Unless otherwise stated we will always assume
that $G\bwr H$ is equipped with the word metric associated with this
canonical set of generators (although in most cases our assertions
will be independent of the choice of generators).} of finite groups
with $\Z$ and the Baumslag-Solitar group) all have Hilbert
compression exponent $1$. The first example of a group with Hilbert
compression exponent in $(0,1)$ was found in~\cite{AGS06}, where it
was proved that R. Thompson's group $F$ satisfies
$\alpha^*(F)=\frac12$. Another well-studied case is the wreath
product $\Z\bwr \Z$: in~\cite{Gal08} it was shown that
$\alpha^*(\Z\bwr \Z)\ge \frac13$, and this lower bound was improved
in~\cite{AGS06} and independently in~\cite{SV07} to $\alpha^*(\Z\bwr
\Z)\ge \frac12$. Moreover it was shown in~\cite{AGS06} that
$\alpha^*(\Z\bwr \Z)\le \frac34$ and a combination of the results
of~\cite{ANP07} and~\cite{NP07}, which established sharp upper and
lower bounds on $\alpha^*(\Z\bwr \Z)$, respectively, settles the
case of the Hilbert compression exponent of $\Z\bwr \Z$ by showing
that $\alpha^*(\Z\bwr \Z)=\frac23$ (nevertheless, the $\frac34$
upper bound on $\alpha^*(\Z\bwr \Z)$ from~\cite{AGS06} has a special
meaning which is important for our current work---we will return to
this topic later in this introduction). More generally, it was shown
in~\cite{NP07} that if we define recursively $\Z_1=\Z$ and
$\Z_{(k+1)}=\Z_{(k)}\bwr \Z$ then
$\alpha^*(\Z_{(k)})=\frac{1}{2-2^{1-k}}$. In~\cite{NP07} it was
shown that $\alpha^*\left(C_2\bwr \Z^2\right)=\frac12$, where $C_2$
denotes the cyclic group of order $2$ (the lower bound of $\frac12$
was  proved earlier in~\cite{Tess06}). Finally, it follows
from~\cite{CSV07,NP07}
 that $\alpha^*(C_2\bwr F_n)=\alpha^\#(C_2\bwr
F_n)=\frac12$, where $F_n$ is the free group on $n\ge 2$ generators
(the upper bound of $\frac12$ on $\alpha^*(C_2\bwr F_n)$ is due
to~\cite{NP07} while the lower bound on $\alpha^\#(C_2\bwr F_n)$ is
the key result of~\cite{CSV07}). Many of the above results have (at
least partial) variants for the $L_p$ compression of the groups in
question----we stated here only the case of Hilbert compression  for
the sake of simplicity, and we refer to the relevant papers for more
information.

 The
difficulty in evaluating compression exponents is the main reason
for our interest in this topic, and our purpose here is to devise
new methods to compute them. In doing so we answer questions posed
in~\cite{Tess06,NP07}. One feature of the known methods for
computing compression exponents is that they involve a novel
interplay between group theory and other mathematical disciplines
such as metric geometry, Banach space theory, analysis and
probability. It isn't only the case that the latter disciplines are
applied to group theory---it turns out that the investigation of
compression exponents improved our understanding of issues in
analysis and metric geometry as well (e.g. in~\cite{NP07}
compression exponents were used to make progress on the theory of
non-linear type). In the present paper we apply our new compression
exponent calculations to the Lipschitz extension problem, and relate
them to the Jones Traveling Salesman problem. These applications
will be described in detail presently.

In~\cite{Tess06} it was shown that for all $d\in \N$ we have
$\alpha^*\left(C_2\bwr \Z^d\right)\ge \frac{1}{d}$. A different
embedding yielding this lower bound was obtained in~\cite{NP07},
together with the matching upper bound when $d=2$. Thus, as stated
above, $\alpha^*\left(C_2\bwr \Z^2\right)=\frac12$. In
Section~\ref{section:polywreath} we investigate the value of
$\alpha_p^*(G\bwr H)$ when $G$ is a general group and $H$ is a group
of polynomial growth. The key feature of our result is that we
obtain a lower bound on $\alpha_p^*(G\bwr H)$ which is independent
of the growth rate of $H$. In combination with the upper bounds
obtained in~\cite{NP07} our lower bound implies that for every $p\in
[1,\infty)$ and every group $H$ of polynomial growth whose growth is
at least quadratic we have:
\begin{eqnarray}\label{eq:intro polywreath}\alpha_p^*(\Z\bwr H)=\alpha_p^*(C_2\bwr
H)=\max\left\{\frac{1}{p},\frac12\right\}.\end{eqnarray} As we
explain in Remark~\ref{rem:old Z^2} below, the embedding
from~\cite{NP07} which yielded the identity $\alpha_2^*\left(C_2\bwr
\Z^2\right)=\frac12$ was based on the trivial fact, which is special
to $2$ dimensions, that for every $A\subseteq \Z^2$ of diameter $D$,
the shortest path in $\Z^2$ which covers $A$ has length at most
$O\left(D^2\right)$. It therefore turns out that the previous method
for bounding $\alpha_p^*\left(C_2\bwr \Z^d\right)$ yields tight
bounds only when $p=d=2$ (this is made precise in
Remark~\ref{rem:old Z^2}). Hence in order to prove~\eqref{eq:intro
polywreath} we devise a new embedding which is in the spirit of (but
simpler than) the multi-scale arguments used in the proof of the
Jones Traveling Salesman Theorem~\cite{Jones90} (see
also~\cite{Oki92} and the survey~\cite{Schul07}).

To explain the connection between our proof and the Jones Traveling
Salesman Theorem take two elements $(f,x),(g,y)$ in the ``planar
lamplighter group" $C_2\bwr \Z^2$, i.e., $x,y\in \Z^2$ and
$f,g:\Z^d\to \{0,1\}$ with finite support. The distance between
$(f,x)$ and $(g,y)$ in $C_2\bwr \Z^2$ is, up to a factor of $2$, the shortest path in the
integer grid $\Z^2$ which starts at $x$, visits all the sites $w\in
\Z^2$ at which $f(w)$ and $g(w)$ differ, and terminates at $y$.
Jones~\cite{Jones90} associates to every set $A\subseteq \R^2$ of
diameter $1$ a sequence of numbers, known as the (squares of the)
Jones $\beta$ numbers, whose appropriately weighted sum is (up to
universal factors) the length of the shortest Lipschitz curve
covering $A$, assuming such a curve exists. Focusing on our proof of
the fact that $\alpha_1^*\left(C_2\bwr \Z^2\right)=1$, in our
setting we do something similar: we associate to every $(f,x)\in
C_2\bwr \Z^2$ a sequence of real numbers such that if we wish to
estimate (up to logarithmic terms) the shortest traveling salesman
tour starting at $x$, ending at $y$, and covering the symmetric
difference of the supports of $f$ and $g$, all we have to do is to
compute the $\ell_1$ norm of the difference of the sequences
associated to $(f,x)$ and $(g,y)$. Since the statement
$\alpha_1^*\left(C_2\bwr \Z^2\right)=1$ does not necessarily imply
that $C_2\bwr \Z^2$ admits a bi-Lipschitz embedding into $L_1$, our
result falls short of obtaining a constant-factor approximation of
the length of this tour, which, if possible, would be an interesting
equivariant version of the Jones Traveling Salesman Theorem (note
that if one wishes to estimate the length of the shortest Lipschitz
curve covering the symmetric difference $A\triangle B$ for some
$A,B\subseteq \R^2$ one cannot ``read" this just from the Jones
$\beta$ numbers of $A$ and $B$ without recomputing the Jones $\beta$
numbers of $A\triangle B$). In view of such a potential
strengthening of the Jones Traveling Salesman Theorem, the question
whether $C_2\bwr \Z^2$ admits a bi-Lipschitz embedding into $L_1$
remains an interesting open problem that arises from our work (which
currently only yields a ``compression $1$" version of this
statement).

In Section~\ref{sec:Lp comp} we compute the $L_p$ compression of
$\Z\bwr \Z$, answering a question posed in~\cite{NP07}. Namely we
show that for $p\in [1,\infty)$ we have:
\begin{eqnarray}\label{eq:intro Z wreath Z}\alpha_p^*(\Z\bwr \Z)=
\max\left\{\frac{p}{2p-1},\frac23\right\}. \end{eqnarray} The fact
that $\alpha_p^*(\Z\bwr \Z)$ is at least the right-hand side
of~\eqref{eq:intro Z wreath Z} was proved in~\cite{NP07}, so the key
issue in~\eqref{eq:intro Z wreath Z} is to show that no embedding of
$\Z\bwr\Z$ can have a compression exponent bigger than the
right-hand side of~\eqref{eq:intro Z wreath Z}. We do so via a
non-trivial enhancement of the {\em Markov type} method for bounding
compression exponents that was introduced in~\cite{ANP07}. In order
to explain the new idea used in proving~\eqref{eq:intro Z wreath Z}
we first briefly recall the basic bound from~\cite{ANP07}.

A Markov chain $\{Z_t\}_{t=0}^\infty$ with transition probabilities
$a_{ij}\coloneqq\Pr(Z_{t+1}=j\mid Z_t=i)$ on the state space
$\{1,\ldots,n\}$ is {\em stationary\/} if $\pi_i\coloneqq\Pr(Z_t=i)$
does not depend on $t$ and it is {\em reversible\/} if
$\pi_i\,a_{ij}=\pi_j\,a_{ji}$ for every $i,j\in\{1,\ldots,n\}$.
Given a metric space $(X,d_X)$ and $p\in [1,\infty)$, we say that
$X$ has {\em Markov type} $p$ if there exists a constant $K>0$ such
that for every stationary reversible Markov chain
$\{Z_t\}_{t=0}^\infty$ on $\{1,\ldots,n\}$, every mapping
$f:\{1,\ldots,n\}\to X$ and every time $t\in \mathbb N$,
\begin{eqnarray}\label{eq:defMarkov}
\E \big[ d_X(f(Z_t),f(Z_0))^p\big]\le K^p\,t\,\E\big[
d_X(f(Z_1),f(Z_0))^p\big].
\end{eqnarray}
The least such $K$ is called the Markov type $p$ constant of $X$,
and is denoted $M_p(X)$. This important concept was introduced by
Ball in~\cite{Bal} and has since found a variety of applications in
metric geometry, including applications to the theory of compression
exponents~\cite{ANP07,NP07}. We refer to~\cite{NPSS06} for examples
of spaces which have Markov type $p$. For our purposes it suffices
to mention that Banach spaces whose modulus of uniform smoothness
has power type $p$ have Markov type $p$~\cite{NPSS06}, and therefore
the Markov type of $L_p$, $p\in [1,\infty)$, is $\min\{p,2\}$.

In~\cite{NP07} a parameter $\beta^*(G)$ is defined to be the
supremum over all $\beta\ge 0$ for which there exists a symmetric
set of generators $S$ of $G$ and $c>0$ such that for all $t\in \N$,
\begin{eqnarray}\label{eq:assumption-ANP} \E\big[d_G(W_t,e)\big]\ge ct^\beta,
\end{eqnarray}
where $\{W_t\}_{t=0}^\infty$ is the canonical simple random walk on
the Cayley graph of $G$ determined by $S$, starting at the identity
element $e_G$. The proof in~\cite{ANP07} shows that if $(X,d_X)$ has
Markov type $p$ and $G$ is amenable then:
\begin{eqnarray}\label{eq:ANP}
\alpha^*_X(G)\le \frac{1}{p\beta^*(G)}.
\end{eqnarray}

In order to prove~\eqref{eq:intro Z wreath Z} we establish in
Section~\ref{sec:amenable beta} a crucial strengthening
of~\eqref{eq:ANP}. Given a symmetric probability measure $\mu$ on
$G$ let $\{g_k\}_{k=1}^\infty$ be i.i.d. elements of $G$ which are
distributed according to $\mu$. The $\mu$-random walk
$\{W_t^\mu\}_{t=0}^\infty$ is defined as $W_0^\mu=e_G$ and
$W_t^\mu=g_1g_2\cdots g_t$ for $t\in \N$. Let $\rho$ be a
left-invariant metric on $G$ such that $B_\rho(e_G,r)=\{x\in G:\
\rho(x,e)\le r\}$ is finite for all $r\ge 0$. Define
$\beta_p^*(G,\rho)$ to be the supremum over all $\beta\ge 0$ such
that there exists an increasing sequence of integers
$\{t_k\}_{k=1}^\infty$ and a sequence of symmetric probability
measures $\{\mu_k\}_{k=1}^\infty$ on $G$ satisfying
\begin{equation}\label{eq:intro integrability}
\forall k\in \N\ \ \int_G \rho(x,e_G)^pd\mu_k(x)<\infty\quad
\mathrm{and}\quad \lim_{k\to\infty} \big( t_k
\mu_k\left(G\setminus\{e_G\}\right)\big)=\infty.
\end{equation}
such that for all $k\in \N$,
$$
\E_{\mu_{k}}\left[\rho\left(W_{t_k}^{\mu_{k}},e_G\right)\right]\ge
t_k^{\beta }\left(
\E_{\mu_{k}}\left[\rho\left(W_1^{\mu_{k}},e_G\right)^p\right]\right)^{\beta}.
$$
%\begin{eqnarray}\label{eq:intro def betap}
%\beta_p^*(G,\rho)\coloneqq \sup_{\{\mu_t\}_{t=1}^\infty}
%\limsup_{t\to\infty}\frac{\log\left(\E_{\mu_t}\left[\rho\left(W_t^{\mu_t},e_G\right)\right]\right)}
%{\log\left(t \E_{\mu_t}
%\left[\rho\left(W_1^{\mu_t},e_G\right)^p\right]\right)},
%\end{eqnarray}
%where the supremum in~\eqref{eq:intro def betap} is over all
%sequences of symmetric probability measures $\{\mu_t\}_{t=1}^\infty$
%on $G$ satisfying
%\begin{equation}\label{eq:intro integrability}
%\forall t\in \N\ \ \int_G \rho(x,e_G)^pd\mu_t(x)<\infty\quad
%\mathrm{and}\quad \liminf_{t\to\infty} \big( t
%\mu_t\left(G\setminus\{e_G\}\right)\big)=\infty.
%\end{equation}
%When $\rho$ is the word metric induced by the symmetric generating
%set $S$ we will use the simplified notation
%$\beta_p^*(G,\rho)=\beta_p^*(G)$. This convention does not create
%any ambiguity since  $\beta_p^*(G,\rho)$ does not depend on the
%choice of the finite symmetric generating set $S$ (see
%Section~\ref{sec:amenable beta}).

In Section~\ref{sec:amenable beta} we show that if $G$ is amenable,
$\rho$ is a left-invariant metric on $G$ with respect to which all
balls are finite, and $(X,d_X)$ has Markov type $p$, then:
\begin{eqnarray}\label{eq:intro new ANP}
\alpha^*_X(G,\rho)\le \frac{1}{p\beta_p^*(G,\rho)},
\end{eqnarray}
where $\alpha^*_X(G,\rho)$ is the supremum over all $\alpha\ge 0$
for which there exists a $\rho$-Lipschitz map $f:G\to X$ which
satisfies $d_X(f(x),f(y))\ge c \rho(x,y)^\alpha$ (we previously
defined this parameter only when $\rho=d_G$). We refer to the
discussion in Section~\ref{sec:amenable beta} for more information
on the parameter $\beta_p^*(G,\rho)$. At this point it suffices to
note that $\beta_p^*(G,d_G)\ge \beta^*(G)$, and
therefore~\eqref{eq:intro new ANP} is stronger than~\eqref{eq:ANP},
since we now consider a variant of~\eqref{eq:assumption-ANP} where
the walk can be induced by an arbitrary symmetric probability
measure, and the measure itself is allowed to depend on the time
$t$. It turns out that~\eqref{eq:intro new ANP} is a crucial {\em
strict} improvement over~\eqref{eq:ANP}, and we require the full
force of this strengthening: we shall use non-standard random walks
(i.e., not only the canonical walk on the Cayley graph of $G$), as
well as an adaptation of the walk to the time $t$
in~\eqref{eq:assumption-ANP}, in addition to invariant metrics
$\rho$ other than the word metric $d_G$.

We establish~\eqref{eq:intro Z wreath Z} by showing that for every
$p\in [1,2)$ we have $\beta_p^*(\Z\bwr
\Z,d_{\Z\bwr\Z})=\frac{2p-1}{p^2}>\frac{3}{4}=\beta^*(\Z\bwr\Z)$ (it
follows in particular that~\eqref{eq:intro new ANP} is indeed
strictly stronger than~\eqref{eq:ANP}. Note that $\Z\bwr\Z$ is
amenable and $L_p$ has Markov type $p$, so we are allowed to
use~\eqref{eq:intro new ANP}). This is achieved by considering a
random walk induced on $\Z\bwr \Z$ from a random walk on $\Z$ whose
increments are discrete versions of $q$-stable random variables for
every $q>p$. We refer to Section~\ref{sec:Lp comp} for the details.
We believe that there is a key novel feature of our proof which
highlights the power of random walk techniques in embedding
problems: we adapt the random walk on $G$ to the target space $L_p$.
Previously~\cite{LMN02,BLMN05,NPSS06,ANP07,NP07} Markov type was
used in embedding problems by considering a Markov chain on the
space we wish to embed which arises intrinsically, and ``ignored"
the intended target space: such chains are typically taken to be the
canonical random walk on some graph, but a different example appears
in~\cite{BLMN05}, where embeddings of arbitrary subsets $A$ of the
Hamming cube $(\{0,1\}^n,\|\cdot\|_1)$ are investigated via a
construction of a special random walk on $A$ which captures the
``largeness" of $A$. Nevertheless, in all known cases the geometric
object which was being embedded dictated the study of some natural
random walk, while in our computation of $\alpha_p^*(\Z\bwr \Z)$ the
target space $L_p$ influences the choice of the random walk.

Recall that we mentioned above that prior to~\cite{ANP07} the best
known upper bound~\cite{AGS06} on $\alpha^*(\Z\bwr \Z)$ was
$\frac34$. An inspection of the proof of this bound in~\cite{AGS06}
reveals that it considered only points in the normal subgroup of
$\Z\bwr \Z$ consisting of all configurations where the lamplighter
is at $0$, i.e., the {\em zero section} of $\Z\bwr\Z$:
$$
(\Z\bwr\Z)_0\coloneqq\{(f,x)\in \Z\bwr\Z:\ x=0\}\lhd \Z\bwr \Z.
$$
Thus~\cite{AGS06} actually establishes the bound
$\alpha^*\left((\Z\bwr\Z)_0,d_{\Z\bwr \Z}\right)\le \frac34$. More
generally, an obvious variant of the proof of this fact
in~\cite{AGS06} (see Lemma 7.8 in~\cite{NP07}) shows that for $p\in
[1,2]$ we have $\alpha_p^*\left((\Z\bwr\Z)_0,d_{\Z\bwr \Z}\right)\le
\frac{p+1}{2p}$. Here we show that
\begin{equation}\label{eq:intro zer section}
\alpha_p^*\left((\Z\bwr\Z)_0,d_{\Z\bwr
\Z}\right)=\max\left\{\frac{p+1}{2p},\frac34\right\}.
\end{equation}
An alternative proof of the fact that the right-hand side
of~\eqref{eq:intro zer section} is greater than
$\alpha_p^*\left((\Z\bwr\Z)_0,d_{\Z\bwr \Z}\right)$, which belongs
to the framework of~\eqref{eq:intro new ANP}, is given in
Section~\ref{sec:zero section walk}, where we show that for every
$p\in [1,2]$ we have $\beta_p^*\left((\Z\bwr\Z)_0,d_{\Z\bwr
\Z}\right)=\frac{2}{p+1}$. The heart of~\eqref{eq:intro zer section}
is the construction of an embedding into $L_p$ of the zero section
$(\Z\bwr\Z)_0$ which achieves the claimed compression exponent. This
turns out to be quite delicate: a Fourier analytic argument
establishing this fact is presented in Section~\ref{sec:zero
section}.

It is worthwhile to note at this point that in all of our new
compression computations, namely~\eqref{eq:intro polywreath},
\eqref{eq:intro Z wreath Z} and~\eqref{eq:intro integrability}, we
claim that for some group $G$ equipped with an invariant metric
$\rho$ and for every $p\in [2,\infty)$ we have
$\alpha_p^*(G,\rho)=\alpha_2^*(G,\rho)$. This is  true since because
$L_2$ is isometric to a subset of $L_p$ we obviously have
$\alpha_p^*(G,\rho)\ge\alpha_2^*(G,\rho)$. In the reverse direction,
all of our upper bounds on $L_p$ compression exponents are based
on~\eqref{eq:intro new ANP}, and since both $L_2$ and $L_p$ have
Markov type $2$~\cite{NPSS06} the resulting upper bound for $L_p$
coincides with the upper bound for $L_2$. For this reason it will
suffice to prove all of our results when $p\in [1,2]$.

In Section~\ref{sec:extension} we apply the fact that
$\alpha^*\left((\Z\bwr\Z)_0,d_{\Z\bwr \Z}\right)\neq \alpha^*(\Z\bwr
\Z)$ to the Lipschitz extension problem. This classical problem asks
for geometric conditions on a pair of metric spaces $(X,d_X)$ and
$(Y,d_Y)$ which ensure that for any subset $A\subseteq X$ any
Lipschitz mapping $f:A\to Y$ can be extended to all of $X$. Among
the motivating themes for  research on the Lipschitz extension
problem is the belief that many classical extension theorems for
linear operators between Banach spaces have Lipschitz analogs. Two
examples of this phenomenon are the non-linear Hahn-Banach theorem
(see for example~\cite{WW75,BL00}), which corresponds to extension
of real valued functions while preserving their Lipschitz constant,
and the non-linear version of Maurey's extension
theorem~\cite{Bal,NPSS06}. It turns out that our investigation of
the Hilbert compression exponent of the zero section of $\Z\bwr \Z$
implies the existence of a Lipschitz function $f:(\Z\bwr\Z)_0\to
L_2$ which cannot be extended to a Lipschitz function defined on all
of $\Z\bwr \Z$. For those who believe in the above analogy between
the Lipschitz extension problem and the extension problem for linear
operators this fact might seem somewhat surprising: after all
$H=(\Z\bwr\Z)_0$ is a normal subgroup of $G=\Z\bwr \Z$ with
$G/H\cong \Z$, so it resembles a non-commutative version of a
subspace of co-dimension $1$ in a Banach space, for which the
Lipschitz extension problem is trivial (again by the Hahn-Banach
theorem). Nevertheless, the analogy with Banach spaces stops
here, as our result shows that the normal subgroup $H$ sits in $G$
in an ``entangled" way which makes it impossible to extend certain
Lipschitz functions while preserving the Lipschitz property.

To explain the connection with the Lipschitz extension problem take
$\psi:(\Z\bwr \Z)_0\to L_2$ which is $1$-Lipschitz and
$\|\psi(x)-\psi(y)\|_2\ge cd_{\Z\bwr \Z}(x,y)^{3/4}$ for all $x,y\in
(\Z\bwr \Z)_0$, where $c>0$ is a universal constant\footnote{It
isn't quite accurate that the fact that
$\alpha^*\left((\Z\bwr\Z)_0,d_{\Z\bwr \Z}\right)=\frac34$ implies
the existence of such a function $\psi$, since all we are assured is
a compression exponent lower bound of $\frac34-\e$ for all $\e>0$.
This is immaterial for the sake of the argument here in the
introduction---a precise proof is given in
Section~\ref{sec:extension}}. We claim that $\psi$ cannot be
extended  to a Lipschitz function $\Psi$ defined on all of
$\Z\bwr\Z$, so assume for the sake of contradiction that $\Psi$
extends $\psi$ and is Lipschitz. To arrive at a contradiction we
need to contrast the $\frac34$ lower bound on the compression
exponent of $\psi$ with the Markov type $2$ proof of the fact that
$\Psi$ cannot have compression larger than $\frac23$
from~\cite{ANP07}. Let $\{W_t\}_{t=0}^\infty$ be the canonical
random walk on $\Z\bwr\Z$ starting at the identity element. Writing
$W_t=(f_t,x_t)\in \Z\bwr \Z$ one can see that with high probability
$|x_t|\lesssim \sqrt{t}$, while the distance between $W_t$ and the
identity element is $\gtrsim t^{3/4}$. The fact that $L_2$ has
Markov type $2$ and $\Psi$ is Lipschitz says that we expect
$\|\Psi(W_t)-\Psi(W_0)\|_2$ to be $\lesssim \sqrt{t}$. But, if we
move $W_t$ to its closest point in the zero section $(\Z\bwr \Z)_0$
then the image under $\Psi$ will (using the Lipschitz condition)
move $\lesssim \sqrt{t}$. Using the compression inequality for
$\psi$ we deduce that for large enough $t$ we have $\sqrt{t}\gtrsim
\|\Psi(W_t)-\Psi(W_0)\|_2\gtrsim
\left(t^{3/4}\right)^{3/4}=t^{9/16}$, which is a contradiction. This
argument is, of course, flawed, since we are allowed to use the fact
that $L_2$ has Markov type $2$ only for Markov chains which are
stationary and reversible, and this is not the case for the
canonical random walk starting at the identity element.
Nevertheless, this proof can be salvaged using the same intuition:
in Section~\ref{sec:extension} we consider a certain finite subset
of $\Z\bwr \Z$ which lies within a narrow tubular neighborhood of
$(\Z\bwr \Z)_0$. We then apply the same ideas to the random walk
obtained by choosing a point in this subset uniformly at random and
preforming a random walk on the subset with appropriate boundary
conditions. We refer to Section~\ref{sec:extension} for the full
details. It is perhaps somewhat amusing to note here that while the
notion of Markov type was introduced by Ball~\cite{Bal} in order to
prove an extension theorem (Ball's extension theorem), here we use
Markov type for the opposite purpose---to prove a non-extendability
result.

Thus far we did not discuss the relation between the parameters
$\alpha_X^*(G)$ and $\alpha_X^\#(G)$ for some Banach space $X$. This
is, in fact, a subtle issue: it is unclear when
$\alpha_X^*(G)=\alpha_X^\#(G)$. Since for every $p\in [1,\infty)$
the free group $F_n$ on $n\ge 2$ generators satisfies
$\alpha_p^*(F_n)=1$ yet
$\alpha_p^\#(F_n)=\max\left\{\frac{1}{p},\frac12\right\}$
(see~\cite{GK04,NP07}) it follows that the compression exponent and
equivariant compression exponent can be different from each other, while in many
cases we know that these two invariants coincide: for example
$\alpha_p^*(C_2\bwr F_n)=\alpha_p^\#(C_2\bwr
F_n)=\max\left\{\frac{1}{p},\frac12\right\}$
(see~\cite{CSV07,NP07}). A useful result of Aharoni, Maurey and
Mityagin~\cite{AhaMauMit} for Abelian groups, and Gromov
(see~\cite{dCTV07}) for general amenable groups, says that for any
amenable group $G$ we have $\alpha_2^*(G)=\alpha_2^\#(G)$. This is
an obviously useful fact (examples of applications can be found
in~\cite{dCTV07,ANV07}): for example in~\cite{NP07} it was shown
that if $X$ is a Banach space whose modulus of uniform smoothness
has power type $p$ then for every finitely generated group $G$ we
have:
\begin{equation}\label{intro NP}
\alpha_X^\#(G)\le \frac{1}{p\beta^*(G)}.
\end{equation}
The bound~\eqref{intro NP} implies the bound~\eqref{eq:ANP} when $G$
is amenable and $X$ is Hilbert space due to the above reduction to
equivariant mappings for amenable groups and Hilbertian targets. At
the time of writing of~\cite{NP07} it was unclear
whether~\eqref{intro NP} implies~\eqref{eq:ANP} in general, since an
Aharoni-Maurey-Mityagin/Gromov type result was not known in
non-Hilbertian settings. In Section~\ref{sec:amenable beta} we
further improve~\eqref{intro NP} by showing that if $X$ is a Banach
space whose modulus of uniform smoothness has power type $p$ then:
\begin{equation}\label{intro betap equivariant}
\alpha_X^\#(G)\le \frac{1}{p\beta_p^*(G)}.
\end{equation}
In Section~\ref{sec:equivariant} we show that for every $p\in
[1,\infty)$ if $G$ is an amenable group and $X$ is a Banach space
then there exists a Banach space $Y$ which is finitely
representable\footnote{A Banach space $U$  is said to be finitely
representable in a Banach space $V$ if for every $\e>0$ and every
finite dimensional subspace $F\subseteq U$ there is a linear
operator $T:F\to V$ such that for every $x\in F$ we have
$\|x\|_U\le\|Tx\|_V\le (1+\e)\|x\|_U$.} in $\ell_p(X)$ and
\begin{eqnarray}\label{eq:equiv L_p}
\alpha_Y^\#(G)\ge \alpha_X^*(G).
\end{eqnarray}
Moreover, if $X=L_p$ then we can also take $Y=L_p$
in~\eqref{eq:equiv L_p}, and thus $\alpha_p^*(G)=\alpha_p^\#(G)$
when $G$ is amenable. Note also that if $X$ has modulus of uniform
smoothness of power type $p$ then so does $\ell_p(X)$, and hence so
does $Y$. Therefore by virtue of~\eqref{eq:equiv L_p} the
inequalities~\eqref{intro NP} and~\eqref{intro betap equivariant}
are indeed stronger than the inequalities~\eqref{eq:ANP}
and~\eqref{eq:intro new ANP} in full generality.

We end this introduction by commenting on why so much of the
literature (and also the present paper) focused  on compression
exponents of wreath products. The obvious answer is that groups such
as $\Z\bwr \Z$ are among the simplest examples of groups for which
it was unknown for a long time how to compute their compression
exponents. As it turns out, understanding such groups required new
ideas and new connections between geometric group theory and other
mathematical disciplines. But, there is also a deeper reason for our
interest in embeddings of wreath products. Ka\v{z}dan's
example~\cite{Kaz67} (see also~\cite{dlHV89}) of $\Z^2\rtimes
SL_2(\Z)$ shows that there can be two groups, each of which has
positive equivariant compression exponent, yet their semidirect
product fails to have a positive equivariant compression exponent,
and even fails the Haagerup property. It seems challenging to
characterize which semidirect products preserve the property of having
positive compression exponents, and wreath products, as examples of
semidirect products, are a good place to start trying to understand this
fundamental question. The literature on compression exponents of
wreath products shows that in many cases this operation preserves
the property of having positive compression exponent, but we do not
know if this is always true, even for amenable groups: the simplest
such example is the groups $C_2\bwr(C_2\bwr \Z)$ for which we do not
know if it has positive Hilbert compression exponent, even though
both $C_2$ and $C_2\bwr \Z$ have Hilbert compression exponent $1$.

\section{Preliminaries}

In what follows we fix two groups $G$ and $H$, which are generated
by the symmetric finite sets $S_G$ and $S_H$, respectively. The
corresponding left invariant word metrics will be denoted $d_G$ and
$d_H$, respectively. The canonical generating set of the wreath
product $G\bwr H$ is
\begin{equation*}\label{eq:cannonical generating wreath}
\left\{\left(\mathbf{e_G},x\right):\ x\in
S_H\right\}\cup\left\{\left(\delta_y,e_H\right):\ y\in S_G\right\},
\end{equation*}
where $\mathbf{e_G}:H\to G$ denotes the constant  $e_G$ function and
for $y\in G$ the function $\delta_y:H\to G$ takes the value $y$ at
$e_H$ and the value $e_H$ elsewhere.

Given a function $f:H\to G$ we denote its support by
$\supp(f)\coloneqq \{x\in H:\ f(x)\neq e_G\}$. For a finite subset
$A\subseteq H$ and $x,y\in H$ we let $\TSP(A;x,y)$ denote the length
of the shortest path in $H$ which starts at $x$, covers $A$, and
terminates at $y$, i.e.,
\begin{eqnarray*}\label{eq:defTSPxy}
\TSP(A;x,y)\coloneqq \inf\left\{\sum_{j=0}^{k-1}d_H(x_j,x_{j+1}):\
k\in \N,\ \ x=x_0,\ldots,x_k=y\in H\ \wedge\  A\subseteq
\{x_0,\ldots,x_k\}\right\}.
\end{eqnarray*}
Thus
$$
|A|+\TSP(A,x,y)=d_{C_2\bwr
H}\left(\left(\1_{y^{-1}A},y^{-1}x\right),(\0,0)\right),
$$
where $\0:H\to C_2$ denotes the constant $0$ function.
Following~\cite{NP07} we let $\mathscr{L}_G(H)$ denote the wreath
product of $G$ with $H$ where the set of generators of $G$ is taken
to be $G\setminus\{e_G\}$ (i.e. any two distinct elements of $G$ are
at distance $1$ from each other). In other words, the difference
between $\mathscr{L}_G(H)$ and the classical lamplighter group
$C_2\bwr H$ is that we allow the ``lamps" to have $G$ types of
different ``lights", where the cost of switching from one type of
light to another is $1$. Thus, with this definition it is immediate
that for every $(f,x),(g,y)\in \mathscr{L}_G(\Z)$ we have
\begin{eqnarray}\label{eq:the metric}
d_{\mathscr{L}_G(\Z)}\big((f,x),(g,y)\big)= d_{C_2\bwr
H}\left((\1_{y^{-1}\supp\left(fg^{-1}\right)},y^{-1}x),(\0,0)\right)=\left|\supp\left(fg^{-1}\right)\right|+\TSP\left(\supp\left(fg^{-1}\right);x,y\right).
\end{eqnarray}
Moreover, distances in the wreath product $G\bwr H$, equipped with
the canonical generating set, can be computed as follows:
\begin{eqnarray}\label{eq;distances in wreath}
d_{G\bwr
H}\big((f,x),(g,y)\big)=\TSP\left(\supp\left(fg^{-1}\right);x,y\right)+\sum_{x\in
H} d_G(f(x),g(x)).
\end{eqnarray}

The following lemma generalizes Lemma 3.1 in~\cite{NP07}, which
deals with the special case $H=\Z$ (in which case the proof is
 easier).

\begin{lem}\label{lem:two lamps} Assume that $G$ contains at least two elements. Then for any $p\ge 1$ we have
$$ \alpha_p^*\big(\mathscr{L}_G(H)\big)=\alpha_p^*\left(C_2\bwr
H\right).$$
\end{lem}

\begin{proof} Obviously $ \alpha_p^*\left(\mathscr{L}_G(H)\right)\le \alpha_p^*\left(C_2\bwr
H\right)$, since $\mathscr{L}_G(H)$ contains an isometric copy of
$C_2\bwr H$. To prove the reverse direction we may assume that
$\alpha_p^*\left(C_2\bwr H\right)>0$. Fix
$0<\alpha<\alpha_p^*\left(C_2\bwr H\right)$ and a mapping $\theta:
C_2\bwr \Z\to L_p$ satisfying
\begin{eqnarray}\label{eq:satisfying}
(f,x),(g,y)\in C_2\bwr H\implies d_{C_2\bwr
H}\big((f,x),(g,y)\big)^{\alpha}\lesssim\|\theta(f,x)-\theta(g,y)\|_p\lesssim
d_{C_2\bwr H}\big((f,x),(g,y)\big).
\end{eqnarray}
Let $\{\e_z\}_{z\in G\setminus\{e_G\}}$ be i.i.d. $\{0,1\}$-valued
Bernoulli random variables, defined on some probability space
$(\Omega,\Pr)$. For every $f:H\to G$ define a random mapping $\e_f:
H\to C_2$ by $$ \e_f(z)\coloneqq \left\{\begin{array}{ll}\e_{f(z)} &
\mathrm{if}\
f(z)\neq e_G,\\
0 & \mathrm{if}\ f(z)=e_G.\end{array}\right.$$ We now define an
embedding $F:\mathscr{L}_G(H)\to L_p(\Omega,L_p)$ by
$$
F(f,x)\coloneqq \theta(\e_f,x).
$$
Given $(f,x),(g,y)\in G\bwr H$ denote $A\coloneqq
\supp\left(fg^{-1}\right)=\{z\in H:\ f(z)\neq g(z)\}$. We also
denote by $A_\e\subseteq H$ the random subset $\supp(\e_f-\e_g)$. By
definition $A_\e\subseteq A$, so that $\TSP(A_\e;x,y)\le
\TSP(A;x,y)$. Hence:
\begin{multline*}
\|F(f,x)-F(g,y)\|_{L_p(\Omega,L_p)}^p=\E
\left[\left\|\theta(\e_f,x)-\theta(\e_g,y)\right\|_p^p\right]\stackrel{\eqref{eq:satisfying}}{\lesssim}
\E \left[d_{C_2\bwr
H}\big((\e_f,x),(\e_g,y)\big)^p\right]\\=\E\left[\TSP(A_\e;x,y)^p\right]\le
\E\left[\TSP(A;x,y)^p\right]\stackrel{\eqref{eq:the
metric}}{=}d_{\mathscr{L}_G(\Z)}\big((f,x),(g,y)\big)^p.
\end{multline*}
In the reverse direction, observe that
\begin{eqnarray}\label{eq:tsp union}
\TSP(A;x,y)\le 2\TSP(A_\e;x,y)+\TSP(A\setminus A_\e;x,y),
\end{eqnarray}
 since
given a path $\gamma$ that starts at $x$, ends at $y$, and covers
$A_\e$, and a path $\delta$ that starts at $x$, ends at $y$, and
covers $A\setminus A_\e$, we can consider the path that starts as
$\gamma$, retraces $\gamma$'s steps from $y$ back to $x$, and then
continues as $\delta$ from $x$ to $y$.  Hence,
\begin{eqnarray}\label{eq:before expectation} d_{\mathscr{L}_G(\Z)}\big((f,x),(g,y)\big)^{p\alpha}\stackrel{\eqref{eq:the
metric}}{=}\TSP(A;x,y)^{p\alpha} \stackrel{\eqref{eq:tsp
union}}{\lesssim} \TSP(A_\e;x,y)^{p\alpha }+\TSP(A\setminus
A_\e;x,y)^{p\alpha}.
\end{eqnarray}
But by the symmetry of our construction the random subsets $A_\e$
and $A\setminus A_\e$ are identically distributed. So, taking
expectation in~\eqref{eq:before expectation} we see that
\begin{multline*}
d_{\mathscr{L}_G(\Z)}\big((f,x),(g,y)\big)^{p\alpha}\lesssim
\E\left[\TSP(A_\e;x,y)^{p\alpha }\right]=\E \left[d_{C_2\bwr
H}\big((\e_f,x),(\e_g,y)\big)^{p\alpha}\right]\\\stackrel{\eqref{eq:satisfying}}{\lesssim}
\E\left[\|\theta(\e_f,x)-\theta(\e_g,y)\|_p^p\right]=\|F(f,x)-F(g,y)\|_{L_p(\Omega,L_p)}^p.
\end{multline*}
Thus $G\bwr H$ embeds into $L_p(\Omega,L_p)$ with compression
$\alpha$, as required.
\end{proof}

A combination of Lemma~\ref{lem:two lamps} and Theorem 3.3
in~\cite{NP07} yields the following corollary:

\begin{cor}\label{coro:Lp}
Let $G, H$ be nontrivial groups and $p\ge 1$. Then
$$
\min\left\{\alpha_p^*(G),\alpha_p^*(C_2\bwr H)\right\}\ge
\frac{1}{p}\implies \alpha_p^*(G\bwr H)\ge
\frac{p\alpha_p^*(G)\alpha_p^*(C_2\bwr
H)}{p\alpha_p^*(G)+p\alpha_p^*(C_2\bwr H)-1},
$$
and
\begin{eqnarray*}\label{eq:Lp second}
\min\left\{\alpha_p^*(G),\alpha_p^*(C_2\bwr H)\right\}\le
\frac{1}{p}\implies \alpha_p^*(G\bwr H)\ge
\min\left\{\alpha_p^*(G),\alpha_p^*(C_2\bwr H)\right\}.
\end{eqnarray*}
\end{cor}

We end this section with a simple multi-scale estimate for the
length of traveling salesmen tours (see for example~\cite{Steele97}
for a similar estimate). For $r\ge 0$ and $x\in H$ we let
$B_H(x,r)\coloneqq \left\{y\in H:\ d_H(x,y)\le r\right\}$ be the
closed ball centered at $x$ with radius $r$. For a bounded set
$A\subseteq H$ and $r>0$ we let $N(A,r)$ be the smallest integer
$n\in \N$ such that there exists $x_1,\ldots,x_n\in H$ for which
$A\subseteq \bigcup_{m=1}^n B_H(x_m,r)$. Finally, for $\ell\ge 0$
let $\TSP_\ell(A)$ denote the length of the shortest path starting
from $e_H$, coming within a distance of at most $2^{\ell-1}$ from
every point in $A$, and returning to $e_H$, i.e.
$$
\TSP_\ell(A)\coloneqq \inf\left\{\sum_{j=0}^{k-1}d_H(x_j,x_{j+1}):\
k\in \N,\ \  e_H=x_0,\ldots,x_k=e_H\in H,\ A\subseteq
\bigcup_{j=0}^k B_H\left(x_j,2^{\ell-1}\right)\right\}.
$$
Thus $\TSP(A)\coloneqq \TSP(A;e_H,e_H)=\TSP_0(A)=d_{C_2\bwr
H}\big((\1_A,e_H),(\mathbf{0}, e_H)\big)$ is the length of the
shortest path starting from $e_H$, covering $A$, and returning to
$e_H$. We shall use the following easy bound, which holds for every
$k,\ell\in \N\cup\{0\}$:
\begin{eqnarray}\label{eq:easybound}
A\subseteq B_H\left(e_H,2^k\right)\implies \TSP_\ell(A)\le
3\sum_{j=\ell}^k 2^{j} N\left(A,2^{j-1}\right).
\end{eqnarray}
The inequality~\eqref{eq:easybound} is valid when $\ell\ge k+1$
since in that case $\TSP_\ell(A)=0$. Now~\eqref{eq:easybound}
follows by induction from the inequality $\TSP_{\ell-1}(A)\le
\TSP_\ell(A)+3\cdot2^{\ell-1} N\left(A,2^{\ell-2}\right)$. This
inequality holds true
 since we can take a
set $C\subseteq H$ of size $N\left(A,2^{\ell-2}\right)$ such that
$\bigcup_{x\in C} B_H\left(x,2^{\ell-2}\right)\supseteq A$, and also
take a path $\Gamma\subseteq H$ of length $\TSP_\ell(A)$ which
starts from $e_H$, comes within a distance of at most $2^{\ell-1}$
from every point in $A$, and returns to $e_H$. If we append to
$\Gamma$ a shortest path from each $x\in C$ to its closest neighbor
in $\Gamma$ (and back) we obtain a new path of length at most
$\TSP_\ell(A)+2\left(2^{\ell-1}+2^{\ell-2}\right)|C|\le
\TSP_\ell(A)+3\cdot2^{\ell-1}|C|$ which starts from $e_H$, comes
within a distance of at most $2^{\ell-2}$ from every point in $A$,
and returns to $e_H$, as required.

\section{Wreath products of groups with polynomial
growth}\label{section:polywreath}

The goal of this section is to prove the following theorem:

\begin{thm}\label{thm:polywreath} Let $G,H$ be nontrivial finitely generated
groups, and assume that $H$ has polynomial growth. Then for every
$p\in [1,2]$ we have
\begin{eqnarray}\label{eq:embedding}
\alpha_p^*(G\bwr H)\ge \min\left\{\frac{1}{p},\alpha_p^*(G)\right\}.
\end{eqnarray}
In particular, if the growth rate of $H$ is at least quadratic then
for every $p\in [1,2]$ we have
\begin{eqnarray}\label{eq:1/p}\alpha_p^*(\Z\bwr H)=\alpha_p^*(C_2\bwr
H)=\frac{1}{p}.\end{eqnarray}
\end{thm}

\begin{proof}
We shall first explain how to deduce the identity~\eqref{eq:1/p}.
The lower bound $\alpha_p^*(\Z\bwr H)=\alpha_p^*(C_2\bwr H)\ge
\frac{1}{p}$ is a consequence of~\eqref{eq:embedding}. Since for
$p\in [1,2]$ the Banach space $L_p$ has Markov type $p$
(see~\cite{Bal}), the result of Austin, Naor and Peres~\cite{ANP07}
implies that $\alpha_p^*(G\bwr H)\le \frac{1}{p\beta^*(G\bwr H)}$.
But, as we proved in~\cite{NP07}, since the growth of $H$ is at
least quadratic we have $\beta^*(G\bwr H)=1$.

To prove~\eqref{eq:embedding} note that by Corollary~\ref{coro:Lp}
it is enough to show that
\begin{eqnarray}\label{eq:goal lamplighter}
\alpha_p^*(C_2\bwr H)\ge \frac{1}{p}. \end{eqnarray} Recall that for
$r\ge 0$ and $x\in H$ we let $B_H(x,r)\coloneqq \left\{y\in H:\
d_H(x,y)\le r\right\}$ be the closed ball centered at $x$ with
radius $r$. Assume that $H$ has polynomial growth $d$, i.e., that
for every $r\ge 1$ we have
\begin{eqnarray}\label{eq:polygrowth}
ar^d\le\left|B_H(e,r)\right|\le b r^d \end{eqnarray} for some
$a,b>0$ which do not depend on $r$. We shall show that for every
$1<p\le 2$ and $\e\in (0,1/p)$ there is a function $F: C_2\bwr H\to
L_p$ such that for all $(f,x),(g,y)\in C_2\bwr H$ we have
\begin{eqnarray}\label{eq:exponent attained}
d_{C_2\bwr H}\big((f,x),(g,y)\big)^{\frac{1}{p}-\e}\lesssim
\|F(f,x)-F(g,y)\|_p\lesssim d_{C_2\bwr H}\big((f,x),(g,y)\big),
\end{eqnarray}
where here, and in the remainder of the proof of
Theorem~\ref{thm:polywreath}, the implied constants depend only on
$a,b,p,d,\e$. Moreover, we will show that we can take $\e=0$
in~\eqref{eq:exponent attained} if $(H,d_H)$ admits a bi-Lipschitz
embedding into $L_p$. Note that~\eqref{eq:exponent attained} implies
also the case $p=1$ of Theorem~\ref{thm:polywreath} since $L_p$ is
isometric to a subspace of $L_1$ for all $p\in (1,2]$ (see
e.g.~\cite{WW75}).

Let $\Omega$ be the disjoint union of the sets of functions $f:A\to
C_2$ where $A$ ranges over all finite subsets of $H$, i.e.
$$
\Omega\coloneqq \bigcup_{\substack{A\subseteq H\\ |A|<\infty}}C_2^A.
$$
We will work with the Banach space $\ell_\infty(\Omega)$, and denote
its standard coordinate basis by $$ \Big\{v_f:\ f:A\to C_2,\
A\subseteq H,\ |A|<\infty\Big\}.$$

Fix   a $1$-Lipschitz function $\f:[0,\infty)\to [0,1]$ which equals
$0$ on $[0,1]$ and equals $1$ on $[2,\infty)$. For every $(f,x)\in
C_2\bwr H$ define a function $\Psi_0(f,x)\in \ell_\infty(\Omega)$ by
\begin{eqnarray}\label{eq:defF}
\Psi_0(f,x)\coloneqq \sum_{k=0}^\infty 2^{-(d-1)k/p} \sum_{y\in
H}\f\left(\frac{d_H(x,y)}{2^k}\right)v_{f\upharpoonright_{
B_H(y,2^k)}}.
\end{eqnarray}
We shall first check that $\Psi_0-\Psi_0(\0,e_H)\in Z^1(H,\pi)$ for
an appropriately chosen action $\pi$ of $C_2\bwr H$ on
$\ell_p(\Omega)$. Recall that the product on $C_2\bwr H$ is given by
$(f,x)(g,y)=(f+T_x(g),xy)$, where $T_x(g)(z)\coloneqq
g\left(x^{-1}z\right)$. Given $(f,x)\in C_2\bwr H$ and a finite
subset $A\subseteq H$ define a bijection $\tau_{(f,x)}^A:C_2^A\to
C_2^{xA}$ by $ \tau_{(f,x)}^A(h)\coloneqq f+T_x(h)$. Note  that for
all $(f,x),(g,y)\in C_2\bwr H$ and every finite $A\subseteq H$ we
have
%\begin{eqnarray*}
%\tau_{(f,x)(g,y)}^A(h)(z)=\tau_{(
%f+T_x(g),xy)}^A(h)=f+T_x(g)+T_{xy}(h)=f+T_x\big(g+T_y(h)\big)=\tau_{(f,x)}^{yA}\big(g+T_y(h)\big)
%=\tau^{yA}_{(f,x)}\circ \tau^A_{(g,y)}(h).
%\end{eqnarray*}
%Thus for every finite $A\subseteq H$ we have
\begin{eqnarray}\label{eq:invariance1}
\tau_{(f,x)(g,y)}^A= \tau^{yA}_{(f,x)}\circ \tau^A_{(g,y)}.
\end{eqnarray}
Hence if we define
$$
\pi(f,x)\left(\sum_{\substack{A\subseteq H\\|A|<\infty}}\sum_{h\in
C_2^A}\alpha_hv_h\right)\coloneqq \sum_{\substack{A\subseteq
H\\|A|<\infty}}\sum_{h\in C_2^A}\alpha_hv_{\tau_{(f,x)}^A(h)},
$$
then $\pi$ is a linear isometric action of $C_2\bwr H$ on
$\ell_p(\Omega)$ for all $p\in [1,\infty]$ ($\pi(f,x)$ corresponds
to a permutation of the coordinates and hence is an isometry. The
fact that $\pi\big((f,x)(g,y)\big)=\pi(f,x)\pi(f,y)$ is an immediate
consequence of~\eqref{eq:invariance1}). The
definition~\eqref{eq:defF} ensures that for every $(f,x),(g,y)\in
C_2\bwr H$ we have $\Psi_0\big((f,x)(g,y)\big)=\pi(f,x)\Psi_0(g,y)$.
Hence, if we define $\Psi(f,x)\coloneqq \Psi_0(f,x)-\Psi_0(\0,e_H)$
then $\Psi\in Z^1(H,\pi)$.

Note that $\Psi(\0,e_H)=0$ and
$$
\Psi(\1_{\{e_H\}},e_H)=\sum_{k=0}^\infty 2^{-(d-1)k/p}\sum_{y\in
B_H(e_H,2^k)}\f\left(\frac{d_H(e_H,y)}{2^k}\right)\left(v_{\delta_{e_H}\upharpoonright_{
B_H(y,2^k)}}-v_{\0\upharpoonright_{B_H(y,2^k)}}\right)=0,
$$
where we used the fact that $\f(t)=0$ for $t\in [0,1]$. Moreover,
for every $s\in S_H$ we have
\begin{eqnarray*}
\left\|\Psi(\0,s)\right\|_p^p&=&\sum_{k=0}^\infty 2^{-(d-1)k}
\sum_{y\in
H}\left|\f\left(\frac{d_H(s,y)}{2^k}\right)-\f\left(\frac{d_H(e_H,y)}{2^k}\right)\right|^p\\&=&\sum_{k=0}^\infty
2^{-(d-1)k} \sum_{\substack{y\in H\\2^k-1\le d_H(e_H,y)\le
2^{k+1}+1}}\left|\f\left(\frac{d_H(s,y)}{2^k}\right)-\f\left(\frac{d_H(e_H,y)}{2^k}\right)\right|^p\\
&\le& \sum_{k=0}^\infty 2^{-(d-1)k} \cdot 2^{-kp} \left|\left\{y\in
H:\
2^k-1\le d_H(e_H,y)\le 2^{k+1}+1\right\}\right|\\
&\le& \sum_{k=0}^\infty 2^{-(d-1)k} \cdot 2^{-kp}\cdot
b\left(2^{k+1}+1\right)^d\\
&\le& 4^db\sum_{k=0}^\infty 2^{-k(p-1)}\lesssim 1,
\end{eqnarray*}
Where we used the fact that $p>1$. Since $\Psi$ is equivariant and
the set $\{(\1_{\{e_H\}},e_H)\}\cup\{(\0,s):\ s\in S_H\}$ generates
$C_2\bwr H$, we deduce that
\begin{eqnarray}\label{eq:lip condition}
\|\Psi\|_{\Lip}\lesssim 1.
\end{eqnarray}

Suppose now that $f:H\to C_2$ and let $m\in \N$ be the minimum
integer such that $\supp(f)\subseteq B_H(e_H,2^m)$. Then
\begin{multline}\label{eq:before jones}
\|\Psi(f,e_H)\|_p^p\ge \sum_{k=0}^\infty
2^{-(d-1)k}\sum_{\substack{y\in H\\
f\upharpoonright_{B_H(y,2^k)}\neq
\0\upharpoonright_{B_H(y,2^k)}}}\f\left(\frac{d_H(e_H,y)}{2^k}\right)^p\\
\ge \sum_{k=0}^\infty 2^{-(d-1)k}\left|\left\{y\in H:\ d_H(e_H,y)\ge
2^{k+1}\ \wedge\ \supp(f)\cap B_H(y,2^k)\neq
\emptyset\right\}\right|.
\end{multline}
Fix $k\le m-3$ and denote $n=N\left(\supp(f),2^{k-1}\right)$. Let
$x_1,\ldots, x_n\in H$ satisfy
\begin{equation}\label{eq:cover}\supp(f)\subseteq \bigcup_{i=1}^n
B_H\left(x_i,2^{k-1}\right).\end{equation} By the minimality of $n$
we are ensured that the balls
$\left\{B_H\left(x_i,2^{k-2}\right)\right\}_{i=1}^n$ are disjoint
and that there exists $y_i\in B_H\left(x_i,2^{k-1}\right)\cap
\supp(f)$. Write
$$
I\coloneqq \left\{i\in \{1,\ldots, n\}:\ d_H(y,e_H)\ge 2^{k+1}\
\forall y\in B_H\left(x_i,2^{k-2}\right)\right\}.
$$
Note that if $i\in I$ and $y\in B_H\left(x_i,2^{k-2}\right)$ then
$d_H(y_i,y)\le d_H(y_i,x_i)+d_H(y,x_i)\le 2^{k-1}+2^{k-2}<2^k$. Thus
in this case $\supp(f)\cap B_H(y,2^k)\neq \emptyset$, and therefore
\begin{equation}\label{eq:lower with I}
\left|\left\{y\in H:\ d_H(e_H,y)\ge 2^{k+1}\ \wedge\ \supp(f)\cap
B_H(y,2^k)\neq \emptyset\right\}\right|\ge
|I|\left|B_H\left(e_H,2^{k-2}\right)\right|\gtrsim 2^{kd} |I|.
\end{equation}
We shall now bound $|I|$ from below. By the minimality of $m$ there
exists $z\in \supp(f)$ such that $d_H(e_H,z)>2^{m-1}$.
By~\eqref{eq:cover} there is some $i\in \{1,\ldots n\}$ for which
$d_H(z,x_i)\le 2^{k-1}$. If $y\in B_H\left(x_i,2^{k-2}\right)$ then
$$d_H(y,e_H)\ge
d_H(e_H,z)-d_H(z,x_i)-d_H(x_i,y)> 2^{m-1}-2^{k-1}-2^{k-2}\ge
2^{k+1},$$ since by assumption $k\le m-3$. This shows that $|I|\ge
1$. Write $J\coloneqq \{1,\ldots,n\}\setminus I$. For each $i\in J$
there is some $y\in B_H\left(x_i,2^{k-2}\right)$ for which
$d_H(e_H,y)<2^{k+1}$. Hence $B_H\left(x_i,2^{k-2}\right)\subseteq
B_H\left(e_H, 2^{k+2}\right)$. Since the balls
$\left\{B_H\left(x_i,2^{k-2}\right)\right\}_{i=1}^n$ are disjoint it
follows that
$$
|J|a2^{(k-2)d}\stackrel{\eqref{eq:polygrowth}}{\le}
|J|\left|B_H\left(e_H,2^{k-2}\right)\right|\le
\left|B_H\left(e_H,2^{k+2}\right)\right|\stackrel{\eqref{eq:polygrowth}}{\le}
b2^{(k+2)d}.
$$
Thus $ n-|I|=|J|\lesssim 1$, which implies that $|I|\gtrsim n$.
Plugging this bound into~\eqref{eq:lower with I} we see that for
every $k\le m-3$ we have
$$
\left|\left\{y\in H:\ d_H(e_H,y)\ge 2^{k+1}\ \wedge\ \supp(f)\cap
B_H(y,2^k)\neq \emptyset\right\}\right|\gtrsim 2^{kd}
N\left(\supp(f),2^{k-1}\right).
$$
In combination with~\eqref{eq:before jones} we see that
\begin{eqnarray}\label{eq:before F0}
\|\Psi(f,e_H)\|_p^p\gtrsim \sum_{k=0}^{m-3} 2^{-(d-1)k}\cdot 2^{kd}
N\left(\supp(f),2^{k-1}\right)=\sum_{k=0}^{m-3} 2^{k}
N\left(\supp(f),2^{k-1}\right).
\end{eqnarray}
We claim that
\begin{eqnarray}\label{eq:for F0}
 \left|\supp(f)\right|+\sum_{k=0}^{m-3}
2^{k} N\left(\supp(f),2^{k-1}\right)\gtrsim d_{C_2\bwr
H}\big((f,e_H),(\0,e_H)\big).
\end{eqnarray}
Indeed, by combining~\eqref{eq:easybound} (with $\ell=0$)
and~\eqref{eq;distances in wreath} we see that
\begin{eqnarray}\label{eq:all the way to m}
\left|\supp(f)\right|+\sum_{k=0}^{m} 2^{k}
N\left(\supp(f),2^{k-1}\right)\gtrsim d_{C_2\bwr
H}\big((f,e_H),(\0,e_H)\big).
\end{eqnarray}
To check that~\eqref{eq:all the way to m} implies~\eqref{eq:for F0}
note that is is enough to deal with the case $\supp(f)\neq
\emptyset$, and that the fact that $\supp(f)\subseteq
B_H\left(e_H,2^m\right)$, combined with the doubling condition for
$(H,d_H)$, implies that for $k\in \{m-2,m-1,m\}$ we have
$N\left(\supp(f),2^{k-1}\right)\lesssim 1$. Thus~\eqref{eq:all the
way to m} implies~\eqref{eq:for F0} by inspecting the cases $m<3$
and $m\ge 3$ separately.

Fix $\e\in (0,1)$. By Assouad's theorem~\cite{Ass83} (see also the
exposition of this theorem in~\cite{Hei01}), since $H$ has
polynomial growth, and hence is a doubling metric space, there is a
function $\theta: H\to L_p$ such that for all $x,y\in H$ we have
\begin{eqnarray}\label{eq:assoud}
d_H(x,y)^{1-\e}\le \|\theta(x)-\theta(y)\|_p\lesssim
d_H(x,y)^{1-\e}\le d_H(x,y).
\end{eqnarray}
By translation we may assume that $\theta(e_H)=0$. We can now define
our embedding $$F:C_2\bwr H\to \ell_p(\Omega)\oplus \ell_p(H)\oplus
L_p$$ by $ F=\Psi\oplus f\oplus \theta$ (here we identify a finitely supported function $f:H\to  C_2$ as a member of $\R^H$, and hence a member of $\ell_p(H)$). Then $\|F\|_{\Lip}\coloneqq
L\lesssim 1$. Thus in order to prove~\eqref{eq:exponent attained},
and hence to complete the proof of Theorem~\ref{thm:polywreath}, it
remains to show that for all $(f,x)\in C_2\bwr H$ we have
\begin{eqnarray}\label{eq:three term}
d_{C_2\bwr
H}\big((f,x),(\0,e_H)\big)^{(1-\e)/p}\lesssim\|F(f,x)-F(\0,e_H)\|_p=\left(\|\Psi(f,x)\|_p^p+\left|\supp(f)\right|+\|\theta(x)\|_p^p\right)^{1/p}.
\end{eqnarray}
A combination of~\eqref{eq:before F0} and \eqref{eq:for F0} implies
that there exists $\eta>0$ which depends only on $a,b,d,p,\e$ such
that
$$
\eta d_{C_2\bwr H}\big((f,e_H),(\0,e_H)\big)^{1/p}\le
\left(\|\Psi(f,e_H)\|_p^p+\left|\supp(f)\right|\right)^{1/p}=\|F(f,e_H)-F(\0,e_H)\|_p.
$$
Hence
\begin{eqnarray*}
\|F(f,x)-F(\0,e_H)\|_p&\ge&
\|F(f,e_H)-F(\0,e_H)\|_p-\|F(f,x)-F(f,e_H)\|_p\\&\ge& \eta
d_{C_2\bwr H}\big((f,e_H),(\0,e_H)\big)^{1/p}-Ld_H(x,e_H)\\&\ge&
\eta \left[\max\left\{0,d_{C_2\bwr
H}\big((f,x),(\0,e_H)\big)-d_{C_2\bwr
H}\big((f,x),(f,e_H)\big)\right\}\right]^{1/p}-Ld_H(x,e_H)\\&=&\eta
\left[\max\left\{0,d_{C_2\bwr
H}\big((f,x),(\0,e_H)\big)-d_H(x,e_H)\right\}\right]^{1/p}-Ld_H(x,e_H)
\\&\ge& \frac{\eta}{4}d_{C_2\bwr H}\big((f,x),(\0,e_H)\big)^{1/p}\\
&\ge& \frac{\eta}{4}d_{C_2\bwr
H}\big((f,x),(\0,e_H)\big)^{(1-\e)/p},
\end{eqnarray*}
provided that
\begin{equation}\label{eq:for assouad}
d_H(x,e_H)\le \min \left\{\frac{\eta}{4L}d_{C_2\bwr
H}\big((f,x),(\0,e_H)\big)^{1/p},\frac12d_{C_2\bwr
H}\big((f,x),(\0,e_H)\big)\right\}.
\end{equation}
But if~\eqref{eq:for assouad} fails then $d_H(x,e_H)\gtrsim
d_{C_2\bwr H}\big((f,x),(\0,e_H)\big)^{1/p}$, in which case we can
use~\eqref{eq:assoud} to deduce that
$$
\|\theta(x)\|_p\ge d_H(e_H,x)^{1-\e}\gtrsim d_{C_2\bwr
H}\big((f,x),(\0,e_H)\big)^{(1-\e)/p},
$$
which implies~\eqref{eq:three term} and concludes the proof
of~\eqref{eq:exponent attained}.
\end{proof}

\begin{remark}\label{rem:attained} {\em Since the only reason for the
loss of $\e$ in~\eqref{eq:exponent attained} is the use of Assouad's
embedding in~\eqref{eq:assoud} we see that if $p>1$ and $(H,d_H)$
admits a bi-Lipschitz embedding into $L_p$ and has at least
quadratic growth then $\alpha_p^*(C_2\bwr H)=\frac{1}{p}$ is
attained. \fin}
\end{remark}

\begin{remark}\label{rem:old Z^2}
{\em In~\cite{NP07} it was shown that $\alpha_2^*(C_2\bwr \Z^2)\ge
\frac12$ via an embedding which we now describe. We are doing so for
several reasons. First of all there are some typos in the formulae
given for the embedding in~\cite{NP07} and we wish to take this
opportunity to publish a correct version. Secondly the embedding was
given in~\cite{NP07} without a detailed proof of its compression
bounds, and since it is based on a different and simpler approach
than our proof of Theorem~\ref{thm:polywreath} it is worthwhile to
explain it here. Most importantly, there are several ``coincidences"
which allow this approach to yield sharp bounds on
$\alpha_p^*(C_2\bwr \Z^d)$ only when $p=2$ and $d=2$, and we wish to
explain these subtleties here. We will therefore first describe the
embedding scheme in~\cite{NP07} for general $p\in[1,2]$ and $d\ge 2$
and then specialize to the case $p=d=2$. }

{\em
 Let $\left\{v_{y,r,g}:y\in \Z^d,\ r\in
\N\cup\{0\},\ g:y+[-r,r]^d\to \{0,1\}\right\}$ be a system of
disjoint unit vectors in $L_p$. Fix a parameter $\gamma>0$ which
will be determined later and define for every $(f,x)\in C_2\bwr
\Z^d$ a vector $F(f,x)=F_0(f,x)-F_0(\0,0)\in L_p$, where
\begin{equation*}\label{fix NP}
F_0(f,x)\coloneqq \sum_{y\in \Z^d} \sum_{r=0}^\infty
\frac{\max\left\{1-\frac{2r}{1+\|x-y\|_\infty},0\right\}}{1+\|x-y\|_\infty^{\gamma}}v_{y,r,f\upharpoonright_{y+[-r,r]^d}}
\end{equation*}
One checks as in the proof of Theorem~\ref{thm:polywreath} that $F$
is equivariant with respect to an appropriate action of $C_2\bwr
\Z^d$ on $L_p$. Moreover, one checks that $\|F(\1_{0},0)\|_p\lesssim
1$ and that for $x\in \{(\pm 1,0),(0,\pm 1)\}$ we have
\begin{multline}\label{eq:assume lower gamma}
\|F(0,x)\|_p^p\lesssim \sum_{y\in \Z^d}\sum_{r\in[0,
1+\|y\|_\infty/2]}\left(\frac{1+r}{(1+\|y\|_\infty)^{2+\gamma}}\right)^p
\lesssim \sum_{r=0}^\infty \sum_{\substack{k\ge 0\\k\ge2(r-1)}}
\sum_{\|y\|_\infty=k}\frac{(1+r)^p}{(1+k)^{(3+\gamma)p}}
\\\lesssim \sum_{r=1}^\infty r^p \sum_{k\ge
r}\frac{k^{d-1}}{(1+k)^{(2+\gamma)p}}\lesssim \sum_{r=0}^\infty
\frac{1}{r^{p+\gamma p-d}}<\infty,
\end{multline}
where in~\eqref{eq:assume lower gamma} we need to assume that
\begin{equation}\label{eq:assumption gamma}
\gamma>\frac{d+1-p}{p}.
\end{equation}
It follows that as long as~\eqref{eq:assumption gamma} holds true
$F$ is Lipschitz.}

{\em For the lower bound fix $(f,x)\in C_2\bwr \Z^d$ such that
$f\neq \0$ and let $R\ge 0$ be the smallest integer for which there
exists $z\in \supp(f)$ such that $\|z-x\|_\infty=R$, i.e., $R$ is
the smallest integer such that $\supp(f)\subseteq x+[-R,R]^d$. Note
that for every $y\in \Z^d$ such that $\|y-z\|_\infty \in [0,R]$ and
every $r\in [\|y-z\|_\infty,(1+R-\|y-z\|_\infty)/4]$ we have $z\in
y+[-r,r]^d$, and hence $\supp(f)\cap \left(y+[-r,r]^d\right)\neq
\emptyset$, and $\|y-x\|_\infty\ge R-\|y-z\|_\infty$, which implies
that $\frac{2r}{1+\|y-x\|_\infty}\le \frac12$. Thus:
\begin{multline}\label{eq:lower R}
\|F(f,x)\|_p^p\gtrsim \sum_{k=0}^R\sum_{\substack{y\in
\Z^d\\\|y-z\|_\infty=k}}\sum_{r\in
[k,(1+R-k)/4]}\frac{1}{\left(1+(k+R)^{\gamma}\right)^p}\\\gtrsim
\sum_{k\in[0,(1+R)/5]}\left(1+k^{d-1}\right)\cdot\frac{1+R-5k}{4}\cdot\frac{1}{\left(1+(k+R)^{\gamma}\right)^p}\gtrsim
R^{d+1-\gamma p}.
\end{multline}
Note the trivial bound:
\begin{equation}\label{eq:trivial TSP}
\TSP(\supp(f);x,x)\le \TSP\left(x+[-R,R]^d;x,x\right)\lesssim R^d.
\end{equation}
Assuming also that $\gamma<\frac{d+1}{p}$ we see that a combination
of~\eqref{eq:lower R} and~\eqref{eq:trivial TSP} implies that:
\begin{equation}\label{eq:use trivial TSP}
\|F(f,x)\|_p\gtrsim \TSP(\supp(f);x,x)^{\frac{d+1-\gamma p}{dp}}.
\end{equation}
Hence if we define $\Psi(x)=x\oplus F(x)\in \ell_p^d\oplus L_p$ we
get the lower bound
\begin{multline}\label{eq:before limit gamma}
\|\Psi(f,x)\|_p\gtrsim \|x\|_1+\TSP(\supp(f);x,x)^{\frac{d+1-\gamma
p}{dp}}\gtrsim
\left(d_{\Z^d}(x,0)+\TSP(\supp(f);x,x)\right)^{\frac{d+1-\gamma
p}{dp}}\\\gtrsim\left(d_{C_2\bwr
\Z^d}\left((f,x),(\0,0)\right)\right)^{\frac{d+1-\gamma p}{dp}}.
\end{multline}
Letting $\gamma$ tend from above to $\frac{d+1-p}{p}$
in~\eqref{eq:before limit gamma} we get the lower bound
\begin{equation}\label{eq:1/d}
\alpha_p^*\left(C_2\bwr \Z^d\right)\ge\frac{1}{d}.
\end{equation}
While~\eqref{eq:1/d} reproduces the result of~\cite{Tess06}, it
yields the sharp bound $\alpha_p^*\left(C_2\bwr \Z^d\right)=
\frac{1}{p}$ only when $p=d=2$, in which case the above embedding
coincides with the embedding used in~\cite{NP07}. This is why we
needed to use a new argument in our proof of
Theorem~\ref{thm:polywreath}. Note that if one attempts to use the
above reasoning while replacing the group $\Z^d$ by a general group
$H$ of growth rate $d$  one realizes that it used the bound
\begin{eqnarray}\label{eq;want d-1}
\left|B_H(e_H,r+1)\right|-\left|B_H(e_H,r)\right|\approx r^{d-1}.
\end{eqnarray}
Unfortunately the validity of~\eqref{eq;want d-1} is open for
general groups $H$ of growth rate $d$. To the best of our knowledge
the best known general upper bound on the growth rate of spheres is
the following fact: there exists $\beta>0$ (depending on the group
$H$ and the choice of generators) such that for every $r\in \N$ we
have
\begin{eqnarray}\label{eq:sphere}
\left|B_H(e_H,r+1)\right|-\left|B_H(e_H,r)\right|\lesssim
r^{d-\beta}.
\end{eqnarray}
This is an immediate corollary of a well known (simple) result in
metric geometry: since $\left|B_H(e,r)\right|\approx r^d$ the metric
space $(H,d_H)$ is doubling (moreover, the counting measure on $H$
is Ahlfors-David $d$-regular. See~\cite{Hei01} for a discussion of
these notions). By Lemma 3.3 in~\cite{CM98} (see also Proposition
6.12 in~\cite{Che99}) if $(X,d,\mu)$ is a geodesic doubling metric
measure space then for all $x\in X$, $r>0$ and $\delta\in (0,1)$ we
have
\begin{eqnarray}\label{eq:CM}
\mu\left(B_X(x,r)\setminus B_X(x,(1-\delta)r)\right)\le
(2\delta)^{\beta}\mu\left(B_X(x,r)\right),
\end{eqnarray}
where $\beta>0$ depends only on the doubling constant of the measure
$\mu$ (see~\cite{CM98,Che99} for a bound on $\beta$. In~\cite{NT07}
it is shown that the bound on $\beta$ from~\cite{CM98,Che99} is
asymptotically sharp as the doubling constant tends to $\infty$).
Clearly~\eqref{eq:CM} implies~\eqref{eq:sphere} if we let $\mu$ be
the counting measure on $H$ and $\delta=\frac{1}{r}$. While it is
natural to conjecture that it is possible to take $\beta=1$
in~\eqref{eq:sphere}, this has been proved when $H$ is a $2$-step
nilpotent group~\cite{Stoll98}, but it is unknown in general.
 \fin}
\end{remark}

\section{The zero section of $\Z\bwr \Z$}\label{sec:zero section}

This section is devoted to the proof of the following theorem:

\begin{thm}\label{thm:zero section}
Let $(\Z\bwr \Z)_0$ be the zero section of $\Z\bwr \Z$, i.e. the
subset of $\Z\bwr \Z$ consisting of all $(f,x)\in \Z\bwr \Z$ with
$x=0$, with the metric inherited from $\Z\bwr \Z$. Then for all
$p\in [1,2]$ we have
$$
\alpha_p^*\left((\Z\bwr \Z)_0,d_{\Z\bwr \Z}\right)=\frac{p+1}{2p}.
$$
\end{thm}

\begin{proof} The fact that $\alpha_p^*\left((\Z\bwr \Z)_0,d_{\Z\bwr\Z}\right)\le
\frac{p+1}{2p}$ follows from a variant of an argument
from~\cite{AGS06}---see Lemma 7.8 in~\cite{NP07}. We present an
alternative proof of this fact in Section~\ref{sec:zero section
walk} below.

Fix $\e\in (0,1)$. In~\cite{NP07} we have shown that there exists a
function $F_0:\Z\bwr\Z\to L_p$ such that the metric
$\|F_0(f_,x_1)-(f_2,x_2)\|_p$ is $\Z\bwr \Z$-invariant and for all
$(f,x)\in Z\bwr Z$ we have
\begin{eqnarray}\label{eq:from NP07}
|x|^{(1-\e)p}+\sum_{j\in \Z}|f(j)|^p+\max\left\{|j|^{(1-\e)p}:\
f(x+j)\neq 0\right\}\lesssim\|F_0(f,x)-F_0(\0,0)\|_p^p\lesssim
d_{Z\bwr Z}\big((f,x),(\0,0)\big)^p,
\end{eqnarray}
where here, and in what follows, the implied constants depend only
on $p$ and $\e$. We note that while~\eqref{eq:from NP07} was not
stated as a separate result in~\cite{NP07}, it is contained in the
proof of Theorem 3.3 there---see equation (28) in~\cite{NP07} with
$a=1$ and $b=1-\e$. Alternatively~\eqref{eq:from NP07} is explained
in detail for the case $p=2$ in Remark 2.2 of~\cite{ANP07}---the
same argument works when we replace in that proof $L_2$ by $L_p$ and
let $\alpha$ be arbitrarily close to $(p-1)/p$ (instead of
arbitrarily close to $1/2$).

Let $\{e_{j,k,\ell}:\ j,k,\ell\in \Z\}$ be the standard basis of
$\ell_p(\Z\times\Z\times\Z)$. For every $(f,0)\in(\Z\bwr \Z)_0$
define
$$
\Phi(f,0)=\sum_{\ell=1}^\infty \sum_{k=0}^\infty
\sum_{\substack{j\in \Z\\ |j|\in
[2^{\ell-1}-1,2^{\ell}-1)}}\frac{2^{(k+(p-1)\ell)/p}}{k+1}\exp\left(\frac{2\pi
if(j)}{2^k}\right)e_{j,k,\ell}.
$$
Our embedding of $(\Z\bwr \Z)_0$ will be
$$
F\coloneqq F_0\oplus \Phi\in \ell_p(\Z\times\Z\times\Z)\oplus L_p.
$$
Observe that for every $(f,0),(g,0)\in (\Z\bwr \Z)_0$ we have
$\|\Phi(f,0)-\Phi(g,0)\|_p=\|\Phi(f-g,0)-\Phi(\0,0)\|_p$, so it will
suffice to prove the required compression bounds for
$\|F(f,0)-F(g,0)\|_p$ when $g=0$.

From now on we shall fix $(f,0)\in (\Z\bwr\Z)_0$. For every
$\ell,m\in \Z$ denote
$$
E(\ell,m)=\left\{j:\in \Z:\ |j|\in [2^{\ell-1}-1,2^{\ell}-1)\
\wedge\  |f(j)|\in [2^m,2^{m+1})\right\}.
$$
We also write $M\coloneqq \max\{|j|:\ f(j)\neq 0\}$, so that
\begin{equation}\label{eq:formula for dist}
d_{\Z\bwr \Z}\big((f,0),(\0,0)\big)\approx M+\|f\|_1=M+\sum_{j\in
\Z} |f(j)|\approx M+\sum_{\ell=1}^\infty\sum_{m=0}^\infty
2^{m}|E(\ell,m)|.
\end{equation}
Now,
\begin{multline}\label{eq:decompose m}
\left\|\Phi(f,0)-\Phi(\0,0)\right\|_p^p=\sum_{\ell=1}^\infty
\sum_{k=0}^\infty \sum_{\substack{j\in \Z\\ |j|\in
[2^{\ell-1}-1,2^{\ell}-1)}}\frac{2^{k+(p-1)\ell}}{(k+1)^{p}}\left|1-\exp\left(\frac{2\pi
if(j)}{2^k}\right)\right|^p\\=\sum_{\ell=1}^\infty \sum_{k=0}^\infty
\frac{2^{k+(p-1)\ell}}{(k+1)^{p}} \sum_{m=0}^\infty\sum_{j\in
E(\ell,m)}\left|1-\exp\left(\frac{2\pi if(j)}{2^k}\right)\right|^p.
\end{multline}
Note that
\begin{eqnarray}\label{eq:small m case}
m\le k-2\implies \sum_{j\in E(\ell,m)}\left|1-\exp\left(\frac{2\pi
if(j)}{2^k}\right)\right|^p\approx 2^{p(m-k)}|E(\ell,m)|.
\end{eqnarray}
and for all $m,k\in \Z$,
\begin{eqnarray}\label{eq:large m case}
\sum_{j\in E(\ell,m)}\left|1-\exp\left(\frac{2\pi
if(j)}{2^k}\right)\right|^p\lesssim |E(\ell,m)|.
\end{eqnarray}
Plugging~\eqref{eq:small m case} and~\eqref{eq:large m case}
into~\eqref{eq:decompose m} we see that
\begin{multline}\label{eq:before AM-GM}
\left\|\Phi(f,0)-\Phi(\0,0)\right\|_p^p\lesssim \sum_{\ell=1}^\infty
\sum_{m=0}^\infty \left(\sum_{k=0}^{m+1}
\frac{2^{k+(p-1)\ell}}{(k+1)^{p}}|E(\ell,m)|+\sum_{k=m+2}^\infty
\frac{2^{k+(p-1)\ell}}{(k+1)^{p}}2^{p(m-k)}|E(\ell,m)|\right)\\
\lesssim \sum_{\ell=1}^\infty \sum_{m=0}^\infty
\frac{2^{m+(p-1)\ell}}{(m+1)^{p}}|E(\ell,m)|\le
\left(\sum_{\ell=1}^\infty \sum_{m=0}^\infty
\left(\frac{2^{m+(p-1)\ell}}{(m+1)^{p}}|E(\ell,m)|\right)^{1/p}\right)^p.
\end{multline}
Using the fact that for all $a,b\ge 0$ we have $ab^{p-1}\le
\left(\frac{a+b}{2}\right)^p$ we can bound the summands
in~\eqref{eq:before AM-GM} as follows:
\begin{equation}\label{eq:AM-GM}
\left(\frac{2^{m+(p-1)\ell}}{(m+1)^{p}}|E(\ell,m)|\right)^{1/p}\lesssim
\left\{\begin{array}{ll}2^m|E(\ell,m)|+\frac{2^\ell}{(m+1)^{p/(p-1)}}&
\mathrm{if}\ E(\ell,m)\neq \emptyset,\\
2^m|E(\ell,m)|& \mathrm{otherwise}.\end{array}\right.
\end{equation}
Note that if $E(\ell,m)\neq \emptyset$ then there exists $j\in \Z$
with $|j|\in [2^{\ell-1}-1,2^{\ell}-1)$ such that $f(j)\neq 0$. By
the definition of $M$ this implies that $2^{\ell}<M$. Using this
observation while substituting the the estimates~\eqref{eq:AM-GM}
in~\eqref{eq:before AM-GM} we see that
\begin{multline}\label{eq:upper Phi}
\left\|\Phi(f,0)-\Phi(\0,0)\right\|_p\lesssim \sum_{m=0}^\infty
\sum_{\ell=1}^\infty 2^m
|E(\ell,m)|+\sum_{\ell=1}^{\lfloor\log_2M\rfloor}2^\ell
\sum_{m=0}^\infty \frac{1}{(m+1)^{p/(p-1)}}\\\lesssim
\sum_{\ell=1}^\infty\sum_{m=0}^\infty 2^m
|E(\ell,m)|+M\stackrel{\eqref{eq:formula for dist}}{\approx}
d_{\Z\bwr \Z}\big((f,0),(\0,0)\big).
\end{multline}
This shows that $\|F\|_{\Lip}\lesssim 1$.

In the reverse direction write
$$
D\coloneqq d_{\Z\bwr
\Z}\big((f,0),(\0,0)\big)\stackrel{\eqref{eq:formula for
dist}}{\approx} M +\sum_{2^\ell<M}\sum_{|j|\in
[2^{\ell-1}-1,2^{\ell}-1)} |f(j)|\approx
\sum_{2^\ell<M}\left(2^\ell+\sum_{|j|\in [2^{\ell-1}-1,2^{\ell}-1)}
|f(j)|\right)
$$
It follows that there exists an integer $\ell\lesssim \log M$ such
that
\begin{eqnarray}\label{eq:find scale 1}
D\lesssim \log(M+1) \cdot\left(2^\ell+\sum_{|j|\in
[2^{\ell-1}-1,2^{\ell}-1)} |f(j)|\right).
\end{eqnarray}
We shall fix this $\ell$ from now on. Observe that
\begin{eqnarray*}\label{eq:find scale 2}
 \sum_{|j|\in [2^{\ell-1}-1,2^{\ell}-1)} |f(j)|\approx
\sum_{2^{m+1}<\|f\|_1} 2^m|E(\ell,m)|.
\end{eqnarray*}
Hence there exists an integer $m\lesssim \log(1+\|f\|_1)$ such that
\begin{eqnarray}\label{eq:find scale 3}
 \sum_{|j|\in [2^{\ell-1}-1,2^{\ell}-1)} |f(j)| \lesssim
 2^m|E(\ell,m)|\cdot\log(1+\|f\|_1).
\end{eqnarray}
We shall fix this $m$ form now on. Combining~\eqref{eq:find scale 1}
with~\eqref{eq:find scale 3} yields the bound:
\begin{equation}\label{eq:log dist}
D\lesssim \log(M+1)
\cdot\left(2^\ell+2^m|E(\ell,m)|\cdot\log(1+\|f\|_1)\right)
\stackrel{\eqref{eq:formula for dist}}{\lesssim}
\left(\log(D+1)\right)^2\cdot \left(2^\ell+2^m|E(\ell,m)|\right).
\end{equation}

Substitute~\eqref{eq:small m case} into~\eqref{eq:decompose m} to
get the lower bound
$$
\left\|\Phi(f,0)-\Phi(\0,0)\right\|_p^p\gtrsim
 \sum_{k=m+2}^\infty
\frac{2^{k+(p-1)\ell}}{(k+1)^p}2^{p(m-k)}|E(\ell,m)|\gtrsim
\frac{2^{m+(p-1)\ell}}{(m+1)^p}|E(\ell,m)|\gtrsim
\frac{2^{m+(p-1)\ell}}{(\log(D+1))^p}|E(\ell,m)|.
$$
Also~\eqref{eq:from NP07} implies that
$$
\left\|F_0(f,0)-F_0(\0,0)\right\|_p^p\gtrsim
M^{(1-\e)p}+2^{mp}|E(\ell,m)|\gtrsim 2^{(1-\e)\ell
p}+2^{mp}|E(\ell,m)|.
$$
Thus \begin{multline}\label{total lower bound at scale}
\left\|F(f,0)-F(\0,0)\right\|_p\gtrsim
2^{(1-\e)\ell}+2^m|E(\ell,m)|^{1/p} + \frac{1}{\log(D+1)}\cdot
2^{m/p}|E(\ell,m)|^{1/p}2^{\ell(p-1)/p}\\ \gtrsim \frac{1}{\log
(D+1)}\cdot
\left(2^{\ell}+2^{m}|E(\ell,m)|^{1/p}+2^{m/p}|E(\ell,m)|^{1/p}2^{\ell(p-1)/p}\right)^{1-\e}.
\end{multline}
We claim that
\begin{eqnarray}\label{eq:strange AM-GM}
2^{\ell}+2^{m}|E(\ell,m)|^{1/p}+2^{m/p}|E(\ell,m)|^{1/p}2^{\ell(p-1)/p}\ge
\frac{2^{\ell}+\left(2^m|E(\ell,m)|\right)^{\frac{p+1}{2p}}}{2}.
\end{eqnarray}
Indeed, if $\left(2^m|E(\ell,m)|\right)^{\frac{p+1}{2p}}\le 2^\ell$
then~\eqref{eq:strange AM-GM} is trivial, so assume that
$a\coloneqq\left(2^m|E(\ell,m)|\right)^{\frac{p+1}{2p}}\ge 2^\ell$.
Since $|E(\ell,m)|= 2^{-m}\cdot a^{2p/(p+1)}$ we see that
\begin{eqnarray}\label{eq:numerical}
2^{\ell}+2^{m}|E(\ell,m)|^{1/p}+2^{m/p}|E(\ell,m)|^{1/p}2^{\ell(p-1)/p}\ge
2^{(p-1)m/p}a^{2/(p+1)}+a^{2/(p+1)}2^{\ell(p-1)/p}.
\end{eqnarray}
Note that by definition $2^{-m}\cdot a^{2p/(p+1)}=|E(\ell,m)|\le
2^{\ell}$, so $2^{m}\ge 2^{-\ell}\cdot a^{2p/(p+1)}$. Substituting
this bound into~\eqref{eq:numerical} we see that
\begin{multline*}
2^{\ell}+2^{m}|E(\ell,m)|^{1/p}+2^{m/p}|E(\ell,m)|^{1/p}2^{\ell(p-1)/p}\ge
2^{-\ell(p-1)/p}\cdot a^{2p/(p+1)}+a^{2/(p+1)}2^{\ell(p-1)/p}\\\ge
2a=2\left(2^m|E(\ell,m)|\right)^{\frac{p+1}{2p}},
\end{multline*}
where we used the arithmetic mean/geometric mean inequality. This
completes the proof of~\eqref{eq:strange AM-GM}.

A combination of~\eqref{eq:log dist}, \eqref{total lower bound at
scale} and~\eqref{eq:strange AM-GM} yields
\begin{multline*}
\left\|F(f,0)-F(\0,0)\right\|_p\gtrsim \frac{1}{\log(D+1)}\cdot
\left(2^{\ell}+\left(2^m|E(\ell,m)|\right)^{\frac{p+1}{2p}}\right)^{1-\e}\\\gtrsim
\frac{1}{\log(D+1)}\cdot
\left(2^{\ell}+2^m|E(\ell,m)|\right)^{(1-\e)\frac{p+1}{2p}}\gtrsim
\frac{D^{(1-\e)\frac{p+1}{2p}}}{(\log(D+1))^{1+2(1-\e)\frac{p+1}{2p}}}\gtrsim
D^{(1-2\e)\frac{p+1}{2p}}.
\end{multline*}
This completes the proof of Theorem~\ref{thm:zero section}.
\end{proof}

\section{General compression upper bounds for amenable
groups}\label{sec:amenable beta}

Let $\Gamma$ be a group which is generated by the finite symmetric
set $S\subseteq \Gamma$. Let $\rho$ be a left-invariant metric on
$\Gamma$ such that $B_\rho(e_\Gamma,r)=\{x\in\Gamma:\ \rho(x,e)\le
r\}$ is finite for all $r\ge 0$. In most of our applications of the
ensuing arguments the metric $\rho$ will be the word metric induced
by $S$, but we will also need to deal with other invariant metrics
(see Section~\ref{sec:zero section walk}).

 Given a symmetric probability measure $\mu$
on $\Gamma$ let $\{g_k\}_{k=1}^\infty$ be i.i.d. elements of
$\Gamma$ which are distributed according to $\mu$. The $\mu$-random
walk $\{W_t^\mu\}_{t=0}^\infty$ is defined as $W_0^\mu=e_\Gamma$ and
$W_t^\mu=g_1g_2\cdots g_t$ for $t\in \N$. Fix $p\ge 1$ and assume
that \begin{equation}\label{eq:integrability}\int_\Gamma
\rho(x,e_\Gamma)^pd\mu(x)=\E_{\mu}\left[\rho\left(W_1^\mu,e_\Gamma\right)^p\right]<\infty.\end{equation}

Let $\{\mu_t\}_{t=1}^\infty$ be a sequence of symmetric probability
measures satisfying the integrability
condition~\eqref{eq:integrability} and define
\begin{eqnarray}\label{eq:def betap}
\beta_p^*\left(\{\mu_t\}_{t=1}^\infty,\rho\right)\coloneqq
\limsup_{t\to\infty}\frac{\log\left(\E_{\mu_t}\left[\rho\left(W_t^{\mu_t},e_\Gamma\right)\right]\right)}
{\log\left(t \E_{\mu_t}
\left[\rho\left(W_1^{\mu_t},e_\Gamma\right)^p\right]\right)}.
\end{eqnarray}
Finally we let $\beta_p^*(\Gamma,\rho)$ be the supremum of
$\beta_p^*\left(\{\mu_t\}_{t=1}^\infty,\rho\right)$ over all
sequences of symmetric probability measures $\{\mu_t\}_{t=1}^\infty$
on $\Gamma$ satisfying
\begin{equation}\label{eq:condition}
\forall t\in \N\ \ \int_\Gamma
\rho(x,e_\Gamma)^pd\mu_t(x)<\infty\quad \mathrm{and}\quad
\lim_{t\to\infty} \big( t
\mu_t\left(\Gamma\setminus\{e_\Gamma\}\right)\big)=\infty.
\end{equation}
When $\rho$ is the word metric induced by the symmetric generating
set $S$ we will use the simplified notation
$\beta_p^*(\Gamma,\rho)=\beta_p^*(\Gamma)$. This convention does not
create any ambiguity since clearly $\beta_p^*(\Gamma,\rho)$ does not
depend on the choice of the finite symmetric generating set $S$
(this follows from the fact that due to~\eqref{eq:condition} the
denominator in~\eqref{eq:def betap} tends to $\infty$ with $t$---we
establish this fact below).

To better explain the definition~\eqref{eq:def betap} we shall make
some preliminary observations before passing to the main results of
this section. We first note that
\begin{equation}\label{eq:beta<1}
\beta_p^*(\Gamma,\rho)\le 1. \end{equation} Indeed, since we are
assuming that all the $\rho$-balls are finite there exists
$\rho_0>0$ such that for every distinct $x,y\in \Gamma$ we have
$\rho(x,y)\ge \rho_0$. Hence for every symmetric probability measure
$\mu$ on $\Gamma$ which satisfies~\eqref{eq:integrability} we have
\begin{equation}\label{eq:zero prob}
\E_{\mu} \left[\rho\left(W_1^{\mu},e_\Gamma\right)^p\right]\ge
\rho_0^p\mu\left(\Gamma\setminus\{e_\Gamma\}\right).
\end{equation}
H\"older's inequality therefore implies that:
\begin{multline}\label{eq:holder}
\E_{\mu}\left[\rho\left(W_1^{\mu},e_\Gamma\right)\right]=
\E_{\mu}\left[\rho\left(W_1^{\mu},e_\Gamma\right)\1_{\Gamma\setminus\{e_\gamma\}}\right]\\\le
\mu\left(\Gamma\setminus\{e_\Gamma\}\right)^{(p-1)/p}
\left(\E_\mu\left[\rho\left(W_1^\mu,e_\Gamma\right)^p\right]\right)^{1/p}\stackrel{\eqref{eq:zero
prob}}{\le}\frac{1}{\rho_0^{p-1}}\E_\mu\left[\rho\left(W_1^\mu,e_\Gamma\right)^p\right].
\end{multline}
On the other hand, by the triangle inequality we have:
\begin{equation}\label{eq:trianlge 1 bound}
\E_\mu\left[\rho\left(W_t^\mu,e_\Gamma\right)\right]\le \sum_{i=1}^t
\E_\mu\left[\rho\left(W_i^\mu,W_{i-1}^\mu\right)\right]=t\E_\mu\left[\rho\left(W_1^\mu,e_\Gamma\right)\right]
\stackrel{\eqref{eq:holder}}{\le}\frac{t}{\rho_0^{p-1}}\E_\mu\left[\rho\left(W_1^\mu,e_\Gamma\right)^p\right].
\end{equation}
It follows that if $\{\mu_t\}_{t=1}^\infty$ are symmetric
probability measures on $\Gamma$ satisfying~\eqref{eq:condition}
then
\begin{equation*}\label{eq:limsup}
\limsup_{t\to\infty}\frac{\log\left(\E_{\mu_t}\left[\rho\left(W_t^{\mu_t},e_\Gamma\right)\right]\right)}
{\log\left(t \E_{\mu_t}
\left[\rho\left(W_1^{\mu_t},e_\Gamma\right)^p\right]\right)}\stackrel{\eqref{eq:zero
prob}\wedge\eqref{eq:trianlge 1
bound}}{\le}\limsup_{t\to\infty}\left(
1-\frac{(p-1)\log\rho_0}{\log\left(t\mu_t\left(\Gamma\setminus\{e_\Gamma\}\right)\right)+p\log\rho_0}\right)
\stackrel{\eqref{eq:condition}}{=}1,
\end{equation*}
implying~\eqref{eq:beta<1}.

We also claim that if $1\le q\le p<\infty$ then
\begin{equation}\label{eq:monotonepq}
\beta_p^*(\Gamma,\rho)\le \beta_q^*(\Gamma,\rho).
\end{equation}
 Indeed, let
$\{\mu_t\}_{t=1}^\infty$ be symmetric probability measures on
$\Gamma$ satisfying~\eqref{eq:condition} and note that
\begin{equation}\label{eq:pq}
\E_{\mu_t} \left[\rho\left(W_1^{\mu_t},e_\Gamma\right)^q\right]\le
\mu_t(\Gamma\setminus\{e_\Gamma\})^{(p-q)/p}\left(\E_{\mu_t}
\left[\rho\left(W_1^{\mu_t},e_\Gamma\right)^p\right]\right)^{q/p}\stackrel{\eqref{eq:zero
prob}}{\le} \frac{1}{\rho_0^{p-q}}\E_{\mu_t}
\left[\rho\left(W_1^{\mu_t},e_\Gamma\right)^p\right].
\end{equation}
Hence,
\begin{multline*}
\limsup_{t\to\infty}\frac{\log\left(\E_{\mu_t}\left[\rho\left(W_t^{\mu_t},e_\Gamma\right)\right]\right)}
{\log\left(t \E_{\mu_t}
\left[\rho\left(W_1^{\mu_t},e_\Gamma\right)^p\right]\right)}\stackrel{\eqref{eq:pq}}{\le}\limsup_{t\to\infty}\frac{\log\left(\E_{\mu_t}\left[\rho\left(W_t^{\mu_t},e_\Gamma\right)\right]\right)}
{\log\left(t \E_{\mu_t}
\left[\rho\left(W_1^{\mu_t},e_\Gamma\right)^q\right]\right)}\cdot\frac{1}{1+\frac{(p-q)\log\rho_0}{\log\left(t
\E_{\mu_t}
\left[\rho\left(W_1^{\mu_t},e_\Gamma\right)^q\right]\right)}}\\
\stackrel{\eqref{eq:zero prob}}{\le}
\limsup_{t\to\infty}\frac{\log\left(\E_{\mu_t}\left[\rho\left(W_t^{\mu_t},e_\Gamma\right)\right]\right)}
{\log\left(t \E_{\mu_t}
\left[\rho\left(W_1^{\mu_t},e_\Gamma\right)^q\right]\right)}\cdot\frac{1}{1+\frac{(p-q)\log\rho_0}{\log\left(t
\mu_t(\Gamma\setminus\{e_\Gamma\})\right)+q\log\rho_0}}
\stackrel{\eqref{eq:condition}}{\le}\beta_q^*(\Gamma,\rho),
\end{multline*}
implying~\eqref{eq:monotonepq}.

The main result of this section is the following theorem:

\begin{thm}\label{thm:beta p}
Assume that $\Gamma$ is amenable and that $X$ is a metric space with
Markov type $p$. Then for every left-invariant metric $\rho$ on
$\Gamma$ such that $|B_\rho(e_\Gamma,r)|<\infty$ for all $r\ge 0$ we
have:
$$
\alpha_X^*(\Gamma,\rho)\le \frac{1}{p\beta_p^*(\Gamma,\rho)}.
$$
\end{thm}

\begin{remark} {\em In~\cite{ANP07,NP07} it was essentially
shown that the bound in Theorem~\ref{thm:beta p} holds true with
$\beta_p^*(\Gamma,\rho)$ replaced by $\beta_\infty^*(\Gamma,\rho)$,
which is a weaker bound due to~\eqref{eq:monotonepq}. More
precisely~\cite{ANP07,NP07} dealt with the case when all the
measures $\mu_t$ equal a fixed measure $\mu$, in which case the
second requirement of~\eqref{eq:condition} is simply that $\mu$ is
not supported on $\{e_\Gamma\}$. If we restrict to this particular
case we can define an analogous parameter by
$$
\widetilde\beta_p^*(\mu,\rho)=
\beta_p^*\left(\{\mu,\mu,\mu,\ldots\},\rho\right)\coloneqq
\limsup_{t\to\infty}\frac{\log\left(\E_{\mu}\left[\rho\left(W_t^{\mu},e_\Gamma\right)\right]\right)}{\log
t}.
$$
and similarly by taking the supremum over all symmetric probability
measures measures $\mu$ satisfying~\eqref{eq:condition} we can
define the parameter $\widetilde\beta_p^*(\Gamma,\rho)$. An
inspection of the results in~\cite{ANP07,NP07} shows that a variant
of Theorem~\ref{thm:beta p} is established there with
$\beta_p^*(\Gamma,\rho)$ replaced by
$\widetilde\beta_\infty^*(\Gamma,\rho)$. Thus Theorem~\ref{thm:beta
p} is formally stronger than the results of~\cite{ANP07,NP07}. As we
shall see in Section~\ref{sec:Lp comp}, this is a strict improvement
which is crucial for our proof of the bound $\alpha_p^*(\Z\bwr
\Z)\le \frac{p}{2p-1}$, and in Section~\ref{sec:zero section walk}
we will also need to use a family of non-identical measures
$\{\mu_t\}_{t=1}^\infty$.\fin}
\end{remark}

\begin{proof}[Proof of Theorem~\ref{thm:beta p}]
 Let
$\{F_n\}_{n=0}^\infty$ be a F\o lner sequence for $\Gamma$, i.e.,
for every $\e>0$ and any finite $K\subseteq \Gamma$, we have
$|F_n\triangle (F_nK)|\le \e|F_n|$ for large enough $n$. Fix $\beta<
\beta_p^*(\Gamma,\rho)$. Then there exists a sequence of symmetric
probability measures $\{\mu_t\}_{t=1}^\infty$ on $\Gamma$ which
satisfy~\eqref{eq:condition} and
$\beta<\beta_p^*\left(\{\mu_t\}_{t=1}^\infty,\rho\right)$. This
implies that there exists an increasing sequence of integers
$\{t_k\}_{k=1}^\infty$ for which
$$
\E_{\mu_{t_k}}\left[\rho\left(W_{t_k}^{\mu_{t_k}},e_\Gamma\right)\right]\ge
t_k^{\beta }\left(
\E_{\mu_{t_k}}\left[\rho\left(W_1^{\mu_{t_k}},e_\Gamma\right)^p\right]\right)^{\beta},
$$
for all $k$. For every $t,r\in \N$ consider the event
$$
\Lambda_t(r)\coloneqq \bigcap_{j=1}^t \left\{W_j^{\mu_t}\in
B_\rho(e_\Gamma,r)\right\}.
$$
By the monotone convergence theorem for every $k\in \N$ there exists
$r_k\in \N$ such that
\begin{eqnarray}\label{eq:monotone}
\E_{\mu_{t_k}}\left[\rho\left(W_{t_k}^{\mu_{t_k}},e_\Gamma\right)\1_{\Lambda_{t_k}(r_k)}\right]\ge
\frac12 t_k^{\beta }\left(
\E_{\mu_{t_k}}\left[\rho\left(W_1^{\mu_{t_k}},e_\Gamma\right)^p\right]\right)^{\beta}.
\end{eqnarray}
Since $|B_\rho(e_\Gamma,r_k)|<\infty$ for every $k\in \N$ we can
find $n_k\in \N$ such that if we denote $A\coloneqq
F_{n_k}B_\rho(e_\Gamma,r_k)\supseteq F_{n_k}$ then we have (by the
F\o lner condition with $\e=1$),
\begin{eqnarray}\label{eq:use folner}
|A\setminus F_{n_k}|\le |F_{n_k}|\implies |F_{n_k}|\ge \frac12 |A|.
\end{eqnarray}

Fix $k\in \N$ and let $\{g_i\}_{i=1}^\infty\subseteq \Gamma$ be
i.i.d. group elements distributed according to $\mu_{t_k}$ such that
$W_t^{\mu_{t_k}}=g_1g_2\cdots g_t$ for every $t\in \N$. Let $Z_0$ be
uniformly distributed over $A$ and independent of
$\{g_i\}_{i=1}^\infty$. For $t\in \N$ define
$$
Z_t\coloneqq \left\{\begin{array}{ll}  Z_{t-1}g_t &\mathrm{if\ } Z_{t-1}g_t\in A,\\
Z_{t-1} & \mathrm{otherwise}.
\end{array}\right.
$$
Consider the event $\Omega\coloneqq \{Z_0\in F_{n_k}\}\cap
\Lambda_{t_{n_k}}(r_k)$. By construction when $\Omega$ occurs we
have $Z_{t_k}=Z_0W_{t_k}^{\mu_{t_k}}$. Hence
\begin{multline}\label{eq:in A}
\E_{\mu_{t_k}}\left[\rho\left(Z_{t_k},Z_0\right)\right]\ge
\E_{\mu_{t_k}}\left[\rho\left(Z_0W_{t_k}^{\mu_{t_k}},Z_0\right)\1_\Omega\right]\\\stackrel{(*)}{=}
\Pr\left[Z_0\in
F_{n_k}\right]\cdot\E_{\mu_{t_k}}\left[\rho\left(W_{t_k}^{\mu_{t_k}},e_\Gamma\right)\1_{\Lambda_{t_k}(r_k)}\right]
\stackrel{\eqref{eq:monotone}\wedge \eqref{eq:use folner}}{\ge}
\frac14 t_k^{\beta }
\left(\E_{\mu_{t_k}}\left[\rho\left(W_1^{\mu_{t_k}},e_\Gamma\right)^p\right]\right)^{\beta},
\end{multline}
where in $(*)$ we used the independence of $Z_0$ and
$\{g_i\}_{i=1}^\infty$ and the left-invariance of $\rho$.

 On the other hand fix $\alpha\in (0,1)$ and assume that there exists an embedding
 $f:\Gamma\to X$ and $c,C\in (0,\infty)$ such that
 \begin{eqnarray}\label{eq:compression assumption}
x,y\in \Gamma \implies c\rho(x,y)^\alpha\le d_X(f(x),f(x))\le
C\rho(x,y).
 \end{eqnarray}
Our goal is to show that $\alpha\le \frac{1}{p\beta}$. Since
$\beta<1$ this inequality is vacuous if $p\alpha<1$. We may
therefore assume that $p\alpha\ge 1$. Since $\{Z_t\}_{t=0}^\infty$
is a stationary reversible Markov chain, for every $M> M_p(X)$ and
$k\in \N$ we have
\begin{multline}\label{eq:upper markov type}
\E\left[d_X(f(Z_{t_k}),f(Z_0))^p\right]\le M^pt_k
\E\left[d_X(f(Z_1),f(Z_0))^p\right]\\\stackrel{\eqref{eq:compression
assumption}}{\le} M^pC^pt_k
\E\left[\rho(Z_1,Z_0)^p\right]\stackrel{(**)}{\le} M^pC^pt_k
\E_{\mu_{t_k}}\left[\rho\left(W_1^{\mu_{t_k}},e_\Gamma\right)^p\right],
\end{multline}
Where in $(**)$ we used the point-wise inequality $\rho(Z_1,Z_0)\le
\rho(g_1,e_\Gamma)=\rho\left(W_1^{\mu_{t_k}},e_\Gamma\right)$. On
the other hand,
\begin{multline}\label{eq:lower markov type}
\E\left[d_X(f(Z_{t_k}),f(Z_0))^p\right]\stackrel{\eqref{eq:compression
assumption}}{\ge} c^p \E\left[\rho\left(Z_{t_k},fZ_0\right)^{\alpha
p}\right]\\\stackrel{\alpha p\ge 1}{\ge}
c^p\left(\E\left[\rho(Z_{t_k},Z_0)\right]\right)^{\alpha
p}\stackrel{\eqref{eq:in A}}{\ge} \frac{c^pt_k^{\alpha\beta
p}}{4^{\alpha
p}}\left(\E_{\mu_{t_k}}\left[\rho\left(W_1^{\mu_{t_k}},e_\Gamma\right)^p\right]\right)^{\alpha\beta
p}.
\end{multline}
Combining~\eqref{eq:upper markov type} and~\eqref{eq:lower markov
type} we deduce that
\begin{eqnarray}\label{eq:punchline}
\left(\rho_0^pt_k\mu_{t_k}(\Gamma\setminus\{e_\Gamma\})\right)^{\alpha\beta
p-1}\stackrel{\eqref{eq:zero
prob}}{\le}\left(t_k\E_{\mu_{t_k}}\left[\rho\left(W_1^{\mu_{t_k}},e_\Gamma\right)^p\right]\right)^{\alpha\beta
p-1}\le \frac{4^{\alpha p}M^pC^p}{c^p}.
\end{eqnarray}
Taking $k\to \infty$ in~\eqref{eq:punchline} while using the
assumption~\eqref{eq:condition} we conclude that $\alpha \beta p\le
1$, as required.
\end{proof}

The following theorem is a variant of Theorem~\ref{thm:beta p} which
deals with equivariant embeddings of general groups (not necessarily
amenable) into uniformly smooth Banach spaces. Its proof is an
obvious modification of the proof of Theorem 2.1 in~\cite{NP07}: one
just has to notice that in that proof the i.i.d. group elements
$\{\sigma_k\}_{k=1}^\infty$ need not be uniformly distributed over a
symmetric generating set $S\subseteq \Gamma$---the argument goes
through identically if they are allowed to be distributed according
to any symmetric probability measure $\mu$ satisfying the
integrability condition~\eqref{eq:integrability}.

\begin{thm}\label{thm:sharp} Let $\Gamma$ be a  group and $\rho$ a left-invariant metric on
$\Gamma$ such that $|B_\rho(e_\Gamma,r)|<\infty$ for all $r\ge 0$.
Assume that $X$ is a Banach space whose modulus of uniform
smoothness has power-type $p\in [1,2]$. Then:
$$
\alpha_X^\#(\Gamma,\rho)\le \frac{1}{p\beta_p^*(\Gamma,\rho)}.
$$
\end{thm}

By the results of Section~\ref{sec:equivariant}
Theorem~\ref{thm:sharp} implies Theorem~\ref{thm:beta p} when $X$ is
a Banach space whose modulus of uniform smoothness has power-type
$p$ rather than a general metric space with Markov type $p$. Note
that the former assumption implies the latter assumption as shown
in~\cite{NPSS06}.

\section{Stable walks and the $L_p$ compression of $\Z\bwr
\Z$}\label{sec:Lp comp}

This section is devoted to the proof of the following theorem:

\begin{thm}\label{thm:betap ZZ}
For every $p\in(1,2)$ we have
$$
\beta_p^*(\Z\bwr \Z)=\frac{2p-1}{p^2}.
$$
\end{thm}
 Note that since in~\cite{NP07} we proved that $\alpha_p^*(\Z\bwr
 \Z)\ge \frac{p}{2p-1}$, Theorem~\ref{thm:beta p} implies that
 $\beta_p^*(\Z\bwr \Z)\le \frac{2p-1}{p^2}$. Thus in order to prove
 Theorem~\ref{thm:betap ZZ} it suffices to show that  $\beta_p^*(\Z\bwr \Z)\ge
 \frac{2p-1}{p^2}$, which would also imply that $\alpha_p^*(\Z\bwr
 \Z)=\frac{p}{2p-1}$. In order to establish this lower bound on $\beta_p^*(\Z\bwr \Z)$ we
 will analyze certain symmetric random walks on $\Z\bwr
 \Z$ which arise from discrete approximations of $q$-stable random
 variables for some $q\in (p,2)$.

\subsection{Some general properties of symmetric walks on $\Z$}

Let $X$ be a $\Z$-valued symmetric random variable and let
$X_1,X_2,\ldots$ be i.i.d. copies of $X$. For each $n\in \N$ define
$S_n=X_1+\cdots+X_n$ (and set $S_0=0$). We also define $S_{[o,n]}$
to be the random set $\{S_0,\ldots,S_n\}$. We will record here for
future use some general properties of the walk $S_n$. These are
simple facts which appeared in various guises in the literate
(though we did not manage to pinpoint cleanly stated references for
them). We include this brief discussion for the sake of
completeness.

\begin{lem}
For $S_n$ as above we have
\begin{equation}\label{eq:size-range}
\E \left[|S_n|\right]\ge \frac14
\E\left[\left|S_{[0,n]}\right|\right].
\end{equation}
\end{lem}

\begin{proof} Fix $R\ge 0$ and denote $\tau\coloneqq \min\{t\ge 0:\
|S_t|\ge R$\}. Note the following inclusion of events:
$$
\left\{|S_n|\ge R\right\}\supseteq \left\{\tau \le n\ \wedge \
 \mathrm{\bf sign}\left(\sum_{k=\tau+1}^nX_k\right)=\mathrm{\bf
 sign}(S_\tau)\right\}.
$$
It follows that:
\begin{multline}\label{eq:size tau}
\Pr\left[|S_n|\ge R\right]\ge \sum_{m=0}^n \Pr\left[\tau=m\  \wedge
\
 \mathrm{\bf sign}\left(\sum_{k=m+1}^nX_k\right)=\mathrm{\bf
 sign}(S_m)\right]=\sum_{m=0}^n \Pr\left[\tau=m\right]\cdot
 \Pr\left[S_{n-m}\ge 0\right]\\\stackrel{(\star)}{\ge} \frac12 \sum_{m=0}^n
 \Pr\left[\tau=m\right]=\frac12 \Pr\left[\tau\le n\right],
\end{multline}
where in $(\star)$ we used the symmetry of $S_{n-m}$. Note that if
$\left|S_{[0,n]}\right|\ge 2R$ then one of the numbers
$\{|S_0|,\ldots,|S_{n}|\}$ must be at least $R$. Thus
\begin{equation}\label{eq:range tau}
\Pr \left[\tau\le n\right]\ge \Pr\left[\left|S_{[0,n]}\right|\ge
2R\right].
\end{equation}
It follows that:
\begin{multline*}
\E\left[|S_n|\right]=\sum_{R=0}^\infty \Pr\left[|S_n|\ge
R\right]\stackrel{\eqref{eq:size tau}\wedge\eqref{eq:range
tau}}{\ge} \frac12\sum_{R=0}^\infty
\Pr\left[\left|S_{[0,n]}\right|\ge 2R\right]\\\ge
\frac12\sum_{R=0}^\infty \frac{\Pr\left[\left|S_{[0,n]}\right|\ge
2R\right]+\Pr\left[\left|S_{[0,n]}\right|\ge 2R+1\right]}{2}=\frac14
\E\left[\left|S_{[0,n]}\right|\right],
\end{multline*}
as required.
\end{proof}

The proof of the following lemma is a slight variant of the argument
used to prove the first assertion of Lemma 6.3 in~\cite{NP07}.

\begin{lem}\label{lem:second moment}
Let $S_n$ be as above and denote $R_n\coloneqq \left|\left\{k\in
\{0,\ldots,n\}:\ S_k=0\right\}\right|$. Then
\begin{eqnarray}\label{eq:second moment} \Pr\left[R_n\ge \frac12
\sum_{\ell=0}^n \Pr[S_\ell=0]\right]\ge \frac18.
\end{eqnarray}
\end{lem}

\begin{proof}
Since $R_n=\sum_{\ell=0}^n \1_{\{S_\ell=0\}}$ we have $\E[R_n]=
\sum_{\ell=0}^n \Pr[S_\ell=0]$ and:
\begin{multline*}
\E\left[R_n^2\right]=\sum_{\ell=0}^n
\Pr[S_\ell=0]+2\sum_{\substack{i,j\in \{0,\ldots,n\}\\i<
j}}\Pr\left[S_i=S_j=0\right]\\=\E\left[R_n\right]+2\sum
_{\substack{i,j\in \{0,\ldots,n\}\\i<
j}}\Pr\left[S_i=0\right]\cdot\Pr\left[S_{j-i}=0\right]\le
\E\left[R_n\right]+\left(\E\left[R_n\right]\right)^2\le
2\left(\E\left[R_n\right]\right)^2.
\end{multline*}
Since for every nonnegative random variable $Z$ we have
$\Pr\left[Z\ge \frac12\E[Z]\right]\ge \frac14
\frac{(\E[Z])^2}{\E[Z^2]}$ (which is an easy consequence of the
Cauchy-Schwartz inequality---see~\cite{PZ32,AS00}) we deduce that
$\Pr\left[R_n\ge \frac12\E[R_n]\right]\ge \frac18$, as required.
\end{proof}

The proof of the following lemma is a slight variant of the argument
used to prove the second assertion of Lemma 6.3 in~\cite{NP07}.

\begin{lem} For $S_n$ as above we have:
\begin{equation}\label{eq:range}
\E\left[\left|S_{[0,n]}\right|\right]\ge \frac{n+1}{2\sum_{\ell=0}^n
\Pr\left[S_\ell=0\right]}.
\end{equation}
\end{lem}

\begin{proof} Fix $k\in \{1,\ldots,n+1\}$ and denote $\widetilde k\coloneqq \min\left\{k,\left|S_{[0,n]}\right|\right\}$.
Let $V_1,\ldots,V_{\widetilde k}$ be the first distinct $\widetilde
k$ integers that were visited by the walk $S_0,S_1,\ldots,S_n$. For
simplicity of notation we also set $V_j=n+1$ when $j\in\{\widetilde
k+1,\ldots,n\}$. Write
%$\tau_j\coloneqq\min\{0\le \tau\le n:\
%S_\tau=V_j\}$
\begin{equation*}
\tau_j\coloneqq\left\{\begin{array}{ll}\min\{0\le \tau\le n:\
S_\tau=V_j\}& j\le \widetilde k,\\
n+1& j>\widetilde k.\end{array}\right. \end{equation*} Denote
$Y_k\coloneqq \left|\left\{0\le j\le n:\ S_j\in
\{V_1,\ldots,V_{\widetilde k}\}\right\}\right|$. Then
\begin{multline}\label{eq:like IMRN1}
\E\left[Y_k\right]=\sum_{j=1}^k\E\left[\left|\left\{0\le \ell\le n:\
S_\ell=V_j\right\}\right|\right]=
\sum_{j=1}^k\E\left[\sum_{\ell=0}^n\1_{\{S_\ell=V_j\}}\right]=\sum_{j=1}^k\E\left[\left.\sum_{\ell=\tau_j}^n
\Pr\left[S_\ell=S_{\tau_j}\right]\right|\tau_j\right]\\
= \sum_{j=1}^k\E\left[\left.\sum_{\ell=\tau_j}^n
\Pr\left[S_{\ell-\tau_j}=0\right]\right|\tau_j\right]\le
k\sum_{\ell=0}^n \Pr\left[S_\ell=0\right].
\end{multline}
Hence
\begin{equation}\label{eq:like IMRN2}
\Pr\left[\left|S_{[0,n]}\right|\le k \right]\le \Pr\left[Y_k\ge
n+1\right]\le \frac{\E[Y_k]}{n+1}\stackrel{\eqref{eq:like
IMRN1}}{\le} \frac{k}{n+1}\sum_{\ell=0}^n \Pr\left[S_\ell=0\right].
\end{equation}
It follows that if we denote $m=\frac{n+1}{\sum_{\ell=0}^n
\Pr\left[S_\ell=0\right]}$ then
\begin{multline*}
\E\left[\left|S_{[0,n]}\right|\right]=\sum_{k=1}^{n+1}\Pr\left[\left|S_{[0,n]}\right|\ge
k \right]=\sum_{k=1}^{n+1}\left(1-\Pr\left[\left|S_{[0,n]}\right|\le
k-1 \right]\right)\stackrel{\eqref{eq:like IMRN2}}{\ge}
\sum_{k=1}^{\lceil m\rceil}\left(1-\frac{k-1}{m}\right)\\=\lceil
m\rceil-\frac{\lceil m\rceil(\lceil m\rceil-1)}{2m}\ge\frac{\lceil
m\rceil}{2}\ge \frac{n+1}{2\sum_{\ell=0}^n
\Pr\left[S_\ell=0\right]},
\end{multline*}
as required.
\end{proof}

\subsection{An analysis of a particular discrete stable walk on
$\Z$}\label{sec:binomial}

In this section we will analyze a specific random walk on $\Z$ which
will be used in estimating $\beta_p^*(\Z\bwr \Z)$. Similar bounds
are known to hold in great generality for arbitrary walks which are
in the domain of attraction of $q$-stable random variables, and not
only for the walk presented below. Specifically, such general
results can be deduced from Gendenko's local central limit theorem
for convergence to stable laws (see Theorem 4.2.1 in~\cite{IL71}),
in combination with some estimates on such walks from~\cite{Fel66}
(see section IX.8, Theorem 1 there). Since for the purpose of
proving compression bounds all we need is to construct one such
walk, we opted for the sake of concreteness to present here a simple
self-contained proof of the required properties of a particular walk
which is perfectly suited for the purpose of our applications to
embedding theory.

In what follows fix $q\in (p,2)$. Define $a_{1}=a_{-1}=0$ and for
$n\in (\mathbb \N\setminus \{1\})\cup\{0\}$,
\begin{equation}\label{eq:def an}
a_{n}=a_{-n}=\frac{(-1)^n}{2q}\binom{q}{n}=\frac{(-1)^n}{2q}\cdot
\frac{q(q-1)\cdots(q-n+1)}{n!}.
\end{equation}
Note that since $q\in (1,2)$ the definition~\eqref{eq:def an}
implies that for $n\neq 1$ we have $a_n> 0$. Since we defined
$a_{\pm1}$ to be equal $0$ it follows that $\{a_n\}_{n\in
\Z}\subseteq [0,\infty)$. An application of Stirling's formula
implies that as $n\to \infty$ we have
\begin{equation}\label{eq:stirling}
a_n= \frac{1}{2q}\binom{n-q-1}{n}\approx \frac{1}{n^{q+1}},
\end{equation}
where the implicit constants depend only on $q$ (and are easily
estimated if so desired). Note in particular that since $q>p$,
\eqref{eq:stirling} implies that
\begin{equation}\label{eq:finite moment}
\sum_{n\in \Z} a_n |n|^p<\infty,
\end{equation}
and
\begin{equation}\label{eq:def charatcteristic}
\varphi(\theta)\coloneqq \sum_{n\in \Z} a_n e^{in\theta}
\end{equation}
converges uniformly on $[-\pi,\pi]$. Moreover it is easy to compute
$\varphi(\theta)$ explicitly:
\begin{multline}\label{eq:compute phi}
\varphi(\theta)=\frac{e^{i\theta}+e^{-i\theta}}{2}+\frac{1}{2q}
\sum_{n=0}^\infty(-1)^n\binom{q}{n}\left(e^{in\theta}+e^{-in\theta}\right)=\cos\theta
+\frac{\left(1-e^{i\theta}\right)^q+\left(1-e^{-i\theta}\right)^q}{2q}\\
=\cos\theta+
\frac{2^{q/2}}{q}(1-\cos\theta)^{q/2}\cos\left(\frac{q(\pi-\theta)}{2}\right)\in
\R.
\end{multline}
An immediate consequence of~\eqref{eq:compute phi} is that
$\sum_{n\in \Z} a_n=\varphi(0)=1$. Thus we can define a symmetric
random variable $X$ on $\Z$ by $\Pr[X=n]=a_n$. With this
notation~\eqref{eq:finite moment} becomes $\E |X|^p<\infty$. Another
corollary of the identity~\eqref{eq:compute phi} is that there exits
$\e=\e(q)\in (0,1)$ and $c=c(q)>0$ such that for every $\theta\in
[-\e,\e]$ we have $ \E \left[e^{i\theta X}\right]=\varphi(\theta)\in
\left[e^{-2c|\theta|^q},e^{-c|\theta|^q}\right]$. Note also that
since for every $\theta\neq 0$ we have $|\varphi(\theta)|<\sum_{n\in
\Z}a_n=1$ there exists some $\delta=\delta(q)\in (0,1)$ such that
for every $\theta\in [-\pi,-\e]\cup[\e,\pi]$ we have
$|\varphi(\theta)|\le 1-\delta$.

Now let $X_1,X_2,\dots$ be i.i.d. copies of $X$. Denote
$S_n=X_1+\cdots+X_n$. Then the above bounds imply that
\begin{multline*}
\Pr\left[S_n=0\right]=\frac{1}{2\pi}\int_{-\pi}^{\pi}\left(\E\left[
e^{i\theta S_n}\right]\right)d\theta=\frac{1}{2\pi}\int_{-\pi}^\pi
\varphi(\theta)^nd\theta\\\in\frac{1}{2\pi}\left[\int_{-\e}^{\e}e^{-2cn|\theta|^q}d\theta-
\int_{[-\pi,-\e]\cup[\e,\pi]}(1-\delta)^nd\theta,\int_{-\e}^{\e}e^{-cn|\theta|^q}d\theta+
\int_{[-\pi,-\e]\cup[\e,\pi]}(1-\delta)^nd\theta\right].
\end{multline*}
This implies that as $n\to \infty$ we have
\begin{eqnarray}\label{eq:return}
\Pr\left[S_n=0\right]\approx \frac{1}{n^{1/q}}.
\end{eqnarray}

Substituting~\eqref{eq:return} into~\eqref{eq:range} we see that
\begin{equation}\label{eq:stable range}
\E\left[\left|S_{[0,n]}\right|\right]=\E\left[\left|\{S_0,\ldots,S_n\}\right|\right]\gtrsim
n^{1/q}.
\end{equation}
In combination with~\eqref{eq:range} it follows that
\begin{equation}\label{eq:stable dist}
\E\left[|S_n|\right]\gtrsim n^{1/q}.
\end{equation}
Additionally, if we let $R_n$ be as in Lemma~\ref{lem:second moment}
(for the particular symmetric walk $S_n$ studied here) then by
plugging~\eqref{eq:return} into~\eqref{eq:second moment} we get the
bound
\begin{equation}\label{eq:stable R_n}
\E\left[R_n^{1/q}\right]\gtrsim n^{(q-1)/q^2}.
\end{equation}

\subsection{The induced walk on $\Z\bwr\Z$ and the lower bound on
$\beta_p^*(\Z\bwr\Z)$}\label{sec:copy walk}

In this section we will conclude the proof of Theorem~\ref{thm:betap
ZZ}. Modulo the previous preparatory sections, the argument below
closely follows the proof of Theorem 6.2 in~\cite{NP07}.

For every $n_1,n_2,n_3\in \Z$ define $f_{n_1,n_2}^{n_3}:\Z\to \Z$ by
$$
f_{n_1,n_2}^{n_3}(k)\coloneqq n_1\1_{\{0\}}+n_2\1_{\{n_3\}}=
\left\{\begin{array}{ll}n_1 &
\mathrm{if}\ k=0,\\
n_2&\mathrm{if}\ k=n_3\ \wedge\ n_3\neq 0,\\n_1+n_2 &\mathrm{if}\
k=0=n_3,\\0&\mathrm{otherwise.}\end{array}\right.
$$
Denote
\begin{equation}\label{eq:def x_ni}
x_{n_1,n_2,n_3}\coloneqq\left(f_{n_1,n_2}^{n_3},n_3\right)\in \Z\bwr
\Z.
\end{equation}
 To better understand the meaning of this group element,
note that for every $(g,\ell)\in \Z\bwr \Z$ we have
$(g,\ell)x_{n_1,n_2,n_3}=(h,\ell+n_3)$ where
$$
h(k)=\left\{\begin{array}{ll}g(k)+n_1 &
\mathrm{if}\ k=\ell,\\
g(k)+n_2&\mathrm{if}\ k=\ell+n_3\ \wedge \ n_3\neq 0,\\g(\ell)+n_1+n_2&\mathrm{if}\ k=\ell\ \wedge\ n_3=0,\\
g(k)&\mathrm{otherwise.}\end{array}\right.
$$
Thus if we let $\mu$ be the symmetric probability measure on $\Z\bwr
\Z$ given by $\mu(\{x_{n_1,n_2,n_3}\})=a_{n_1}a_{n_2}a_{n_3}$, where
$\{a_n\}_{n\in \Z}$ are the coefficients from
Section~\ref{sec:binomial}, then the walk $\{W_t^\mu\}_{t=0}^\infty$
can be described in words as follows: start at $(\0,0)$ and at each
step choose three i.i.d. numbers $n_1,n_2,n_3\in \Z$ distributed
according to the random variable $X$ from
Section~\ref{sec:binomial}. Add $n_1$ to the current location of the
lamplighter, move the lamplighter $n_3$ units and add $n_2$ to the
new location of the lamplighter.

Write $W_t^\mu=(f_t,m_t)$. By the above description $m_t$ has the
same distribution as the walk $S_t$ from Section~\ref{sec:binomial}.
Fix $n\in \N$ and for $m\in \Z$ denote $T_m\coloneqq
\left|\left\{t\in \{0,\ldots,n\}:\ m_t=m\right\}\right|$. The above
description of the walk $W_t^\mu$ ensures that conditioned on
$\{T_m\}_{m\in \Z}$ and on ``terminal point" $m_n$, if $k\in
\Z\setminus \{0,m_n\}$ then $f_n(m)$ has the same distribution as
$S_{2T_m}$, if $m\in \{0,m_n\}$ and $m_n\neq 0$ then $f_n(m)$ has
the same distribution as $S_{\max\{2T_m-1,0\}}$, and if $m\in
\{0,m_n\}$ and $m_n=0$ then $f_n(m)$ has the same distribution as
$S_{2T_m}$. Thus using~\eqref{eq:stable dist} we see that
\begin{equation}\label{eq:each site}
\E\left[|f_n(m)|\right]\gtrsim \E\left[T_m^{1/q}\right].
\end{equation}
Fix $m\in \Z$ and for $t\in \{0,\ldots,n\}$ define the event $A_t
\coloneqq \left\{m_t=m\ \wedge\  m\notin
\{0,\ldots,m_{\ell-1}\right\}$. Note that conditioned on $A_t$ the
random variable $T_m$ has the same distribution as $R_{T_m}$, where
$\{R_k\}_{k=0}^n$ is as in~\eqref{eq:stable R_n}. Hence,
\begin{equation}\label{eq:walk per site}
\E\left[T_m^{1/q}\right]\ge \sum_{t=0}^{\lfloor n/2\rfloor}
\Pr(A_t)\cdot\E\left[\left.T_m^{1/q}\right|A_t\right]\gtrsim
\sum_{t=0}^{\lfloor n/2\rfloor} \Pr(A_t)\cdot
n^{(q-1)/q^2}=n^{(q-1)/q^2}\Pr\left[m\in \{m_0,\ldots,m_{\lfloor
n/2\rfloor}\}\right].
\end{equation}
It follows that
\begin{multline}
\E\left[d_{\Z\bwr \Z}\left(W_n^\mu,(\0,0)\right)\right]\gtrsim
\sum_{m\in \Z}\E\left[|f_n(m)|\right]\stackrel{\eqref{eq:each
site}}{\gtrsim}\sum_{m\in
\Z}\E\left[T_m^{1/q}\right]\stackrel{\eqref{eq:walk per
site}}{\gtrsim}n^{(q-1)/q^2}\sum_{m\in \Z}\Pr\left[m\in
\{m_0,\ldots,m_{\lfloor
n/2\rfloor}\}\right]\\=n^{(q-1)/q^2}\E\left[\left|\{S_0,\ldots,S_{\lfloor
n/2\rfloor}\}\right|\right]\stackrel{\eqref{eq:stable
range}}{\gtrsim} n^{(q-1)/q^2}\cdot n^{1/q}=n^{(2q-1)/q^2}.
\end{multline}
On the other hand it follows from~\eqref{eq:finite moment} that
$\E\left[d_{\Z\bwr \Z}\left(W_1^\mu,(\0,0)\right)^p\right]<\infty$
so we deduce from the definition of $\beta_p^*(\Z\bwr \Z)$ that
$$
\beta_p^*(\Z\bwr \Z)\ge \frac{2q-1}{q^2}.
$$
Letting $q\to p^+$ we deduce Theorem~\ref{thm:betap ZZ}.\qed

\begin{remark}\label{rem:G wreath Z}
{\em The same argument as above actually shows that for every
finitely generated group $G$ and every $p\in (1,2]$ we have
\begin{equation}\label{eq:G wreath Z}
\beta_p^*(G\bwr\Z)\ge
\frac{1}{p}+\left(1-\frac{1}{p}\right)\beta_p^*(G).
\end{equation}
This implies Theorem~\ref{thm:betap ZZ} since the computations in
Section~\ref{sec:binomial} show that $\beta_p^*(\Z)\ge \frac{1}{p}$.
Note of course that due to Theorem~\ref{thm:beta p} we actually know
that $\beta_p^*(\Z)=\frac{1}{p}$. We also observe that if $H$ is a
finitely generated group whose growth is at least quadratic then
$\beta_p^*(G\bwr H)=1$. Indeed we have established the fact that
$\beta_p^*(G\bwr H)\le 1$ in~\eqref{eq:beta<1}, while the lower
bound follows from Theorem 6.1 in~\cite{NP07} which states that
$\beta^*(G\bwr H)=1$, combined with the obvious fact that
$\beta^*(G\bwr H)\le \beta_p^*(G\bwr H)$. \fin}
\end{remark}

\begin{remark}\label{rem:iterated wreath}{\em Define inductively
$\Z_{(1)}=\Z$ and $\Z_{k+1}=\Z_{(k)}\bwr \Z$. Then for $p\in (1,2]$
we have  $\beta_p^*(Z_{(1)})=\frac{1}{p}$ and~\eqref{eq:G wreath Z}
implies that $\beta_p^*(\Z_{(k+1)})\ge
\frac{1}{p}+\left(1-\frac{1}{p}\right)\beta_p^*(\Z_{(k)})$. It
follows by induction that for all $k\in \N$ we have
\begin{equation}\label{eq:iterated beta}\beta_p^*(\Z_{(k)})\ge 1-\left(1-\frac{1}{p}\right)^k.\end{equation} Note
that $\alpha_p^*(\Z_{(1)})=1$ and by~\cite{Tess06}
$\alpha_p^*(C_2\bwr \Z)=1$ (see also the different proof of this
fact in~\cite{NP07}). Thus Corollary~\ref{coro:Lp} implies that
$$
\alpha_p^*(\Z_{(k+1)})\ge
\frac{p\alpha_p^*(\Z_{(k)})}{p\alpha_p^*(\Z_{(k)})+p-1}.
$$
It follows by induction that
\begin{equation}\label{eq:iterated alpha}
\alpha_p^*(\Z_{(k)})\ge
\frac{1}{p\left(1-\left(1-\frac{1}{p}\right)^k\right)}.
\end{equation}
By combining~\eqref{eq:iterated beta} and~\eqref{eq:iterated alpha}
with Theorem~\ref{thm:beta p} we see that
$$
\alpha_p^*(\Z_{(k)})=
\frac{1}{p\left(1-\left(1-\frac{1}{p}\right)^k\right)}\quad\mathrm{and}\quad
\beta_p^*(\Z_{(k)})= 1-\left(1-\frac{1}{p}\right)^k.
$$
 For $p\in (2,\infty)$ the same reasoning (using the fact that $L_p$ has Markov type 2~\cite{NPSS06}) shows that
 $\alpha_p^*(\Z_{(k)})=\alpha_2^*(\Z_{(k)})$ and $\beta_p^*(\Z_{(k)})=\beta_2^*(\Z_{(k)})$.\fin}
\end{remark}

\section{A computation of $\beta_p^*\left((\Z\bwr \Z)_0,d_{\Z\bwr
\Z}\right)$}\label{sec:zero section walk}

The purpose of this section is to prove the following result:

\begin{thm}\label{thm:infinite groups}
Let $G,H$ be infinite groups generated by the finite symmetric sets
$S_G\subseteq G$ and $S_H\subseteq H$, respectively. Let $(G\bwr
H)_0=\{(f,x)\in G\bwr H:\ x=e_H\}$ be the zero section of $G\bwr H$.
Then for all $p\in [1,2]$ we have
\begin{equation}\label{eq:general zero}
\beta_p^*\left((G\bwr H)_0,d_{G\bwr H}\right)\ge \frac{2}{p+1}.
\end{equation}
\end{thm}
Specializing to the case $G=H=\Z$ we can apply Theorem~\ref{thm:beta
p} when $\rho$ is the metric induced from $\Z\bwr \Z$ on the
amenable group $(Z\bwr \Z)_0$ to deduce that
\begin{equation}\label{eq:use theorems zero}
\frac{p+1}{2p}\ge \frac{1}{p\beta_p^*\left((\Z\bwr\Z)_0,d_{\Z\bwr
\Z}\right)}\ge \alpha_p^*\left((\Z\bwr\Z)_0,d_{\Z\bwr
\Z}\right)\stackrel{\mathrm{(Thm.~\ref{thm:zero
section})}}{=}\frac{p+1}{2p}.
\end{equation}
Thus in particular there is equality in~\eqref{eq:general zero} when
$G=H=\Z$.

\begin{proof}[Proof of Theorem~\ref{thm:infinite groups}] For every
$k\in \N$ let $g_k\in G$ and $h_k\in H$ be elements satisfying
$d_G(g_k,e_G)=k$ and $d_H(h_k,e_H)=k$. Such elements exists since
$G,H$ are assumed to be infinite. We shall write below
$h_k^{-1}=h_{-k}$. Fix an even integer $n\in \N$. For every $k\in
[1,n/2]\cup[-n/2,-1]$ and $\e,\delta\in \{-1,1\}$ define
$f_{k,\e,\delta}:H\to G$ by
$$
f_{k,\e,\delta}(x)\coloneqq \left\{\begin{array}{ll}g_n^\e&
\mathrm{if}\ x=e_H,\\g_n^\delta& \mathrm{if}\ x=h_k,\\e_G&
\mathrm{otherwise}.\end{array}\right.
$$
Let $\mu_n$ be the symmetric measure on $(G\bwr H)_0$ which is
uniformly distributed on the $4n$ elements
$$
\left\{(f_{k,\e,\delta},e_H):\ k\in [1,n/2]\cup[-n/2,-1],\
\e,\delta\in \{-1,1\}\right\}\subseteq (G\bwr H)_0.
$$
Then the following point-wise inequality holds true:
\begin{eqnarray}\label{eq:pointwise}
0<d_{G\bwr H}\left(W_1^{\mu_n},e_{G\bwr H}\right)\le 3n.
\end{eqnarray}
It follows in particular that the conditions in~\eqref{eq:condition}
hold true for the sequence $\{\mu_n\}_{n=1}^\infty$. Moreover, for
each $k\in [1,n/2]\cup[-n/2,-1]$ the probability that in exactly one
of the first $n$ steps of the walk
$\left\{W_t^{\mu_n}\right\}_{t=0}^\infty$ the $h_k$ coordinate was
altered is $\left(1-\frac{1}{n}\right)^{n-1}> \frac13$. Therefore
the expected number of of coordinates $h_k$ that were altered
exactly once is greater than $n/3$. Each such coordinate contributes
$n$ to the distance between $W_n^{\mu_n}$ and $e_{G\bwr H}$. Hence
\begin{equation}\label{n-step}
\E_{\mu_n}\left[d_{G\bwr H}\left(W_n^{\mu_n},e_{G\bwr
H}\right)\right]\ge \frac{n^2}{3}.
\end{equation}
It follows from the definition~\eqref{eq:def betap} that
$$
\beta_p^*\left((G\bwr H)_0,d_{G\bwr H}\right)\ge
\beta_p^*\left(\left\{\mu_n\right\}_{n=1}^\infty,d_{G\bwr
H}\right)\stackrel{\eqref{eq:pointwise}\wedge\eqref{n-step}}{\ge}\limsup_{n\to\infty}
\frac{\log(n^2/3)}{\log(3^pn^{1+p})}=\frac{2}{p+1},
$$
as required.
\end{proof}

\section{An application to the Lipschitz extension
problem}\label{sec:extension}

The purpose of this section is to prove the following theorem:

\begin{thm}\label{thm:no extend}
There exists a Lipschitz function $F:(\Z\bwr \Z)_0\to L_2$ which
cannot be extended to a Lipschitz function from $\Z\bwr \Z$ to
$L_2$.
\end{thm}
The key step in the proof of Theorem~\ref{thm:no extend} is the use
of the function constructed in Theorem~\ref{thm:zero section}. The
other fact that we will need is Lemma~\ref{lem:tubular} below.
Recall that a Markov chain $\{Z_t\}_{t=0}^\infty$ is called a
symmetric Markov chain on $\Z\bwr \Z$ if there exists an $N$-point
subset $\{z_1,\ldots,z_N\}\subseteq \Z\bwr \Z$ and an $N\times N$
symmetric stochastic matrix $A=(a_{ij})$ such that
$\Pr[Z_0=z_i]=\frac{1}{N}$ for all $i\in \{1,\ldots,N\}$ and for all
$i,j\in \{1,\ldots,N\}$ and $t\in \N$ we have
$\Pr[Z_{t+1}=x_j|Z_t=z_i]=a_{ij}$.

The following lemma asserts that there is a fast-diverging symmetric
Markov chain on $\Z\bwr\Z$ which remains within a relatively narrow
tubular neighborhood around the zero section $(\Z\bwr \Z)_0$.

\begin{lem}\label{lem:tubular} For every $\e>0$ there exists an integer $n_0(\e)\in N$ such that for all $n\ge n_0(\e)$
there is a symmetric Markov chain $\{Z_t\}_{t=0}^\infty$ on $\Z\bwr
\Z$ which satisfies the following conditions:
\begin{enumerate}
\item $d_{\Z\bwr \Z}(Z_1,Z_0)\le 4$ (point-wise),
\item $d_{\Z\bwr\Z}\left(Z_t,(\Z\bwr \Z)_0\right)\le 2n^{(1+\e)/2}$
for all $t\ge 0$ (point-wise),
\item $\E\left[d_{\Z\bwr\Z}\left(Z_n,Z_0\right)\right]\gtrsim
n^{3/4}.$
\end{enumerate}
\end{lem}

Assuming Lemma~\ref{lem:tubular} for the moment we shall prove
Theorem~\ref{thm:no extend}.

\begin{proof}[Proof of Theorem~\ref{thm:no extend}]
Fix $\e\in (0,1/11)$. By Theorem~\ref{thm:zero section} there exists
a function $F:(\Z\bwr\Z)_0\to L_2$ and $c=c(\e)>0$ such that
$\|F\|_{\Lip}=1$ and for every $x,y\in (\Z\bwr\Z)_0$ we have
\begin{equation}\label{eq:lower no extend}
\|F(x)-F(y)\|_2\ge cd_{\Z\bwr \Z}(x,y)^{(3-\e)/4}.
\end{equation}
Assume for the sake of contradiction that there exists a function
$\widetilde F:\Z\bwr \Z\to L_2$ such that $\widetilde
F\upharpoonright_{(\Z\bwr\Z)_0}=F$ and $\left\|\widetilde
F\right\|_{\Lip}=L<\infty$.

Let $n_0(\e)$ and $\{Z_t\}_{t=0}^\infty$ be as in
Lemma~\ref{lem:tubular} and fix $n\ge n_0(\e)$. Write
$Z_t=(f_t,k_t)$ and define $Z_t^0=(f_t,0)\in (\Z\bwr \Z)_0$. The
second assertion of Lemma~\ref{lem:tubular} implies that for all
$t\ge 0$ we have
\begin{equation}\label{eq:close to zero}
d_{\Z\bwr\Z}\left(Z_t,Z_t^0\right)\le 2n^{(1+\e)/2}.
\end{equation}
Using the Markov type $2$ property of $L_2$~\cite{Bal} (with
constant $1$) and the first assertion of Lemma~\ref{lem:tubular} we
see that:
\begin{equation}\label{eq:use markov l_2}
 \E\left[\left\|\widetilde
F(Z_n)-\widetilde F(Z_0)\right\|_2^2\right]\le
n\E\left[\left\|\widetilde F(Z_1)-\widetilde
F(Z_0)\right\|_2^2\right]\le nL^2 \E\left[d_{\Z\bwr
\Z}\left(Z_1,Z_0\right)^2\right]\le 16nL^2.
\end{equation}
Note the following elementary corollary of the triangle inequality
which holds for every metric space $(X,d)$, every $p\ge 1$ and every
$a_1,a_2,b_1,b_2\in X$:
\begin{equation}\label{eq:triangle}
d(a_1,b_1)^p\ge
\frac{1}{3^{p-1}}d(a_2,b_2)^p-d(a_1,a_2)^p-d(b_1,b_2)^p.
\end{equation}
Hence we have the following point-wise inequality:
\begin{eqnarray}\label{triangle twice}
\nonumber
&&\!\!\!\!\!\!\!\!\!\!\!\!\!\!\!\!\!\!\!\!\!\!\!\!\!\!\!\!\!\!\!\!\!\!\left\|\widetilde
F(Z_n)-\widetilde
F(Z_0)\right\|_2^2\stackrel{\eqref{eq:triangle}}{\ge} \frac13
\left\|
F\left(Z_n^0\right)-F\left(Z_0^0\right)\right\|_2^2-\left\|\widetilde
F(Z_n)-\widetilde F\left(Z_n^0\right)\right\|_2^2-\left\|\widetilde
F(Z_0)-\widetilde F\left(Z_0^0\right)\right\|_2^2\\
\nonumber&\stackrel{\eqref{eq:lower no extend}}{\ge}& \frac{c^2}{3}
d_{\Z\bwr \Z}\left(Z_n^0,Z_0^0\right)^{(3-\e)/2}-L^2d_{\Z\bwr
\Z}\left(Z_n,Z_n^0\right)^2-L^2d_{\Z\bwr \Z}\left(Z_0,Z_0^0\right)^2
\\\nonumber&\stackrel{\eqref{eq:triangle}\wedge \eqref{eq:close to zero}}{\ge}& \frac{c^2}{3}\left(\frac13d_{\Z\bwr
\Z}\left(Z_n,Z_0\right)^{(3-\e)/2}-d_{\Z\bwr
\Z}\left(Z_n,Z_n^0\right)^{(3-\e)/2}-d_{\Z\bwr
\Z}\left(Z_0,Z_0^0\right)^{(3-\e)/2}\right)-8L^2n^{1+\e} \\
&\stackrel{\eqref{eq:close to zero}}{\ge}& \frac{c^2}{9}d_{\Z\bwr
\Z}\left(Z_n,Z_0\right)^{(3-\e)/2}-10L^2n^{1+\e}.
\end{eqnarray}
Taking expectation in~\eqref{triangle twice} and using the third
assertion of Lemma~\ref{lem:tubular} we see that:
\begin{multline*}
16nL^2\ge\E\left[\left\|\widetilde F(Z_n)-\widetilde
F(Z_0)\right\|_2^2\right]\ge \frac{c^2}{9}\E\left[d_{\Z\bwr
\Z}\left(Z_n,Z_0\right)^{(3-\e)/2}\right]-10L^2n^{1+\e}\\\ge
\left(\E\left[d_{\Z\bwr
\Z}\left(Z_n,Z_0\right)\right]\right)^{(3-\e)/2}-10L^2n^{1+\e}\gtrsim
n^{3(3-\e)/8}-10L^2n^{1+\e},
\end{multline*}
which is a contradiction for large enough $n$ since the assumption
$\e<1/11$ implies that $\frac{3(3-\e)}{8}>1+\e$.
\end{proof}

It remains to prove Lemma~\ref{lem:tubular}.

\begin{proof}[Proof of Lemma~\ref{lem:tubular}]
Fix an integer $n\in \N$ and $\e\in (0,1/4)$.
%such that
%$n^{(1+\e)/2}$ is an integer (the latter assumption is just for
%convenience and plays no meaningful role in the ensuing argument).
Define two subsets $U_n,V_n\subseteq \Z\bwr \Z$ by
$$
U_n\coloneqq \left\{(f,k)\in \Z\bwr \Z:\ \supp(f)\subseteq
\left[-n,n\right],\ |k|\le 2n^{(1+\e)/2},\ |f(\ell)|\le n^2\ \
\forall\  \ell\in \Z\right\},
$$
$$
V_n\coloneqq \left\{(f,k)\in \Z\bwr \Z:\ \supp(f)\subseteq
\left[-n,n\right],\ |k|\le n^{(1+\e)/2}, |f(\ell)|\le n^2-2n\
\forall\ \ell\in \Z\right\}.
$$
Then
$|U_n|\approx\left(2n^2+1\right)^{2n+1}\left(4n^{(1+\e)/2}+1\right)$
and
$|V_n|\approx\left(2n^2-4n+1\right)^{2n+1}\left(2n^{(1+\e)/2}+1\right)$
so that \begin{equation}\label{eq:V_n/U_n}
\frac{|V_n|}{|U_n|}\gtrsim 1. \end{equation}

Consider the set $S=\{x_{n_1,n_2,n_3}:\ n_1,n_2,n_3\in \{-1,1\}\}$,
where $x_{n_1,n_2,n_3}$ are as defined in~\eqref{eq:def x_ni}. Then
$S$ is a symmetric generating set of $\Z\bwr \Z$ consisting of $8$
elements. Let $g_1,g_2,\ldots$ be i.i.d. elements of $\Z\bwr \Z$
which are uniformly distributed over $S$ and denote $W_m\coloneqq
g_1\cdots g_m=(f_m,k_m)$. Then by construction the sequence
$\{k_m\}_{m=1}^\infty$ has the same distribution as the standard
random walk on $\Z$, i.e., the same distribution as
$\{S_m=\e_1+\cdots+\e_m\}_{m=1}^\infty$ where $\e_1,\e_2,\ldots$ are
i.i.d. Bernoulli random variables (this fact was explained in
greater generality in Section~\ref{sec:copy walk}). Also, as shown
by {\`E}rschler~\cite{Ersh01}, we have
\begin{eqnarray}\label{eq:erschler}
\E\left[d_{\Z\bwr\Z}\left(W_n,(\0,0)\right)\right]\ge c n^{3/4},
\end{eqnarray}
where $c>0$ is a universal constant. Note  that since
$d_{\Z\bwr\Z}(x_{n_1,n_2,n_3},(\0,0))\le 4$ for every
$n_1,n_2,n_3\in \{-1,1\}$ we have point-wise bound
\begin{eqnarray}\label{eq:pointwise special generators}
d_{\Z\bwr\Z}\left(W_n,(\0,0)\right)\le 4n.
\end{eqnarray}
 Now let $Z_0$ be
uniformly distributed over $U_n$ and independent of
$\{g_i\}_{i=1}^\infty$. For $t\in \N$ define
$$
Z_t\coloneqq \left\{\begin{array}{ll}  Z_{t-1}g_t &\mathrm{if\ } Z_{t-1}g_t\in U_n,\\
Z_{t-1} & \mathrm{otherwise}.
\end{array}\right.
$$
The first two assertions of Lemma~\ref{lem:tubular} hold true by
construction. It remains to establish the third assertion of
Lemma~\ref{lem:tubular}.

Consider the events $\mathcal E\coloneqq\{Z_0\in V_n\}$ and
$\mathcal F\coloneqq \left\{\max_{m\le n}|k_m|\le
n^{(1+\e)/2}\right\}$. Note that if the event $\mathcal E\cap
\mathcal F$ occurs then $Z_n=Z_0W_n$ since by design in this case
$Z_0\in V_n$ and therefore $Z_0W_t$ cannot leave $U_n$ for all $t\le
n$. It follows that
\begin{multline}\label{eq:before doob}
\E\left[d_{\Z\bwr\Z}\left(Z_n,Z_0\right)\right]\ge
\E\left[d_{\Z\bwr\Z}\left(W_n,(\0,0)\right)\1_{\mathcal E\cap
\mathcal F}\right]=\Pr\left[\mathcal
E\right]\left(\E\left[d_{\Z\bwr\Z}\left(W_n,(\0,0)\right)\right]-\E\left[d_{\Z\bwr\Z}\left(W_n,(\0,0)\right)
\1_{\mathcal{F}^c}\right]\right)\\\stackrel{\eqref{eq:erschler}\wedge\eqref{eq:pointwise
special generators}}{\ge}
\frac{|V_n|}{|U_n|}\left(cn^{3/4}-4n(1-\Pr[\mathcal
F])\right)\stackrel{\eqref{eq:V_n/U_n}}{\gtrsim}
cn^{3/4}-4n(1-\Pr[\mathcal F]).
\end{multline}
For large enough $n$ (depending on $\e$) we have
\begin{equation}\label{eq:use doob}
4n(1-\Pr[\mathcal F])\le \frac{c}{2}n^{3/4}, \end{equation} since
Doob's maximal inequality (see e.g.~\cite{Durrett96}) implies that
for every $p>1$ we have
\begin{equation}\label{eq:doob}
1-\Pr[\mathcal F]=\Pr\left[\max_{m\le n}|k_m|>
n^{(1+\e)/2}\right]\le
\left(\frac{p}{p-1}\right)^p\frac{\E\left[|\e_1+\cdots+\e_n|^p\right]}{n^{p(1+\e)/2}}\stackrel{(\clubsuit)}{\lesssim}
\left(\frac{p}{p-1}\right)^p\frac{(10np)^{p/2}}{n^{p(1+\e)/2}}=
\frac{C(p)}{n^{p\e/2}},
\end{equation}
where in ($\clubsuit$) we used Khinchine's inequality (see
e.g.~\cite{MS86}) and $C(p)$ depends only on $p$. Hence choosing $p$
large enough in~\eqref{eq:doob} (depending on $\e$)
implies~\eqref{eq:use doob}. Combining~\eqref{eq:before doob}
and~\eqref{eq:use doob} implies that
$$
\E\left[d_{\Z\bwr\Z}\left(Z_n,Z_0\right)\right]\gtrsim n^{3/4},
$$
which completes the proof of Lemma~\ref{lem:tubular}.
\end{proof}
\section{Reduction to equivariant embeddings}\label{sec:equivariant}

Recall that a Banach space $(X,\|\cdot\|_X)$ is said to be finitely
representable in a Banach space $(Y,\|\cdot\|_Y)$ if for every
$\e>0$ and every finite dimensional subspace $F\subseteq X$ there is
a linear operator $T:F\to Y$ such that for every $x\in F$ we have
$\|x\|_X\le\|Tx\|_Y\le (1+\e)\|x\|_X$.

\begin{thm}\label{gromov} Let $\Gamma$ be an amenable group
which is generated by a finite symmetric set $S\subseteq \Gamma$.
Fix $p\ge 1$, two functions $\omega,\Omega:[0,\infty)\to [0,\infty)$
and a Banach space $(X,\|\cdot\|_X)$ such that there is a mapping
$\psi:\Gamma\to X$ which satisfies:
\begin{equation}\label{eq:assumption uniform}
g,h\in \Gamma\implies \omega\left(d_\Gamma(g,h)\right)\le
\|\psi(g)-\psi(h)\|_X\le \Omega\left(d_\Gamma(g,h)\right).
\end{equation}
Then there exists a Banach space $Y$ which is finitely representable
in $\ell_p(X)$ and an equivariant mapping $\Psi:\Gamma\to Y$ such
that
\begin{equation}\label{eq:goal uniform}
g,h\in \Gamma\implies \omega\left(d_\Gamma(g,h)\right)\le
\left\|\Psi(g)-\Psi(h)\right\|_{Y}\le
\Omega\left(d_\Gamma(g,h)\right).
\end{equation}
Moreover, if $X=L_p(\mu)$ for some measure $\mu$ then $Y$ can be
taken to be isometric to $L_p$.
\end{thm}
Note that as a special case of Theorem~\ref{gromov} we conclude that
for every $p\ge 1$ if $\Gamma$ is an amenable group then
$\alpha_p^*(\Gamma)=\alpha_p^\#(\Gamma)$.

In what follows given a Banach space $X$ we denote by $\Isom(X)$ the
group of all linear isometric automorphims of $X$. We shall require
the following lemma in the proof of Theorem~\ref{gromov}:

\begin{lem}\label{lem:atomic} Fix $p\in [1,\infty)$. Let $G$ be a finitely generated group and
$(\Omega,\FF,\mu)$ be a measure space (thus $\Omega$ is a set, $\FF$
is a $\sigma$ algebra, and $\mu$ is a measure on $\FF$). Assume that
$\pi_0:G\to \Isom\left(L_p(\mu,\FF)\right)$ is a homomorphism and
that $f_0\in Z^1(G,\pi_0)$ a $1$-cocycle. Then there exists a
homomorphism $\pi:G\to \Isom\left(L_p\right)$ and a $1$-cocycle
$f\in Z^1(G,\pi)$ such that
$\|f(x)\|_{L_p}=\|f_0(x)\|_{L_p(\mu,\FF)}$ for all $x\in G$.
\end{lem}

\begin{proof} Given $A\subseteq L_p(\mu,\FF)$ we denote as usual the smallest sub-$\sigma$ algebra of $\FF$ with respect to which
all the elements of $A$ are measurable by $\sigma(A)$. Define
inductively a sequence $\{\FF_n\}_{n=1}^\infty$ of sub-$\sigma$
algebras of $\FF$ and two sequences $\{U_n\}_{n=1}^\infty$,
$\{V_n\}_{n=1}^\infty$ of linear subspaces of $L_p(\mu,\FF)$ as
follows:
$$U_1=\span\left(\bigcup_{x\in G}\pi_0(x)f_0(G)\right), \quad\FF_1\coloneqq
\sigma\left(U_1\right),\quad V_1=L_p(\mu,\FF_1), $$
and inductively
$$
U_{n+1}\coloneqq \span\left(\bigcup_{x\in G}\pi_0(x)V_n\right),\quad
\FF_{n+1}\coloneqq \sigma(U_{n+1}),\quad V_{n+1}=L_p(\mu,\FF_{n+1}).
$$
By construction for each $n\in \N$ we have $U_n\subseteq
V_n\subseteq U_{n+1}$, the measure space $(\Omega,\FF_n,\mu)$ is
separable (since $G$ is countable) and $\FF_{n+1}\supseteq \FF_n$.
Let $\FF_\infty$ be the $\sigma$-algebra generated by
$\bigcup_{n=1}^\infty \FF_n$. Note that for every $\e>0$ and every
$A\in \FF_\infty$ there is some $n\in \N$ and $B\in \FF_n$ such that
$\mu(A\triangle B)\le \e$ (this is because the set of all such $A\in
\FF$ forms a $\sigma$ algebra, and therefore contains $\FF_\infty$).
By considering approximations by simple functions we deduce that
\begin{equation}\label{eq:infinite union}
L_p(\mu,\FF_\infty)=\overline{\bigcup_{n=1}^\infty V_n},
\end{equation}
where the closure is taken in $L_p(\mu,\FF)$. We claim that for each
$x\in G$ we have $\pi_0(x)\in
\Isom\left(L_p(\mu,\FF_\infty)\right)$. Indeed, by construction
$\pi_0(x)U_n=U_n$ for all $n\in \N$, and therefore $V_n\subseteq
\pi_0(x)V_{n+1}\subseteq V_{n+2}$, which implies that
$\pi_0(x)L_p(\mu,\FF_\infty)=L_p(\mu,\FF_\infty)$, as required. Note
also that $f(G)\subseteq L_p(\mu,\FF_\infty)$.

Since $L_p(\mu,\FF_\infty)$ is separable it is isometric to one of
the spaces:
\begin{equation}\label{eq:list}
L_p,\quad \ell_p,\quad, \left\{\ell_p^n\right\}_{n=1}^\infty,\quad
L_p\oplus \ell_p,\quad \left\{L_p\oplus
\ell_p^n\right\}_{n=1}^\infty,
\end{equation}
where the direct sums in~\eqref{eq:list} are $\ell_p$ direct sums
(see~\cite{Woj91}). In what follows we will slightly abuse notation
by saying that $L_p(\mu,\FF_\infty)$ is equal to one of the spaces
listed in~\eqref{eq:list}. The standard fact~\eqref{eq:list} follows
from decomposing the measure $\mu\upharpoonright_{\FF_\infty}$ into
a non-atomic part and a purely atomic part, and noting that the
purely atomic part can contain at most countably many atoms while
the non-atomic part is isomorphic to $[0,1]$ (equipped with the
Lebesgue measure) by Lebesgue's isomorphism theorem
(see~\cite{Halmos}).

If $L_p(\mu,\FF_\infty)=L_p$ then we are done, since we can take
$\pi=\pi_0\upharpoonright_{L_p(\mu,\FF_\infty)}$, so assume that
$L_p(\mu,\FF_\infty)$ is not isometric to $L_p$. We may therefore
also assume that $p\neq 2$. If $L_p(\mu,\FF_\infty)=\ell_p$ then by
Lamperti's theorem~\cite{Lam58} (see also Chapter 3 in~\cite{FJ03})
for every $x\in G$, since $\pi_0(x)$ is a linear isometric
automorphism of $\ell_p$ (and $p\neq 2$) we have
$\pi_0(x)e_i=\theta^x_ie_{\tau^x(i)}$ for all $i\in \N$, where
$\{e_i\}_{i=1}^\infty$ is the standard coordinate basis of $\ell_p$,
the function $\tau^x:\N\to \N$ is one-to-one and onto and
$|\theta^x|\equiv 1$. Define $\pi(x)\in \Isom\left(L_p\right)$ and
$f:G\to L_p$ by setting for $h\in L_p$ and $t\in
\left[2^{-i},2^{-i+1}\right]$,
$$
\pi(x)h(t)\coloneqq\theta^x_ih\left(2^{i-\tau^x(i)}t\right)\quad
\mathrm{and}\quad f(x)(t)=2^{i/p}\langle f_0(x), e_i\rangle.
$$
It is immediate to check that $\pi,f$ satisfy the assertion of
Lemma~\ref{lem:atomic}.

It remains to deal with the case $L_p(\mu,\FF_\infty)=L_p\oplus
\ell_p(S)$ where $S$ is a nonempty set which is finite or countable.
In this case we use Lamperti's theorem once more to deduce that for
each $x\in G$ the linear isometric automorphism $\pi_0(x)$ maps
disjoint functions to disjoint functions, and therefore it maps
indicators of atoms to indicators of atoms. Hence $\pi_0(x)L_p=L_p$
and $\pi_0(x)\ell_p(S)=\ell_p(S)$. Now, as above
$\pi_0(x)\upharpoonright_{\ell_p(S)}$ must correspond (up to changes
of sign) to a permutation of the coordinates. Hence, denoting the
projection from $L_p\oplus \ell_p(S)$ onto $L_p$ by $Q$, the same
reasoning as above shows that there exists a homomorphism $\pi':G\to
L_p$ and $f'\in Z^1(G,\pi')$ such that for all $x\in G$ we have
$\|f'(x)\|_{L_p}=\|f_0(x)-Qf_0(x)\|_{\ell_p(S)}$. It follows that if
we define $\pi(x)\in \Isom\left(L_p\oplus L_p\right)$ by
$\pi(x)=\pi_0(x)\upharpoonright_{L_p}\oplus \pi'$ and $f:G\to
L_p\oplus L_p$ by $f(x)=(Qf)\oplus f'$ then (using the fact that
$L_p\oplus L_p$ is isometric to $L_p$) the assertion of
Lemma~\ref{lem:atomic} follows in this case as well.
\end{proof}

\begin{proof}[Proof of Theorem~\ref{gromov}] Let $\{F_n\}_{n=0}^\infty$ be a F\o lner sequence for $\Gamma$ and
let $\mathscr U$ be a free ultrafilter on $\N$. Define
$\M:\ell_\infty(\Gamma)\to \R$ by
\begin{equation}\label{eq:def special mean}
\M(f)=\lim_{\mathscr U} \frac{1}{|F_n|}\sum_{x\in F_n} f(x).
\end{equation}
It follows immediately from the F\o lner condition that $\M$ is an
invariant mean on $\Gamma$, i.e., a linear functional
$\M:\ell_\infty(\Gamma)\to \R$ which maps the constant $1$ function
to $1$, assigns non-negative values to non-negative functions and
$\M(R_yf)=\M(f)$ for every $y\in \Gamma$, where $R_yf(x)=f(xy)$ (we
refer to~\cite{Wag93} for proofs and more information on this
topic). Define a semi-norm $\|\cdot\|_{\M,p}$ on
$\ell_\infty(\Gamma,X)$ (the space of all $X$-valued bounded
functions on $\Gamma$) by:
\begin{eqnarray*}\label{eq:def seminorm}
f\in \ell_\infty(\Gamma,X)\implies \|f\|_{\M,p}\coloneqq
\left(\M\left(\|f\|_X^p\right)\right)^{1/p}.
\end{eqnarray*}
This is indeed a semi-norm  since invariant means satisfy H\"older's
inequality (see for example Lemma 2 on page 119 of Section III.3
in~\cite{DS88}). Hence if we let $W=\{f\in \ell_\infty(\Gamma,X):\
\|f\|_{\M,p}=0\}$ then $W$ is a linear subspace and $Y_0\coloneqq
\ell_\infty(\Gamma,X)/W$ is a normed space. Let $Y$ be the
completion of $Y_0$.

By a slight abuse of notation we denote for $y\in \Gamma$ and $f\in
\ell_\infty(\Gamma,X)$, $R_y(f+W)\coloneqq R_yf+W$, which is a well
defined linear isometric automorphism of $Y_0$ since
$\|\cdot\|_{\M,p}$ is $R_y$-invariant. Moreover $R$ is an action of
$\Gamma$ on $Y_0$ by linear isometric automorphisms, and it
therefore extends to such an action on $Y$ as well.

Note that by virtue of the upper bound in~\eqref{eq:assumption
uniform} for every $g,x\in \Gamma$ we have
$\|\psi(xg)-\psi(x)\|_X\le \Omega\left(d_\Gamma(g,e_\Gamma)\right)$.
Thus $R_g\psi-\psi\in \ell_\infty(\Gamma,X)$ and we can define
$\Psi(g)\in Y$ by $\Psi(g)=(R_g\psi-\psi)+W$. Then
%$\Psi(gh)=(R_gR_h\psi-\psi)+W=R_g(\Psi(h))+\Psi(g)$
$\Psi\in Z^1(\Gamma,R)$. Moreover $\Psi(e_\Gamma)=0$ and for every
$g_1,g_2\in \Gamma$ we have
\begin{equation}\label{eq:modulus tilde}
\nonumber\left\|\Psi(g_1)-\Psi(g_2)\right\|_{Y}=
\left(\M\left(\left\|R_{g_1}\psi-R_{g_2}\psi\right\|_X^p\right)\right)^{1/p}\stackrel{\eqref{eq:assumption
uniform}}{\in}\left[\omega\left(d_{\Gamma}\left(g_1,g_2\right)\right),\Omega\left(d_{\Gamma}\left(g_1,g_2\right)\right)\right].
\end{equation}
This establishes~\eqref{eq:goal uniform}, so it remains to prove the
required properties of $Y$, i.e., that it is finitely representable
in $\ell_p(X)$ and that it is an $L_p(\nu)$ space if $X$ is an
$L_p(\mu)$ space.

Up to this point we did not use the fact that $\M$ was constructed
as an ultralimit of averages along F\o lner sets as in~\eqref{eq:def
special mean} and we could have taken $\M$ to be any invariant mean
on $\Gamma$. But now we will use the special structure of $\M$ to
relate the space $Y$ to a certain ultraproduct of Banach spaces. We
do not know whether the properties required of $Y$ hold true for
general invariant means on $\Gamma$. We did not investigate this
question since it is irrelevant for our purposes.

For each $n\ge 0$ let $X_n$ be the Banach space $X^{F_n}$ equipped
with the norm:
$$
\psi:F_n\to X\implies \|\psi\|_{X_n}=\left(\frac{1}{|F_n|}\sum_{h\in
F_n} \|\psi(h)\|_{X}^p\right)^{1/p}.
$$
Let $\widetilde X$ be the ultraproduct $\left(\prod_{n=0}^\infty
X_n\right)_{\mathscr U}$. We briefly recall the definition of
$\widetilde X$ for the sake of completeness
(see~\cite{D-CK70,D-CK72,Hein80} for more details and complete
proofs of the ensuing claims). Let $Z$ be the space
$\left(\prod_{n=0}^\infty X_n\right)_\infty$, i.e., the space of all
sequences $x=(x_0,x_1,x_2,\ldots)$ where $x_n\in X_n$ for each $n$
and $\|x\|_Z\coloneqq\sup_{n\ge 0} \|x_n\|_{X_n}<\infty$. Let
$N\subseteq Z$ be the subspace consisting of sequences
$(x_n)_{n=0}^\infty$ for which $\lim_{\mathscr U}\|x_n\|_{X_n}=0$.
Then $N$ is a closed subspace of $Z$ and $\widetilde X$ is the
quotient space $Z/N$, equipped with the usual quotient norm. We
shall denote an element of $\widetilde X$, which is an equivalence
class of elements in $Z$, by $[x_n]_{n=0}^\infty$. The norm on
$\widetilde X$ is given by the concrete formula $
\left\|[x_n]_{n=0}^\infty\right\|_{\widetilde X}=\lim_{\mathscr
U}\|x_n\|_{X_n}$.

Since by construction each of the spaces $X_n$ embeds isometrically
into $\ell_p(X)$, by classical ulraproduct theory
(see~\cite{Hein80}) $\widetilde X$ is finitely representable in
$\ell_p(X)$. Moreover, if $X=L_p(\mu)$ for some measure $\mu$ then,
as shown in~\cite{D-CK70,D-CK72,Hein80}, $\widetilde X=L_p(\tau)$
for some measure $\tau$.

Define $T:Y_0\to \widetilde X$ by
$T(f+W)=[f\upharpoonright_{F_n}]_{n=0}^\infty$. Then by construction
(and the definition of $W$) $T$ is well defined and is an isometric
embedding of $Y_0$ into $\widetilde X$. Hence also $Y$ embeds
isometrically into $\widetilde X$, and for ease of notation we will
identify $Y$ with $\overline{T(Y_0)}\subseteq \widetilde X$. It
follows in particular that $Y$ is finitely representable in
$\ell_p(X)$.

It remains to show that if $X=L_p(\mu)$ then $Y=L_p(\nu)$ for some
measure $\nu$ since once this is achieved we can apply
Lemma~\ref{lem:atomic} in order to replace $Y$ by $L_p$. We know
that in this case $\widetilde X=L_p(\tau)$ but we need to recall the
lattice structure on $\widetilde X$ in order to proceed (since we do
not know whether the action of $\Gamma$ on $Y$ extends to an action
of $\Gamma$ on $\widetilde X$ by isometric linear automorphisms).
Since each $X_n$ is of the form $L_p(\mu_n)$ for some measure
$\mu_n$, the ultraproduct $\widetilde X$ has a Banach lattice
structure whose positive cone is $\left\{[x_n]_{n=0}^\infty:\ x_n\ge
0\ \forall n\right\}$ and $[x_n]_{n=0}^\infty\wedge
[y_n]_{n=0}^\infty=[x_n\wedge y_n]_{n=0}^\infty$,
$[x_n]_{n=0}^\infty\vee [y_n]_{n=0}^\infty=[x_n\vee
y_n]_{n=0}^\infty$ (all of this is discussed in detail
in~\cite{Hein80}). The explicit embedding of $Y_0$ into $\widetilde
X$ ensures that $x\wedge y, x\vee y\in Y_0$ for all $x,y\in Y_0$.
Moreover if $x,y\in Y_0$ are disjoint, i.e., $|x|\wedge |y|=0$, then
$\|x+y\|_{\widetilde X}=\left(\|x\|_{\widetilde
X}^p+\|y\|_{\widetilde X}^p\right)^{1/p}$. These identities pass to
the closure $Y$ of $Y_0$ (since, for example, we know that
$\widetilde X=L_p(\tau)$ and therefore convergence in $\widetilde X$
implies almost everywhere convergence along a subsequence). This
shows that  the Banach space $Y$ is an abstract $L_p$ space, and
therefore by Kakutani's representation theorem~\cite{Kak41} (see
also the presentation in~\cite{LT79}) $Y=L_p(\nu)$ for some measure
$\nu$.
\end{proof}

\section{Open problems}

We list below several of the many interesting open questions related
to the computation of compression exponents.

\begin{ques}\label{Q:planar TSP}
Does $C_2\bwr \Z^2$ admit a bi-Lipschitz embedding into $L_1$?
\end{ques}

The significance of Question~\ref{Q:planar TSP} was explained in the
introduction. Since we know that $\alpha_1^*\left(C_2\bwr
\Z^2\right)=1$ the following question is more general
then~\ref{Q:planar TSP}:

\begin{ques}\label{Q:attained}
For which finitely generated groups $G$ and $p\ge 1$ is
$\alpha_p^*(G)$ attained?
\end{ques}

Somewhat less ambitiously than Question~\ref{Q:attained} one might
ask for meaningful conditions on $G$ which imply that
$\alpha_p^*(G)$ is attained. As explained in
Remark~\ref{rem:attained}, this holds true if $p>1$ and $G=C_2\bwr
H$ where $H$ is a finitely generated group with super-linear
polynomial growth which admits a bi-Lipschitz embedding into $L_p$.
In particular this holds true for $G=C_2\bwr \Z^2$ and $p>1$. Note
that not every group of polynomial growth $H$ admits a bi-Lipschitz
embedding into $L_1$, as shown by Cheeger and Kleiner~\cite{CK06}
when $H$ is the discrete Heisenberg group, i.e. the group of
$3\times 3$ matrices generated by the following symmetric set
$S\subseteq GL_3(\mathbb Q)$ and equipped with the associated word
metric:
\begin{equation*}
S=\left\{ \begin{pmatrix}
  1 & 1&0 \\
   0 & 1&0\\
   0&0&1
   \end{pmatrix},  \begin{pmatrix}
  1 & -1&0 \\
   0 & 1&0\\
   0&0&1
   \end{pmatrix},\begin{pmatrix}
  1 & 0&1 \\
   0 & 1&0\\
   0&0&1
   \end{pmatrix},\begin{pmatrix}
  1 & 0&-1 \\
   0 & 1&0\\
   0&0&1
   \end{pmatrix},\begin{pmatrix}
  1 & 0&0 \\
   0 & 1&1\\
   0&0&1
   \end{pmatrix},\begin{pmatrix}
  1 & 0&0 \\
   0 & 1&-1\\
   0&0&1
   \end{pmatrix}\right\}.
   \end{equation*}

Similarly to Question 7.1 in~\cite{NP07} one might ask the following
question:

\begin{ques}\label{Q:alpha beta}
Is it true that for every finitely generated amenable group $G$ and
every $p\in [1,2]$ we have $\alpha_p^*(G)=\frac{1}{p\beta_p^*(G)}$?
\end{ques}

It was shown in~\cite{ADS06} the for every $\alpha\in [0,1]$ there
exists a finitely generated group $G$ such that
$\alpha_2^*(G)=\alpha$. Since there are only countably many finitely
presented groups the set $$\Omega_p^*\coloneqq \{\alpha_p^*(G):\ G\
\mathrm{finitely \ presented}\}\subseteq [0,1]$$ is at most
countable for every $p\in [1,\infty)$ (though it seems to be unknown
whether or not it is infinite). One can similarly define the set
$\Omega_p^\#$ of possible equivariant compression exponents of
finitely presented groups. Several restrictions on the relations
between these sets follow from the following inequalities which hold
for every finitely generated group $G$: for every $p\ge 1$ we have
$\alpha_p^*(G)\ge \alpha_2^*(G)$ since $L_2$ embeds isometrically
into $L_p$ (see e.g.~\cite{Woj91}). Similarly Lemma 2.3
in~\cite{NP07} states that $\alpha_p^\#(G)\ge \alpha_2^\#(G)$. Since
$L_q$ embeds isometrically into $L_p$ for $1\le p\le q\le 2$
 (see~\cite{WW75}) we also know that in this case $\alpha_p^*(G)\ge
\alpha_q^*(G)$. For every $1\le p\le q$ the metric space
$\left(L_p,\|x-y\|_p^{p/q}\right)$ embeds isometrically into $L_q$
(for $1\le p\le q\le 2$ this follows from~\cite{BD-CK65,WW75} and
for the remaining range this is proved in Remark 5.10
of~\cite{MN04}). Hence if $p\in [1,2]$ and $p\le q$ then
$\alpha_q^*(G)\ge \max\left\{\frac{p}{q},\frac{p}{2}\right\}\cdot
\alpha_p^*(G)$ and if $2\le p\le q$ then  $\alpha_q^*(G)\ge
\frac{p}{q}\alpha_p^*(G)$.

%Note that it is always that case that $\Omega_p^*\supseteq
%\Omega_2^*$ since $L_2$ embeds isometrically into $L_p$. Since $L_q$
%embeds isometrically into $L_p$ for $1\le p\le q\le 2$
%(see~\cite{WW75}) we also know that in this case
%$\Omega_p^*\supseteq \Omega_q^*$. For every $1\le p\le q$ the metric
%space $\left(L_p,\|x-y\|_p^{p/q}\right)$ embeds isometrically into
%$L_q$ (for $1\le p\le q\le 2$ this follows from~\cite{BD-CK65,WW75}
%and for the remaining range follows this is proved in Remark 5.10
%of~\cite{MN04}). Hence if $p\in [1,2]$ and $p\le q$ then for every
%$\alpha\in \Omega_p^*$ there is $\alpha' \in \Omega^*_q$ such that
%$\alpha'\ge \alpha\cdot\max\left\{p/q,p/2\right\}$ and if $2\le p\le
%q$ then for every $\alpha\in \Omega_p^*$ there is $\alpha' \in
%\Omega^*_q$ such that $\alpha'\ge (\alpha p)/q$. Finally, Lemma 2.3
%in~\cite{NP07} implies that for every $p\ge 1$ and every $\alpha\in
%\Omega_2^\#$ the exists $\alpha'\in \Omega_p^\#$ such that
%$\alpha'\ge \alpha$.
%Finally we note that by Remark~\ref{rem:iterated wreath} the set
%$\Omega_p^\#$, and hence also $\Omega_p^*$ is an infinite set for
%all $p\in (1,\infty)$. We do not know if $\Omega_1^*$ is finite or
%infinite.

\begin{ques}\label{Q:Omega}
Evaluate the (at most countable) sets $\Omega_p^*,\Omega_p^\#$. Is
$\Omega_p^*$ finite or infinite? How do the sets
$\Omega_p^*,\Omega_p^\#$ vary with $p$? Is it true that
$\Omega_p^*=\Omega_p^\#$?
\end{ques}

In this paper we computed $\alpha_p^*((\Z\bwr\Z)_0,d_{\Z\bwr \Z})$.
Note that the metric on the zero section $(\Z\bwr\Z)_0$ is not
equivalent to a geodesic metric. This fact makes it meaningful to
consider embeddings of $((\Z\bwr\Z)_0,d_{\Z\bwr \Z})$ into $L_p$
which are not necessarily Lipschitz, leading to the following
question:

\begin{ques}\label{Q:two exponents} For every $\alpha_1>0$ evaluate
the supremum over $\alpha_2\ge 0$ such that there exists an
embedding $f:(\Z\bwr\Z)_0\to L_p$ which satisfies
$$
x,y\in (\Z\bwr\Z)_0\implies cd_{\Z\bwr\Z}(x,y)^{\alpha_2}\le
\|f(x)-f(y)\|_p\le d_{\Z\bwr\Z}(x,y)^{\alpha_1},
$$
for some constant $c$.
\end{ques}
We believe that the methods of the present paper can be used to
answer Question~\ref{Q:two exponents} at least for some additional
values of $\alpha_1$ (we dealt here only with $\alpha_1=1$), but we
did not pursue this research direction.

\begin{ques}
The present paper contributes methods for evaluating compression
exponents of wreath products $G\bwr H$ in terms of the compression
exponents of $G$ and $H$. This continues the lines of research
studied in~\cite{Gal08,ADS06,Tess06,SV07,ANP07,NP07,CSV07}. It would
be of great interest (and probably quite challenging) to design such
methods for more general semi-direct products $G\rtimes H$.
\end{ques}

In Theorem~\ref{thm:polywreath} we computed $\alpha_p^*(C_2\bwr H)$
when $H$ has polynomial growth. It seems likely that our methods
yield non-trivial compression bounds also when $H$ has intermediate
growth. But, it would be of great interest to design methods which
deal with the case when $H$ has exponential growth. A simple example
of this type is the group $C_2\bwr (C_2\bwr\Z)$, for which we do not
even know whether the Hilbert compression exponent is positive.

\begin{ques}\label{Q:lp}
In our definition of $L_p$ compression we considered embeddings into
$L_p$ because it contains isometrically all separable $L_p(\mu)$
spaces. Nevertheless, the embeddings that we construct take values
in the sequence space $\ell_p$. Does there exist a finitely
generated group $G$ for which $\alpha_p^*(G)\neq
\alpha_{\ell_p}^*(G)$? Is the $\ell_p$ compression exponent of a net
in $L_p$ equal to $1$? Note that for $p\neq 2$ the function space
$L_p$ does not admit a bi-Lipschitz embedding into the sequence
space $\ell_p$---this follows via a differentiation argument
(see~\cite{BL00}) from the corresponding statement for linear
isomorphic embeddings (see~\cite{Pel60}).
\end{ques}

The subtlety between embeddings into $L_p$ and embeddings into
$\ell_p$ which is highlighted in Question~\ref{Q:lp} was pointed out
to us by Marc Bourdon.

\bibliographystyle{abbrv}
\bibliography{wreath-general}

\def\cprime{$'$} \def\cprime{$'$}
\begin{thebibliography}{10}

\bibitem{AhaMauMit}
I.~Aharoni, B.~Maurey, and B.~S. Mityagin.
\newblock Uniform embeddings of metric spaces and of {B}anach spaces into
  {H}ilbert spaces.
\newblock {\em Israel J. Math.}, 52(3):251--265, 1985.

\bibitem{AS00}
N.~Alon and J.~H. Spencer.
\newblock {\em The probabilistic method}.
\newblock Wiley-Interscience Series in Discrete Mathematics and Optimization.
  Wiley-Interscience [John Wiley \& Sons], New York, second edition, 2000.
\newblock With an appendix on the life and work of Paul Erd\H{o}s.

\bibitem{ADS06}
G.~Arzhantseva, C.~Drutu, and M.~Sapir.
\newblock Compression functions of uniform embeddings of groups into {H}ilbert
  and {B}anach spaces.
\newblock Preprint, 2006. To appear in J. Reine Angew. Math.

\bibitem{AGS06}
G.~N. Arzhantseva, V.~S. Guba, and M.~V. Sapir.
\newblock Metrics on diagram groups and uniform embeddings in a {H}ilbert
  space.
\newblock {\em Comment. Math. Helv.}, 81(4):911--929, 2006.

\bibitem{Ass83}
P.~Assouad.
\newblock Plongements lipschitziens dans {${\bf R}\sp{n}$}.
\newblock {\em Bull. Soc. Math. France}, 111(4):429--448, 1983.

\bibitem{ANP07}
T.~Austin, A.~Naor, and Y.~Peres.
\newblock The wreath product of {$\Bbb Z$} with {$\Bbb Z$} has {H}ilbert
  compression exponent {$\frac{2}{3}$}.
\newblock {\em Proc. Amer. Math. Soc.}, 137(1):85--90, 2009.

\bibitem{ANV07}
T.~Austin, A.~Naor, and A.~Valette.
\newblock The {E}uclidean distortion of the lamplighter group.
\newblock Preprint, 2007. To appear in Disc. Comput. Geom.

\bibitem{Bal}
K.~Ball.
\newblock Markov chains, {R}iesz transforms and {L}ipschitz maps.
\newblock {\em Geom. Funct. Anal.}, 2(2):137--172, 1992.

\bibitem{BLMN05}
Y.~Bartal, N.~Linial, M.~Mendel, and A.~Naor.
\newblock On metric {R}amsey-type phenomena.
\newblock {\em Ann. of Math. (2)}, 162(2):643--709, 2005.

\bibitem{BL00}
Y.~Benyamini and J.~Lindenstrauss.
\newblock {\em Geometric nonlinear functional analysis. {V}ol. 1}, volume~48 of
  {\em American Mathematical Society Colloquium Publications}.
\newblock American Mathematical Society, Providence, RI, 2000.

\bibitem{Bourgain86}
J.~Bourgain.
\newblock The metrical interpretation of superreflexivity in {B}anach spaces.
\newblock {\em Israel J. Math.}, 56(2):222--230, 1986.

\bibitem{BD-CK65}
J.~Bretagnolle, D.~Dacunha-Castelle, and J.-L. Krivine.
\newblock Fonctions de type positif sur les espaces {$L\sp{p}$}.
\newblock {\em C. R. Acad. Sci. Paris}, 261:2153--2156, 1965.

\bibitem{BS08}
N.~Brodskiy and D.~Sonkin.
\newblock Compression of uniform embeddings into {H}ilbert space.
\newblock {\em Topology Appl.}, 155(7):725--732, 2008.

\bibitem{CN05}
S.~Campbell and G.~A. Niblo.
\newblock Hilbert space compression and exactness of discrete groups.
\newblock {\em J. Funct. Anal.}, 222(2):292--305, 2005.

\bibitem{Che99}
J.~Cheeger.
\newblock Differentiability of {L}ipschitz functions on metric measure spaces.
\newblock {\em Geom. Funct. Anal.}, 9(3):428--517, 1999.

\bibitem{CK06}
J.~Cheeger and B.~Kleiner.
\newblock Differentiating maps into ${L}^1$ and the geometry of {B}{V}
  functions.
\newblock Preprint, 2006. To appear in Ann. Math.

\bibitem{CCJJV01}
P.-A. Cherix, M.~Cowling, P.~Jolissaint, P.~Julg, and A.~Valette.
\newblock {\em Groups with the {H}aagerup property}, volume 197 of {\em
  Progress in Mathematics}.
\newblock Birkh\"auser Verlag, Basel, 2001.
\newblock Gromov's a-T-menability.

\bibitem{CM98}
T.~H. Colding and W.~P. Minicozzi, II.
\newblock Liouville theorems for harmonic sections and applications.
\newblock {\em Comm. Pure Appl. Math.}, 51(2):113--138, 1998.

\bibitem{D-CK70}
D.~Dacunha-Castelle and J.-L. Krivine.
\newblock Ultraproduits d'espaces d'{O}rlicz et applications g\'eom\'etriques.
\newblock {\em C. R. Acad. Sci. Paris S\'er. A-B}, 271:A987--A989, 1970.

\bibitem{D-CK72}
D.~Dacunha-Castelle and J.~L. Krivine.
\newblock Applications des ultraproduits \`a l'\'etude des espaces et des
  alg\`ebres de {B}anach.
\newblock {\em Studia Math.}, 41:315--334, 1972.

\bibitem{CSV07}
Y.~de~Cornulier, Y.~Stalder, and A.~Valette.
\newblock Proper actions of lamplighter groups associated with free groups.
\newblock {\em C. R. Acad. Sci. Paris}, 346(3-4):173--176, 2008.

\bibitem{dCTV07}
Y.~de~Cornulier, R.~Tessera, and A.~Valette.
\newblock Isometric group actions on {H}ilbert spaces: growth of cocycles.
\newblock {\em Geom. Funct. Anal.}, 17(3):770--792, 2007.

\bibitem{dlHV89}
P.~de~la Harpe and A.~Valette.
\newblock La propri\'et\'e {$(T)$} de {K}azhdan pour les groupes localement
  compacts (avec un appendice de {M}arc {B}urger).
\newblock {\em Ast\'erisque}, (175):158, 1989.
\newblock With an appendix by M. Burger.

\bibitem{DS88}
N.~Dunford and J.~T. Schwartz.
\newblock {\em Linear operators. {P}art {I}}.
\newblock Wiley Classics Library. John Wiley \& Sons Inc., New York, 1988.
\newblock General theory, With the assistance of William G. Bade and Robert G.
  Bartle, Reprint of the 1958 original, A Wiley-Interscience Publication.

\bibitem{Durrett96}
R.~Durrett.
\newblock {\em Probability: theory and examples}.
\newblock Duxbury Press, Belmont, CA, second edition, 1996.

\bibitem{Ersh01}
A.~G. {\`E}rschler.
\newblock On the asymptotics of the rate of departure to infinity.
\newblock {\em Zap. Nauchn. Sem. S.-Peterburg. Otdel. Mat. Inst. Steklov.
  (POMI)}, 283(Teor. Predst. Din. Sist. Komb. i Algoritm. Metody. 6):251--257,
  263, 2001.

\bibitem{Fel66}
W.~Feller.
\newblock {\em An introduction to probability theory and its applications.
  {V}ol. {II}}.
\newblock John Wiley \& Sons Inc., New York, 1966.

\bibitem{FJ03}
R.~J. Fleming and J.~E. Jamison.
\newblock {\em Isometries on {B}anach spaces: function spaces}, volume 129 of
  {\em Chapman \& Hall/CRC Monographs and Surveys in Pure and Applied
  Mathematics}.
\newblock Chapman \& Hall/CRC, Boca Raton, FL, 2003.

\bibitem{Gal08}
{\'S}.~R. Gal.
\newblock Asymptotic dimension and uniform embeddings.
\newblock {\em Groups Geom. Dyn.}, 2(1):63--84, 2008.

\bibitem{Gro03}
M.~Gromov.
\newblock Random walk in random groups.
\newblock {\em Geom. Funct. Anal.}, 13(1):73--146, 2003.

\bibitem{GK04}
E.~Guentner and J.~Kaminker.
\newblock Exactness and uniform embeddability of discrete groups.
\newblock {\em J. London Math. Soc. (2)}, 70(3):703--718, 2004.

\bibitem{Halmos}
P.~R. Halmos.
\newblock {\em Measure {T}heory}.
\newblock D. Van Nostrand Company, Inc., New York, N. Y., 1950.

\bibitem{Hei01}
J.~Heinonen.
\newblock {\em Lectures on analysis on metric spaces}.
\newblock Universitext. Springer-Verlag, New York, 2001.

\bibitem{Hein80}
S.~Heinrich.
\newblock Ultraproducts in {B}anach space theory.
\newblock {\em J. Reine Angew. Math.}, 313:72--104, 1980.

\bibitem{IL71}
I.~A. Ibragimov and Y.~V. Linnik.
\newblock {\em Independent and stationary sequences of random variables}.
\newblock Wolters-Noordhoff Publishing, Groningen, 1971.
\newblock With a supplementary chapter by I. A. Ibragimov and V. V. Petrov,
  Translation from the Russian edited by J. F. C. Kingman.

\bibitem{Jones90}
P.~W. Jones.
\newblock Rectifiable sets and the traveling salesman problem.
\newblock {\em Invent. Math.}, 102(1):1--15, 1990.

\bibitem{Kak41}
S.~Kakutani.
\newblock Concrete representation of abstract {$(L)$}-spaces and the mean
  ergodic theorem.
\newblock {\em Ann. of Math. (2)}, 42:523--537, 1941.

\bibitem{Kaz67}
D.~A. Ka{\v{z}}dan.
\newblock On the connection of the dual space of a group with the structure of
  its closed subgroups.
\newblock {\em Funkcional. Anal. i Prilo\v zen.}, 1:71--74, 1967.

\bibitem{Lam58}
J.~Lamperti.
\newblock On the isometries of certain function-spaces.
\newblock {\em Pacific J. Math.}, 8:459--466, 1958.

\bibitem{LT79}
J.~Lindenstrauss and L.~Tzafriri.
\newblock {\em Classical {B}anach spaces. {II}}, volume~97 of {\em Ergebnisse
  der Mathematik und ihrer Grenzgebiete [Results in Mathematics and Related
  Areas]}.
\newblock Springer-Verlag, Berlin, 1979.
\newblock Function spaces.

\bibitem{LMN02}
N.~Linial, A.~Magen, and A.~Naor.
\newblock Girth and {E}uclidean distortion.
\newblock {\em Geom. Funct. Anal.}, 12(2):380--394, 2002.

\bibitem{MN04}
M.~Mendel and A.~Naor.
\newblock Euclidean quotients of finite metric spaces.
\newblock {\em Adv. Math.}, 189(2):451--494, 2004.

\bibitem{MS86}
V.~D. Milman and G.~Schechtman.
\newblock {\em Asymptotic theory of finite-dimensional normed spaces}, volume
  1200 of {\em Lecture Notes in Mathematics}.
\newblock Springer-Verlag, Berlin, 1986.
\newblock With an appendix by M. Gromov.

\bibitem{NP07}
A.~Naor and Y.~Peres.
\newblock Embeddings of discrete groups and the speed of random walks.
\newblock {\em Int. Math. Res. Not.}, 2008.
\newblock Article ID rnn076, 34 pages, doi:10.1093/imrn/rnn076.

\bibitem{NPSS06}
A.~Naor, Y.~Peres, O.~Schramm, and S.~Sheffield.
\newblock Markov chains in smooth {B}anach spaces and {G}romov-hyperbolic
  metric spaces.
\newblock {\em Duke Math. J.}, 134(1):165--197, 2006.

\bibitem{NT07}
A.~Naor and T.~Tao.
\newblock Random martingales and the {H}ardy-{L}ittlewood maximal inequality.
\newblock Preprint, 2007.

\bibitem{Oki92}
K.~Okikiolu.
\newblock Characterization of subsets of rectifiable curves in {${\bf R}\sp
  n$}.
\newblock {\em J. London Math. Soc. (2)}, 46(2):336--348, 1992.

\bibitem{PZ32}
R.~Paley and A.~Zygmund.
\newblock A note on analytic functions in the unit circle.
\newblock {\em Proc. Camb. Phil. Soc.}, 28:266--272, 1932.

\bibitem{Pel60}
A.~Pelczy{\'n}ski.
\newblock Projections in certain {B}anach spaces.
\newblock {\em Studia Math.}, 19:209--228, 1960.

\bibitem{Schul07}
R.~Schul.
\newblock Analyst's traveling salesman theorems. {A} survey.
\newblock In {\em In the tradition of {A}hlfors-{B}ers. {IV}}, volume 432 of
  {\em Contemp. Math.}, pages 209--220. Amer. Math. Soc., Providence, RI, 2007.

\bibitem{SV07}
Y.~Stalder and A.~Valette.
\newblock Wreath products with the integers, proper actions and {H}ilbert space
  compression.
\newblock {\em Geom. Dedicata}, 124:199--211, 2007.

\bibitem{Steele97}
J.~M. Steele.
\newblock {\em Probability theory and combinatorial optimization}, volume~69 of
  {\em CBMS-NSF Regional Conference Series in Applied Mathematics}.
\newblock Society for Industrial and Applied Mathematics (SIAM), Philadelphia,
  PA, 1997.

\bibitem{Stoll98}
M.~Stoll.
\newblock On the asymptotics of the growth of {$2$}-step nilpotent groups.
\newblock {\em J. London Math. Soc. (2)}, 58(1):38--48, 1998.

\bibitem{Tess06}
R.~Tessera.
\newblock Asymptotic isoperimetry on groups and uniform embeddings into
  {B}anach spaces.
\newblock Preprint, 2006. Available at \url{http://arxiv.org/abs/math/0603138}.

\bibitem{Wag93}
S.~Wagon.
\newblock {\em The {B}anach-{T}arski paradox}.
\newblock Cambridge University Press, Cambridge, 1993.
\newblock With a foreword by Jan Mycielski, Corrected reprint of the 1985
  original.

\bibitem{WW75}
J.~H. Wells and L.~R. Williams.
\newblock {\em Embeddings and extensions in analysis}.
\newblock Springer-Verlag, New York, 1975.
\newblock Ergebnisse der Mathematik und ihrer Grenzgebiete, Band 84.

\bibitem{Woj91}
P.~Wojtaszczyk.
\newblock {\em Banach spaces for analysts}, volume~25 of {\em Cambridge Studies
  in Advanced Mathematics}.
\newblock Cambridge University Press, Cambridge, 1991.

\end{thebibliography}

\end{document}